\documentclass[12pt]{article}
\usepackage{amsmath, amsthm, amsfonts,amssymb}
\usepackage{bm}
\usepackage{xcolor}
\usepackage[round, sort]{natbib}
\usepackage[flushleft]{threeparttable}
\usepackage[left=1in,right=1in,top=1in,bottom=1in]{geometry}
\usepackage{fancyhdr}
\usepackage{subcaption}
\usepackage[affil-it]{authblk}
\usepackage{graphicx, color}
\usepackage{epstopdf}
\usepackage{float}
\usepackage{array}
\usepackage{booktabs}
\usepackage{multirow}

\usepackage[font=small,skip=0pt]{caption}
\DeclareMathOperator*{\argmin}{argmin}
\newcommand{\norm}[1]{\left\lVert#1\right\rVert}

\usepackage{geometry}

\hypersetup{ 
    citecolor=blue,
    urlcolor  = black,
}

\theoremstyle{plain}
\newtheorem{thm}{Theorem}[section]
\newtheorem{cor}[thm]{Corollary}
\newtheorem{lem}[thm]{Lemma}
\newtheorem{assu}[thm]{Assumption}
\newtheorem{prop}[thm]{Proposition}
\newtheorem{defn}[thm]{Definition}
\newtheorem{exam}[thm]{Example}
\theoremstyle{remark}
\newtheorem{rem}[thm]{Remark}

\numberwithin{equation}{section}

\date{}
\begin{document}
\title{\Large Globally Optimal And Adaptive Short-Term Forecast of Locally Stationary Time Series And A Test for Its Stability }

\author[1]{ Xiucai Ding \thanks{E-mail: xiucai.ding@duke.edu. }}  
\author[2]{ Zhou Zhou \thanks{E-mail: zhou@utstat.utoronto.ca.}}
\affil[1]{Department of Mathematics, Duke University}
\affil[2]{Department of Statistical Sciences, University of Toronto}

\maketitle


\begin{abstract} 
Forecasting the evolution of complex systems is one of the grand challenges of modern data science. The fundamental difficulty lies in understanding the structure of the observed stochastic process. In this paper, we show that every uniformly-positive-definite-in-covariance and sufficiently short-range dependent non-stationary and nonlinear time series can be well approximated globally by an auto-regressive process of slowly diverging order. When linear prediction with ${\cal L}^2$ loss is concerned, the latter result facilitates a unified globally-optimal short-term forecasting theory for a wide class of locally stationary time series asymptotically.  A nonparametric sieve method is proposed to globally and adaptively estimate the optimal forecasting coefficient functions and the associated mean squared error of forecast. An adaptive stability test is proposed to check whether the optimal forecasting coefficients are time-varying, a frequently-encountered question for practitioners and researchers of time series. Furthermore, partial auto-correlation functions (PACF) of general non-stationary time series are studied and used as a visual tool to explore the linear dependence structure of such series.  We use extensive numerical simulations and two real data examples to illustrate the usefulness of our results.

\end{abstract}

\section{Introduction}
It is of critical importance to understand the structure of time series in order to accurately forecast the future. For a stationary process $\{z_{i}\}_{i=1}^n$, Baxter established an important result on its structure in \cite{Baxter_1962, baxter1963}. Together with the deep representation theorems of stationary processes formed in, for instance, \cite{Wiener1958} and \cite{pourahmadi2001foundations}, Baxter's inequality implies that $\{z_i\}_{i=1}^n$ can be well approximated by an auto-regressive (AR) process of slowly diverging order provided that $\{z_i\}$ is of short memory and uniformly positive spectral density. Consequently,  as long as linear prediction with minimum mean squared error (MSE) is concerned, Baxter's inequality serves as a theoretical foundation that guarantees the asymptotic optimality of forecasting a wide class of stationary processes by AR models with slowly diverging order.  Nowadays, as increasingly longer time series are being collected in the modern information age, it has become more appropriate to model many of those series as non-stationary processes whose data generating mechanisms evolve over time. As a result there has been an increasing demand for a systematic optimal forecasting theory for non-stationary processes.  Nevertheless, it has been a difficult and open problem to establish general structural representations such as those of Baxter in the non-stationary domain.  The main difficulty lies in the loss of Toeplitz structure for covariance matrices of general non-stationary time series. As a consequence deep connections between  Toeplitz  matrices and their spectral density functions (cf. e.g., \cite{MR1511625}, \cite{kac1954}, \cite{grenander2001toeplitz}) which are crucial in the proof of  structural representations of stationary processes cannot be used directly in the non-stationary case.

The purpose of the paper is fourfold. Firstly, we establish a unified structural representation result  that every short memory and uniformly-positive-definite-in-covariance (UPDC) non-stationary time series $\{x_{i,n}\}_{i=1}^n$ can be well approximated globally by a non-stationary white-noise-driven AR process of slowly diverging order. Here the speed of the divergence is determined by the strength of the temporal dependence. In the best scenario where the  temporal dependence is of exponential  decay, the order is $O(\log n)$.  Instead of resorting to Toeplitz matrix and spectral density techniques, our proof of the result heavily depends on random matrix theory which controls the proximity of non-stationary covariance matrices and their banded truncations as well as modern spectral theory \cite{DMS} that controls the decay rates of inverse of banded matrices. In the special case of locally stationary time series, that is, non-stationary time series whose data generating mechanisms evolve smoothly over time,  the UPDC condition is shown to be equivalent to uniform time-frequency positiveness of the spectral density of $\{x_{i,n}\}_{i=1}^n$ and the approximating AR process is shown to have smoothly time-varying coefficients. In particular, when linear prediction with ${\cal L}^2$ loss is concerned, our structural representation result implies that a wide class of locally stationary time series can be asymptotically optimally predicted by an AR model with slowly diverging order and smoothly changing coefficients in the short term.

Secondly, we propose a nonparametric sieve-based regression method to adaptively estimate the time-varying optimal linear forecast coefficients and the associated MSE of forecast.   Specifically, we approximate every smooth coefficient function by a finite but diverging term orthonormal basis expansion and perform one high-dimensional simple linear regression to estimate all the coefficient functions which is computationally easy and stable to implement. Contrary to most non-stationary time series forecasting methods in the literature where only data near the end of the sequence are used to estimate the parameters of the forecast, our nonparametric sieve regression is global in the sense that it utilizes all available time series observations to determine the optimal forecast coefficients and hence is expected to be more efficient.  Indeed, by controlling the number of basis functions used in the regression, we demonstrate that the sieve method is adaptive in the sense that the estimation accuracy achieves global minimax rate for nonparametric function estimation in the sense of \cite{stone1982}.  Additionally, since the sieve regression uses all time series observations, the estimated coefficient functions  do not have inferior performances at the boundary of the estimating interval when certain sieves such as the Fourier and wavelet expansions are used. The latter property  is  an important advantage of the nonparametric sieve method as the short term forecast is determined by the estimated regression coefficient at the right boundary. On the contrary, local nonparametric methods such as the kernel regression face relatively sparse data near the boundary and hence produce more volatile estimates in those regions.  

Our third purpose is to develop an adaptive stability test for the optimal forecast coefficients. Many practitioners tend to use classic ARMA models with constant coefficients for time series prediction. Hence it is of great importance to check whether the optimal forecast coefficient functions are time-varying in order to either justify or invalidate such practice. To our knowledge, there exist no results on testing stability of the optimal forecast coefficients for general classes of non-stationary time series in the literature. In this paper, we develop an ${\cal L}^2$ nonparametric test for the constancy of the optimal forecast coefficients based on their sieve estimators. The test is shown to be adaptive to the strength of the time series dependence as well as the smoothness of the underlying data generating mechanism. The theoretical investigation of the test critically depends on a result on Gaussian approximations to quadratic forms of high-dimensional non-stationary time series developed in the current paper. In particular, uniform Gaussian approximations over high-dimensional convex sets (\cite{CF} and \cite{FX}) as well as $m$-dependent approximations to quadratic forms of non-stationary time series are important techniques used in the proofs.  On the other hand, we demonstrate that stability of the forecast coefficients is asymptotically equivalent to correlation stationarity of locally stationary time series. Here correlation stationarity means that the correlation structure of the time series does not change over time. As a result, our stability test can also be viewed as an adaptive test for correlation stationarity. In the statistics literature, there is a recent surge of interest in testing {\it covariance} stationarity of a time series using techniques from the spectral domain. See, for instance,  \cite{EP2010},  \cite{DR}, \cite{DPV} and \cite{GN}. Observe that the time-varying marginal variance has to be estimated and removed from the time series in order to apply those tests to checking correlation stationarity. However, it is unknown whether the errors introduced in such estimation would influence the finite sample and asymptotic behaviour of the tests. Furthermore, estimating the marginal variance involves the difficult choice of a smoothing parameter. One major advantage of our test when used as a test of correlation stationarity is that it is totally free from the marginal variance as the latter quantity is absorbed into the errors of the high-dimensional linear regression and hence is independent of the optimal forecast coefficients. Additionally, our test is expected to be more powerful than the aforementioned tests of covariance stationarity as the latter tests are generally not adaptive to the strength of time series dependence or the smoothness of the data generating mechanism. We refer the readers to Section \ref{sec:simu} for a related simulation study.  Finally, we mention that \cite{MR3931381} studied change point tests for correlations of non-stationary time series. However, their test can only be applied to a fixed number of lags and cannot be used as a test for overall correlation stationarity of time series.




Finally, we study the partial auto-correlation function (PACF) for general non-stationary time series and use it as a visual tool to study non-stationary time series dependence structure and preliminarily determine an appropriate order of the AR approximation for the optimal forecast. The PACF is a commonly used tool to study the pattern of temporal dependence and determine the order of an AR model in stationary time series analysis (cf. e.g. \cite{tsbookfore}). However, to our knowledge there exists no work in the literature conducting statistical inference of PACF under non-stationarity. For a general non-stationary time series, we investigate the PACF as a two-dimensional function of time and lag and develop its uniform decay rate which is determined by the magnitude of the time series dependence measure. In the special case of locally stationary time series, a sieve method is proposed to estimate the smoothly time-varying PACF which is shown to be adaptive and uniformly consistent. For a groups of lags, an ${\cal L}^2$ test is developed to check whether the PACF at those lags are uniformly zero across time. 
Consequently, one can visually investigate the pattern of time series dependence not only across lag but also over time from the estimated PACF plot. Together with the $p$-values of the  ${\cal L}^2$ tests, one is able to identity when the time series dependence disappears from the PACF plot and hence preliminarily determine an appropriate order of the AR model for the optimal forecast.


%
%


In the statistics literature, there have been some scattered works discussing non-stationary time series prediction.   Among others,  \cite{FBS} considered forecasting locally stationary time series by their wavelet process representations and established a wavelet-based prediction equation which is derived from the corresponding Yule-Walker equation; \cite{RSP} used time-varying AR models of a fixed order to forecast a locally stationary time series; \cite{KPF} investigated finite-sample forecasting performances of locally stationary time series using Yule-Walker estimators of both fixed and time-varying parameters. In all the above mentioned works, the optimality of the truncated or clipped AR approximation was not discussed and the non-stationary auto-covariance functions was estimated by simple kernel methods which were not adaptive to the smoothness of the latter functions. On the other hand, \cite{DP} considered optimal model-based and model-free predictions of two special classes of locally stationary time series; that is, locally stationary time series which are correlation stationary and locally stationary processes that can be marginally transformed into stationary Gaussian processes.

At last, we would like to mention that Baxter's inequality has been extended in many different ways and the application of it goes way beyond optimal forecasting. See, for instance, \cite{Cheng1993} for an extension to multivariate processes, \cite{Meyer2015} for an extension to triangular arrays, and \cite{inoue2018} for an extension to long memory processes. On the application side, among others, Baxter's inequality is a key component for the theoretical investigation of the sieve bootstrap (\cite{kreiss1988} and \cite{kreiss2011}).

This paper is organized as follows. In Section \ref{sec:preliminary}, we introduce AR approximation results for general non-stationary time series. The time-varying PACF is also properly defined in this section. In Section \ref{sec:arappoximate}, we study  AR approximation of locally stationary time series.  In Section \ref{sec:predict}, asymptotically globally optimal forecast of locally stationary time series using the AR approximation is studied and the nonparametric sieve method is proposed to estimate the best forecast coefficient functions and the associated MSE of forecast. In Section \ref{sec:test}, we test the stability of the best linear forecast using  $\mathcal{L}^2$ statistics of the estimated forecast coefficient functions. A robust bootstrapping procedure is proposed for practical implementation. In Section \ref{sec:simu}, we use extensive Monte Carlo simulations to verify the accuracy of our prediction and test.  In  Section \ref{sec:realdata}, we conduct analysis on two real datasets using our proposed methods.   Technical proofs  are deferred to an online supplementary material. 

\vspace{5pt}

%

\section{AR approximations and PACF for general non-stationary time series } \label{sec:preliminary}

In this section, we establish a general AR approximation theory for non-stationary time series. A study of the PACF of such series will also be conducted. We start with introducing some notation. 
For a matrix $Y$ or vector $\bm{y},$ we use $Y^*$ and $\bm{y}^*$ to stand for their transposes. For a sequence of  random variables $\{x_n\}$ and real values $\{a_n\},$ we use the notation $x_{n}=O_{\mathbb{P}}(a_n)$ to state that $x_n/a_n$ is stochastically bounded. 
 Similarly, we use the notation $x_n=o_{\mathbb{P}}(a_n)$ to say that $x_n/a_n$ converges to 0  in probability. 
In this paper, unless otherwise specified, for a sequence of random variables $\{x_{i,n} \},$ we use the notation $x_{i,n}=O_{\mathbb{P}}(a_n)$ to state that $x_{i, n}/a_n$ is stochastically bounded uniformly in the index $i.$ 
For general non-stationary time series $\{x_{i,n}\}$, we assume that it  has the following form 
\begin{equation}\label{eq_xi}
x_{i,n}=G_{i,n}(\mathcal{F}_i), \ i=1,2,\cdots,n,
\end{equation}
where $\mathcal{F}_i:=(\cdots, \eta_{i-1}, \eta_i)$ and $\eta_i, i \in \mathbb{Z}$ are i.i.d. random variables and the sequence of functions $ G_{i,n}: \mathbb{R}^{\infty} \times \mathbb{R}^{\infty} \rightarrow \mathbb{R} $ are measurable functions such that for all $ 1 \leq
 i_0 \leq n,$ $G_{i_0,n}(\mathcal{F}_i) $ is a properly defined random variable. The above representation is very general since any non-stationary time series can be represented in the form of \eqref{eq_xi} via the Rosenblatt transform \citep{rosenblatt1952}. Till the end of the paper, we omit the subscript $n$ and simply write $x_i \equiv x_{i,n}$ without causing any confusion.
 
Next we introduce the physical dependence measure defined in \cite{WW, ZZ1, WZ2} to quantify the temporal dependence of $\{x_i\}$ defined in (\ref{eq_xi}).
\begin{defn} Let $\{\eta_i^{\prime}\}$ be an i.i.d. copy of $\{\eta_i\}.$ Assuming that for some $q>2,$ 
\begin{equation} \label{eq_boundx}
||x_i||_q<\infty.
\end{equation}
 Then for $j \geq 0,$ we define the physical dependence measure of $\{x_i\}$ by {
\begin{equation}\label{eq_phygeneral}
\delta^g(j,q):=\max_{k} \max_i || G_{i,k}(\mathcal{F}_0)-G_{i,k}(\mathcal{F}_{0,j}) ||_q,
\end{equation}
where $\mathcal{F}_{0,j}:=(\mathcal{F}_{-j-1}, \eta_{-j}^{\prime},\eta_{-j+1},\cdots, \eta_0).$ }
\end{defn}
In this paper, we focus on time series with short-range temporal dependence. Specifically, 
we impose the following assumption on the physical dependence measure $\delta^g(\cdot, \cdot)$.
{
\begin{assu}\label{phy_generalts}  There exists  a constant $\tau>5+\varpi$, where $\varpi>0$ is some fixed small constant, such that for some constant $C>0,$ we have 
\begin{equation}
\delta^g(j,q) \leq Cj^{-\tau}, \ j \geq 1. 
\end{equation} 
\end{assu} 
} 
{The above assumption guarantees that the temporal dependence of $\{x_i\}$ decays polynomially fast. Additionally, in order to avoid erratic behaviour of the best linear forecast operators, the smallest eigenvalue of the time series covaraince matrix should be bounded away from zero. For stationary time series, this is equivalent to the uniform positiveness of the spectral density function widely used in \cite{Baxter_1962, baxter1963} et al. Further note that the latter assumption is mild and frequently used in the statistics literature of covariance and precision matrix estimation; see, for instance, \cite{cai2016, chen2013, Yuan2010} and the references therein. In this paper we shall call this uniformly-positive-definite-in-covariance (UPDC) condition and we formally summarize it as follows. 
\begin{assu}[UPDC]\label{assu_pdc} For all $n \in \mathbb{N}, $ there exists a universal constant $\kappa>0$ such that 
\begin{equation}\label{eq_defnkappa}
\lambda_n(\operatorname{Cov}(x_1, \cdots, x_n)) \geq \kappa,
\end{equation}
where $\lambda_n(\cdot)$ is the smallest eigenvalue of the given matrix and $\operatorname{Cov}(\cdot)$ is the covariance matrix of the given vector. 
\end{assu}
}   
{
We then provide a simple sufficient condition for UPDC. 
Denote the covariance matrix of $\{x_i\}_{i=1}^n$ as $\Sigma_{x,n}=(\sigma_{ij,n})_{i,j=1}^n.$ For $k \in \mathbb{N}$ and $k<n,$ we denote the banded truncation of $\Sigma_{x,n}$ by $\Sigma_{x,k}=(\sigma_{ij,k})_{i,j=1}^n$ such that 
\begin{equation*}
\sigma_{ij,k}=
\begin{cases}
\sigma_{ij,n} & |i-j| \leq k; \\
0 & \text{Otherwise}.
\end{cases}
\end{equation*}

{

\begin{lem}\label{eq_sufficinetconditionupdc}  Suppose that for all $n \in \mathbb{N},$ there exists an $0 \leq n_0 \leq n$ such that for some universal constant $\varsigma>0,$ we have  
\begin{equation}\label{eq_condition1}
\lambda_n(\Sigma_{x,n_0}) \geq \varsigma.
\end{equation}
Moreover, assume that for some positive constant $\delta \equiv \delta(n)$ such that {$\delta<\varsigma/2$ } 
\begin{equation}\label{eq_conidition2}
\max_i \left|\sum_{j=n_0+1}^n\sigma_{ij,n}\right|< \delta.
\end{equation}
Then the UPDC condition holds. 
\end{lem}
}
}


{
\subsection{AR approximation for general non-stationary time series} \label{sec:nonzeromeandiscussion}
Throughout the paper, unless otherwise specified, we always {assume} that 
\begin{equation} \label{eq_defnbeplsilon}
b=O(n^{(1+\epsilon)/\tau}),
\end{equation}
{ where $0<\epsilon \leq \varpi/10$ is an arbitrarily small and fixed constant. Here $\varpi$ is defined in Assumption \ref{phy_generalts}. }
In this section, we show that under Assumptions \ref{phy_generalts} and \ref{assu_pdc}, $x_i$ can be well approximated by an $\text{AR}(b)$ process. 


For $i>b,$ denote $\widehat{x}_i$ as the best linear prediction based on its predecessors $x_1, \cdots, x_{i-1},$ i.e.,
\begin{equation*}
\widehat{x}_i=\phi_{i0}+\sum_{j=1}^{i-1} \phi_{ij} x_{i-j}, \ i=b+1, \cdots, n.
\end{equation*}

Denote $\epsilon_i:=x_i-\widehat{x}_i.$ Then
\begin{equation}\label{eq_arapproxiamtionnoncenter}
x_i=\phi_{i0}+\sum_{j=1}^{i-1} \phi_{ij} x_{i-j}+\epsilon_i, \ i=b+1,\cdots, n.
\end{equation} 

The theorem below provides a control for $\phi_{ij}$. It extends Baxter's inequality to the non-stationary domain and is an important consequence of Assumptions \ref{phy_generalts} and \ref{assu_pdc} . It states that the magnitude of $\phi_{ij}$ is negligible when $j \geq b$ uniformly for $i>j$ and the best linear forecast coefficients of $x_i$ based on $i-1$ and $b$ predecessors are close uniformly in $i$. 
{
\begin{thm}\label{lem_phibound} 
For $x_i$ defined in (\ref{eq_xi}), suppose Assumptions \ref{phy_generalts} and \ref{assu_pdc} hold true. For any fixed small constant $\epsilon>0$ defined in (\ref{eq_defnbeplsilon}) and $\kappa$ defined in (\ref{eq_defnkappa}),  there exists some constant $C>0,$ such that for $n$ satisfying $n \geq (\frac{2}{\kappa})^{C(1+\epsilon)/(\tau-1)}$ and $n-n^{(1+\epsilon/2)/(1+\epsilon)} \geq C$,we have for $j<i,$  
\begin{equation}\label{eq_phibound1}
| \phi_{ij}| \leq C j^{-(\tau-1)/(1+\epsilon)}, \ j \geq 1. 
\end{equation}
Moreover, denote by $\{\phi_{ij}^b\}$ the best linear forecast coefficients of $x_i$ based on $x_{i-1}, \cdots, x_{i-b}$ satisfying  
\begin{equation*}
x_i=\phi_{i0}^b+\sum_{j=1}^b \phi_{ij}^b x_{i-j}+\epsilon_i^b, \ i>b. 
\end{equation*} 
Then we have that 
\begin{equation}\label{eq_phibound2}
\max_{1 \leq j \leq b}|\phi_{ij}-\phi_{ij}^b| \leq C n^{-1+(3+2\epsilon)/\tau}, \ |\phi_{i0}-\phi_{i0}^b| \leq C n^{-1+(3.5+2.5\epsilon)/\tau}.
\end{equation}
%
\end{thm}
}

An important consequence of Theorem  \ref{lem_phibound} is that,  under the short-range dependence and UPDC assumptions, any non-stationary time series can be efficiently approximated by an AR process of slowly diverging order.  And the order of such approximation is adaptive to the temporal decay rate of the time series dependence. Formally, we summarize the above statements in the following proposition and theorem. 


\begin{prop}\label{prop_dependent}
Suppose Assumptions \ref{phy_generalts} and \ref{assu_pdc} hold true. Then we have 
\begin{equation}\label{eq_nonzeromean11}
x_i=\phi_{i0}+\sum_{j=1}^{\min\{b,i-1\}} \phi_{ij} x_{i-j}+\epsilon_i+O_{\mathbb{P}}(n^{-1+(2+\epsilon)/\tau}).
\end{equation}
\end{prop}
Observe that $\{\epsilon_i\}$ is a time-varying white noise process, i.e.,
\begin{equation*}
\mathbb{E}\epsilon_i=0 \ , \operatorname{Cov} (\epsilon_i, \epsilon_j)= \mathbf{1}(i=j) \sigma_i^2. 
\end{equation*}
Furthermore, denote the process $\{x_i^*\}$ by
\begin{equation}\label{defn_xistart}
x_i^*=
\begin{cases}
x_i,  & i \leq b; \\
\phi_{i0}+\sum_{j=1}^b \phi_{ij} x_{i-j}^*+\epsilon_i, & i>b.  
\end{cases}
\end{equation}
By definition, $\{x_i^*\}_{i\ge 1}$ is an AR($b$) process.


\begin{thm} \label{thm_arrepresent} Suppose Assumptions \ref{phy_generalts} and \ref{assu_pdc} hold true. Then we have that 
\begin{equation}\label{eq_induction}
\max_{1 \leq  i \leq n} |x_i-x_i^*|=O_{\mathbb{P}}(n^{-1+(2+\epsilon)/\tau}). 
\end{equation} 
\end{thm}
}

\begin{rem}\label{rem_exponentialdecay}
In this paper, our discussions are carried out under Assumption 2.2. Our results can be easily extended to the case when the temporal dependence is of  exponential decay; i.e.
\begin{equation}\label{phy_exponential}
\delta^g(j,q) \leq C a^{j}, \ 0<a<1.
\end{equation}
In this case, we can choose $b=O(\log n)$ and Theorem \ref{lem_phibound} can be updated to 
\begin{equation*}
|\phi_{ij}| \leq C \max\{n^{-2}, a^{j/2}\}, \ j>1, \ C>0 \ \text{is some constant}, 
\end{equation*}
and the bounds in equations (\ref{eq_phibound2}), (\ref{eq_nonzeromean11}) and (\ref{eq_induction}) can be changed to $\log n/n$.
\end{rem}

{}

\subsection{{PACF for general non-stationary time series}}
{The PACF  is a commonly used tool for dependence monitoring and  model identification in time series analysis. In particular, it is well known that the PACF is useful in identifying the order of an AR model for stationary processes. The techniques developed for AR  approximation in the last subsection can be easily employed to study the behaviour of PACF for general non-stationary time series. In the following, we shall explore this aspect in detail.


Consider the non-stationary time series (\ref{eq_xi}), denote the $j$-th order best linear forecast of  $x_i$ as 
\begin{equation*}
\widehat{x}_{i,j}=\phi_{i0,j}+\sum_{k=1}^j \phi_{ik,j} x_{i-k}, \ j \leq i-1.
\end{equation*}
 Let $\epsilon_{i,j}=x_i-\widehat{x}_{i,j}$ and write
\begin{equation}\label{eq_jorderlocal}
x_i=\phi_{i0,j}+\sum_{k=1}^j \phi_{ik,j} x_{i-k}+\epsilon_{i,j}, \ 1\leq j \leq i-1.  
\end{equation}
We next introduce the definition of $j$-th order PACF for  non-stationary time series, which is a natural extension for the corresponding definition of stationary process.  
\begin{defn}\label{defn_pacf} For the  non-stationary time series (\ref{eq_xi}), the $j$-th order PACF at time $i$ is defined as
\begin{equation*}
\rho_{i,j}=\phi_{ij,j}, \ 1 \leq j \leq i-1. 
\end{equation*}
\end{defn}
Under Assumptions \ref{phy_generalts} and \ref{assu_pdc}, similar to Theorem \ref{lem_phibound},  we are able to establish the uniform speed of decay for the PACF. This is formally summarized as the following lemma. 
\begin{lem}\label{lem_pacftruncation} Under Assumptions \ref{phy_generalts} and \ref{assu_pdc}, for some constant $C>0,$ we have that 
\begin{equation*}
|\rho_{i,j}| \leq C j^{-(\tau-1)/(1+\epsilon)}, \  1 \leq j \leq i-1. 
\end{equation*}
\end{lem} 
We remark that when the physical dependence measure is of exponential decay, we can derive similar results as in Remark \ref{rem_exponentialdecay}.



\section{AR approximation for locally stationary time series}\label{sec:arappoximate}

We now focus our study on an important subclass of (\ref{eq_xi}), the locally stationary time series. This class of non-stationary time series is characterized by the fact that the underlying data generating mechanism evolves {\it smoothly} over time.

\subsection{Locally stationary time series}

Following  \cite{WZ1,WZ2}, we say that $x_i$ is a locally stationary time series if 
\begin{equation}\label{defn_model}
x_i=G(\frac{i}{n}, \mathcal{F}_i),
\end{equation}
where 
$G:[0,1] \times \mathbb{R}^{\infty} \rightarrow \mathbb{R}$ is a measurable function such that $\xi_i(t):=G(t, \mathcal{F}_i)$ is a properly defined random variable for all $t \in [0,1].$ { In (\ref{defn_model}), by allowing the data generating mechanism $G$ depending on the time
index $t$ in such a way that $G(t,\mathcal{F}_i)$ changes smoothly with respect to $t$, one
has local stationarity in the sense that the subsequence $\{x_i
, . . . , x_{i+j-1} \}$ is
approximately stationary if its length $j$ is sufficiently small compared to $n$.} For locally stationary time series $x_i$, define the physical dependence measures
{
\begin{equation}\label{eq_phyoriginal}
\delta(j,q):=\sup_{t \in [0,1]} \left|\left|G(t, \mathcal{F}_0)-G(t, \mathcal{F}_{0,j})\right|\right|_q.
\end{equation}
}
{

The following assumption guarantees that the data generating mechanism changes smoothly over time and thus the time series can be locally approximated by a stationary one. } 
\begin{assu}\label{assum_local}
$G(\cdot, \cdot)$ defined in (\ref{defn_model}) satisfies the property of stochastic Lipschitz continuity, i.e.,  for some $q>2$ and $C>0,$ 
\begin{equation}\label{assum_lip}
\left| \left| G(t_1, \mathcal{F}_{i})-G(t_2,\mathcal{F}_i) \right|\right|_q \leq C|t_1-t_2|, 
\end{equation}
where $t_1, t_2 \in [0,1].$ Furthermore, 
\begin{equation}\label{assum_moment}
\sup_t \max_i ||G(t,\mathcal{F}_i) ||_q<\infty.
\end{equation}
\end{assu}

 The following assumptions \ref{assu_smoothtrend} and \ref{assu_smmothness} states that the mean and covariance functions of $x_i$ are $d$-times continuous differentiable for some positive integer $d$.
\begin{assu}\label{assu_smoothtrend} For some given integer $d>0$, we assume that there exists a smooth function $\mu(\cdot) \in C^d([0,1]),$ where $C^d([0,1])$ is the function space on $[0,1]$ of continuous functions that have continuous first $d$ derivatives, such that
\begin{equation*}
\mathbb{E} \ G(t, \mathcal{F}_0)=\mu(t). 
\end{equation*}
\end{assu}

 For each fixed $t \in [0,1],$ we denote the covariance function of the locally stationary time series $\{x_i\}$ as 
  \begin{equation}\label{eq_defncov}
\gamma(t,j)=\text{Cov}(G(t, \mathcal{F}_0), G(t, \mathcal{F}_{j})).
\end{equation}
The assumptions
(\ref{assum_lip}) and (\ref{assum_moment}) ensure that   $\gamma(t,j)$ is Lipschiz continuous in $t$. Next, we impose the following mild assumption on the smoothness of $\gamma(t,j).$
\begin{assu}\label{assu_smmothness} There exists some integer $d>0,$  such that $\gamma(t,j) \in C^d([0,1])$ for any $j \geq 0$.
\end{assu}

Armed with  Assumptions \ref{phy_generalts}, \ref{assu_pdc}, \ref{assum_local} and \ref{assu_smmothness}, we can conclude that the covariance function $\gamma(t, j)$ decays polynomially fast uniformly in $t$ (c.f. Lemma \ref{lem_coll}).  Before concluding this section, we provide an insight on how to check UPDC condition for locally stationary time series. {
For stationary time series, Herglotz's theorem asserts that UPDC holds if the spectral density function is bounded from below by a constant (see \cite[Section 4.3]{BD} for more details). Our next proposition extends such results to locally stationary time series with short-range dependence.   
\begin{prop}\label{prop_pdc} If $\{x_i\}$ is  locally stationary time series satisfying Assumptions \ref{phy_generalts}, \ref{assum_local} and \ref{assu_smmothness}, and there exists some constant $\kappa>0$ such that $f(t,\omega) \geq \kappa$ for all $t$ and $\omega,$ where 
\begin{equation}\label{eq_spetraldensity}
f(t, \omega)=\sum_{j=-\infty}^{\infty} \gamma(t,j) e^{-\mathrm{i} j \omega}, \ \mathrm{i}=\sqrt{-1},
\end{equation} 
then $\{x_i\}$ satisfies UPDC. Conversely, if $\{x_i\}$ satisfies Assumptions \ref{phy_generalts}, \ref{assu_pdc}, \ref{assum_local} and \ref{assu_smmothness}, then there exists some constant $\kappa>0,$ such that $f(t,\omega) \geq \kappa$ for all $t$ and $\omega.$
\end{prop}

We shall call $f(t,w)$  the \emph{instantaneous spectral density function.} Proposition \ref{prop_pdc} implies that the verification of UPDC reduces to showing that the instantaneous spectral density function is uniformly bounded from below by a constant, which can be easily checked for many non-stationary processes. Finally, we list the following example satisfying Assumptions \ref{phy_generalts}, \ref{assum_local} and \ref{assu_smmothness} and the UPDC condition using Proposition \ref{prop_pdc}.} 

\begin{exam}[Non-stationary linear process]\label{eexample_linear} Let $\{\epsilon_i\}$ be {zero-mean i.i.d. random variables with variance $\sigma^2$.} We also assume $a_j(\cdot), j=0,1,\cdots$ be $C^d([0,1])$ functions such that 
\begin{equation}\label{ex_linear}
G(t, \mathcal{F}_i)=\sum_{k=0}^{\infty} a_k(t) \epsilon_{i-k}. 
\end{equation} 
It is easy to see that Assumptions \ref{phy_generalts}, \ref{assum_local} and \ref{assu_smmothness} will be satisfied if 
\begin{equation*}
 \sup_{t \in [0,1]} |a_j(t)| \leq C j^{-\tau}, \ j \geq 1; \ \sum_{j=0}^{\infty} \sup_{t \in [0,1]} |a_j'(t)|<\infty, 
\end{equation*} 
 and 
\begin{equation*}
 \sup_{t \in [0,1]} |a_j^{(d)}(t)|\leq  C j^{-\tau}, \ j \geq 1.
\end{equation*} 
{Further, we note that the instantaneous spectral density function of $G(t, \mathcal{F}_i)$ can be written as $f(t,w)= \sigma^2|\psi(t, e^{-\mathrm{i} j \omega})|^2,$ where $\psi(\cdot,\cdot)$ is defined such that $G(t, \mathcal{F}_i)=\psi(t, B) \epsilon_i$ with $B$ being the  backshift operator. By Proposition \ref{prop_pdc}, the UPDC is satisfied if $\sigma^2|\psi(t, e^{-\mathrm{i} j \omega})|^2 \geq \kappa$ for all $t$ and $\omega,$ where $\kappa>0$ is some universal constant. } 
\end{exam}

\subsection{Smooth AR approximation for locally stationary time series}

 In this subsection, we focus our discussion on locally stationary time series (\ref{defn_model}) satisfying Assumptions \ref{phy_generalts}, \ref{assu_pdc}, 
 \ref{assum_local}, \ref{assu_smoothtrend} and \ref{assu_smmothness}. 
We will show that there exists a smooth function $\phi_j(i/n)$ which approximates $\phi_{ij}$ when $i >b.$ {Specifically, denote $\widetilde{\bm{\phi}}^b(\frac{i}{n}):=(\phi_1(\frac{i}{n}), \cdots, \phi_b(\frac{i}{n}))^*$ via $\widetilde{\bm{\phi}}^b(\frac{i}{n})=(\widetilde{\Gamma}^b_i)^{-1} \widetilde{\bm{\gamma}}^b_i,$
where $\widetilde{\Gamma}^b_i$ and $\widetilde{\bm{\gamma}}^b_i$ are defined as 
\begin{equation*}
\widetilde{\Gamma}^b_i=\operatorname{Cov}(\widetilde{\bm{x}}_{i-1},\widetilde{\bm{x}}_{i-1}), \ \widetilde{\bm{\gamma}}_i=\operatorname{Cov}(\widetilde{\bm{x}}_{i-1}, \widetilde{x}_i),
\end{equation*}
with 
$\widetilde{\bm{x}}_{i-1,k}=G(\frac{i}{n},\mathcal{F}_{i-k}), \ k=1,2,\cdots,b,$ and $\widetilde{\bm{x}}_{i-1,k}$ is the $k$-th entry of $\widetilde{\bm{x}}_{i-1}.$ Moreover, we denote the time-varying function $\phi_0(t)$ by  
{
\begin{equation*}
\phi_0(\frac{i}{n})=\mu(\frac{i}{n})-\sum_{j=1}^{b} \phi_j(\frac{i}{n}) { \mu(\frac{i}{n})}.
\end{equation*}
}
The statistical properties of the coefficients are summarized in the following theorem. 
%

\begin{thm} \label{thm_locallynonzero} Consider the locally stationary time series (\ref{defn_model}). Suppose Assumptions \ref{phy_generalts}, \ref{assu_pdc}, \ref{assum_local}, \ref{assu_smoothtrend} and \ref{assu_smmothness} hold true. Then we have that $\phi_j(t) \in C^d([0,1]), 0 \leq j \leq b.$ Furthermore, there exists some constant $C>0,$ such that for all $1 \leq j \leq b,$ we have 
\begin{equation*}
\sup_{i>b}\left| \phi_{ij}-\phi_j(\frac{i}{n}) \right| \leq C n^{-1+3(1+\epsilon)/(2\tau)}, \ \
\sup_{i>b}\left| \phi_{i0}-\phi_0(\frac{i}{n}) \right| \leq C n^{-1+5(1+\epsilon)/(2\tau)}.
\end{equation*}
\end{thm}

Armed with Theorem \ref{thm_locallynonzero}, we find that 
\begin{equation}\label{eq_choleskylocal}
x_i=\phi_{0}(\frac{i}{n})+\sum_{j=1}^b \phi_j(\frac{i}{n}) x_{i-j}+\epsilon_i+O_{\mathbb{P}}(n^{-1+5(1+\epsilon)/(2\tau)}).
\end{equation}
Moreover, results similar to Theorem \ref{thm_arrepresent} can be proved, which is summarized in the following corollary. Denote
\begin{equation*}
x_i^{**}=
\begin{cases}
x_i,  & i \leq b; \\
\phi_0(\frac{i}{n})+\sum_{j=1}^b \phi_{j}(\frac{i}{n}) x_{i-j}^{**}+\epsilon_i, & i>b.  
\end{cases}
\end{equation*}
\begin{cor} \label{cor_localboundmse} Suppose that the assumptions of Theorem \ref{thm_locallynonzero} hold.  Then we have 
\begin{equation*}
\max_{1 \le i \le n  }|x_i-x_i^{**}|=O_{\mathbb{P}}(n^{-1+5(1+\epsilon)/(2\tau)}). 
\end{equation*}
\end{cor} 
}


Next, when $i>b,$ we show that there exists a smooth function of time $\rho_j(t),$ such that for each lag $j,$ the PACF $\rho_{i,j}$ can be well approximated by $\rho_j(i/n):=\phi_{j,j}(i/n).$ Denote
\begin{equation}\label{eq_yulewalkereqlocalstationary}
\widetilde{\bm{\phi}}_j(\frac{i}{n})=\widetilde{\Omega}_{i,j} \widetilde{\bm{\gamma}}_{i,j}, \ \widetilde{\bm{\phi}}_j(\frac{i}{n})=\Big(\phi_{1,j}(\frac{i}{n}), \cdots,\phi_{j,j}(\frac{i}{n}) \Big)^*,
\end{equation}
where $\widetilde{\Omega}_{i,j}=[\text{Cov}(\widetilde{\bm{x}}_i^j, \widetilde{\bm{x}}_i^j)]^{-1}$ and $\widetilde{\bm{\gamma}}_{i,j}=\text{Cov}(\widetilde{\bm{x}}_i^j,\widetilde{x}_i),$ where $\widetilde{\bm{x}}_i^j:=(\widetilde{x}_{i-1}, \cdots, \widetilde{x}_{i-j})^*$ with   $\widetilde{x}_{i-r}=G(\frac{i}{n}, \mathcal{F}_{i-r}).$ Denote $\rho_j(\frac{i}{n})=\phi_{j,j}(\frac{i}{n}),\ i>b.$ We summarize the properties of $\rho_j(t)$ in the lemma below.  An example of PACF plot can be found in Figure \ref{pacf_exam}.

\begin{figure}[H]
\includegraphics[width=16cm,height=8cm]{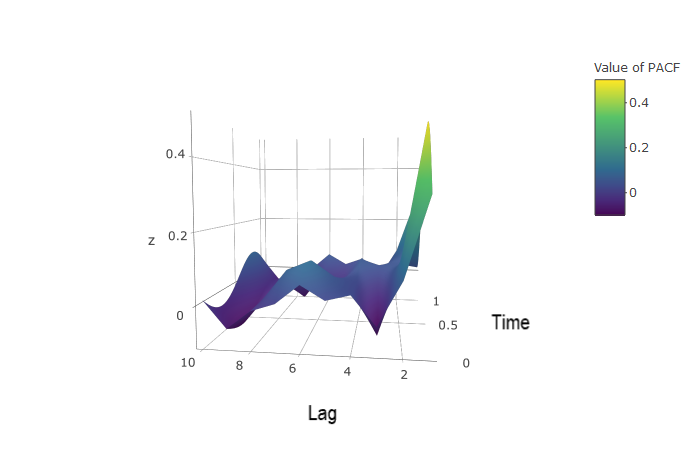}
\caption{Sample PACF plots for the first 10 lags for Model 3 defined in Section \ref{simu_intro}. It can be seen that the lag one sample PACF is much larger than the other lags at any time. }
\label{pacf_exam}
\end{figure}

\begin{lem}\label{lem_pacf}
Under Assumptions \ref{phy_generalts}, \ref{assu_pdc}, \ref{assum_local} and \ref{assu_smmothness}, for $i>b$ and $j<i,$ we have $\rho_j(t) \in C^d([0,1]).$ Moreover,  for some constant $C>0,$ we have   
\begin{equation*}
\sup_{i>b} \left| \rho_j(\frac{i}{n})-\rho_{i,j} \right| \leq Cn^{-1+3(1+\epsilon)/(2\tau)}.
\end{equation*}
\end{lem}

%
%

Finally, we can get similar results as in Remark \ref{rem_exponentialdecay} when the physical dependence measure is of exponential decay. 
%
%
%
%
%

\section{Globally optimal and adaptive forecasting for locally stationary time series}\label{sec:predict}
In this section, we consider short-term forecasting of locally stationary time series by estimating the smooth forecasting coefficients using sieve expansion. We first introduce the notation of \emph{asymptotically optimal predictor}.
{
\begin{defn}\label{defn_asymptotic}
A linear predictor $\widetilde{z}$ of a random variable $z$ based on $x_1,\cdots, x_n$ is called asymptotically optimal if 
\begin{equation}\label{eq_asymptotic}
\mathbb{E}(z-\widetilde z)^2\le \sigma_n^2+o(1/n),
\end{equation} 
 where $\sigma_n^2$ is the mean squared error (MSE) of the best linear predictor of $z$ based on $x_1,\cdots, x_n$. 
\end{defn}
}

The rationale for such definition is that, in practice, the MSE of forecast can only be estimated with a smallest possible error of $O(1/n)$ when time series length is $n$. It is well-known that the parametric rate for estimating the coefficients of a time series model is $O_p(n^{-1/2})$. When one uses the estimated coefficients to forecast the future, the corresponding influence on the MSE of forecast is $O(1/n)$ (at best). Therefore, if a linear predictor achieves an MSE of forecast within $o(1/n)$ range of the optimal one, it is practically indistinguishable from the optimal predictor asymptotically.  



\subsection{Asymptotically optimal short-term forecast for locally stationary time series}
In this subsection, we shall focus on the discussion of one-step ahead prediction. The general case will be discussed briefly due to similarity.  In order to make the forecasting feasible, we assume that the smooth data generating mechanism extends to time $n+1$. That is, we assume $x_{n+1}=G((n+1)/n,{\cal F}_{n+1})$. Naturally, we propose the following estimate for  $\hat x_{n+1}$, the best linear predictor of $x_{n+1}$ based on all its predecessors $x_1, \cdots, x_n$,
\begin{equation}\label{eq_forecast}
\widehat{x}_{n+1}^b=\phi_0(1)+\sum_{j=1}^b \phi_j(1)x_{n+1-j}, \ n>b.
\end{equation}
The next theorem shows that $\widehat{x}^b_{n+1}$ is an asymptotic optimal predictor satisfying (\ref{eq_asymptotic}) in Definition \ref{defn_asymptotic}. 

}


{
\begin{thm}\label{thm_prediction} Suppose Assumptions \ref{phy_generalts}, \ref{assu_pdc}, \ref{assum_local}, \ref{assu_smoothtrend} and \ref{assu_smmothness} hold true. Then there exists some constant $C>0,$ such that for sufficiently large $n$, 
\begin{equation}\label{eq_msemm}
\mathbb{E}(x_{n+1}-\widehat{x}^b_{n+1})^2 \leq \mathbb{E}(x_{n+1}-\widehat{x}_{n+1})^2+ C n^{-2+5(1+\epsilon)/\tau},
\end{equation}
\end{thm}
}
{
Theorem \ref{thm_prediction} states that the estimator (\ref{eq_forecast}) is an asymptotic optimal one-step ahead forecast {since $\tau>5+\varpi$  and $0<\epsilon \leq \varpi/10$. }


}


\begin{rem}
In the present paper, we focus on one-step ahead prediction.  However, our results can be easily extended to $h$-step ahead prediction for $h \leq h_0,$ where $h_0 \in \mathbb{N}$ is a fixed constant.  We briefly discuss such extension. For general non-stationary time series,  we denote by $\widehat{x}_{i,h}$ the $h$-step ahead best linear prediction of $x_i,$ i.e.
\begin{equation}\label{eq_hstepahead}
\widehat{x}_{i,h}=\phi_{i0,h}+\sum_{j=h}^{i-1} \phi_{ij,h} x_{i-j}, \ i \geq b+h. 
\end{equation}

For locally  stationary time series,  under Assumptions \ref{phy_generalts}, \ref{assu_pdc},  \ref{assum_local} and \ref{assu_smmothness}, it is easy to see that Theorem \ref{thm_locallynonzero} holds true when we replace $\phi_{ij}, \phi_{j}(\cdot)$ with $\phi_{ij,h}, \phi_{j,h}(\cdot).$ Further, (\ref{eq_choleskylocal})  holds true with $\widetilde{\bm{\phi}}^b(i/n)$ replaced with $$\widetilde{\bm{\phi}}^b_{h}(i/n):=(\phi_{1,h}(i/n), \cdots, \phi_{b,h}(i/n))=\widetilde{\Omega}_{i,h}^b \widetilde{\bm{\gamma}}_{i,h}^b,$$ where 
\begin{equation*}
\widetilde{\Omega}_{i,h}^b=[\text{Cov}(\widetilde{\bm{x}}_{i,h}, \widetilde{\bm{x}}_{i,h})]^{-1}, \ \widetilde{\bm{\gamma}}_{i,h}=\text{Cov}(\widetilde{\bm{x}}_{i,h}, \widetilde{x}_i), \  \widetilde{\bm{x}}_{i,h}=(x_{i-h}, \cdots, x_{i-h-b+1}) \in \mathbb{R}^b. 
\end{equation*} 

 For the $h$-step ahead prediction,  we use   
\begin{equation}\label{eq_hhhhh}
\widehat{x}_{n+h}^b=\phi_{0,h}(1)+\sum_{j=h}^{b} \phi_{j,h}(1) x_{n+1-j}, \ n \geq b+h. 
\end{equation} 
Finally, Theorem \ref{thm_prediction} holds true when we place $x_{n+1}, \widehat{x}_{n+1}^b$ with $x_{n+h}, \widehat{x}_{n+h}^b.$
\end{rem}  
%

\subsection{Sieve estimation of AR coefficients and MSE of forecast}\label{subsec:arcoeff}
In this section, we propose a global and adaptive nonparametric sieve method to estimate the coefficient functions $\phi_{j}(\cdot)$ $j=0,\cdots, b$ and the associated MSE of forecast. 
Specifically,  since $\phi_j(t) \in C^d([0,1]),$ we employ the sieve method to approximate it via a finite and diverging term basis expansion.  By \cite[Section 2.3]{CXH}, we have that
\begin{equation}\label{eq_phiform}
\phi_j(\frac{i}{n})=\sum_{k=1}^c a_{jk} \alpha_k(\frac{i}{n})+O(c^{-d}), \ 0 \leq j \leq b, \ i>b,
\end{equation}
where $\{\alpha_k(t)\}$ are some pre-chosen basis functions on $[0,1]$ and $c$ is the number of basis functions which is of the order
\begin{equation}\label{eq_defnc}
c=O(n^{\alpha_1}), \ 0<\alpha_1<1 \ \text{is some given constant whose value is determined by } d.  
\end{equation} 

In light of (\ref{eq_phiform}),  we need to estimate $a_{jk}.$ For $i>b,$ by (\ref{eq_choleskylocal}), write
\begin{equation}\label{eq_choeq}
x_i=\sum_{j=0}^b \sum_{k=1}^c a_{jk} z_{kj}+\epsilon_i+O_{\mathbb{P}}(n^{-1+5(1+\epsilon)/(2\tau)}+bc^{-d}), 
\end{equation}
where $z_{kj} \equiv z_{kj}(i/n):=\alpha_k(i/n) x_{i-j}$ for $j \geq 1$ and $z_{k0}=\alpha_k(i/n).$ {By (\ref{eq_choeq}), we estimate all the $a_{jk}'s$ using one ordinary least squares (OLS) regression with diverging number of predictors. In particular, we write all $a_{jk}, \ j=0,1,2 \cdots, b, \ k=1,2,\cdots, c$ as a vector $\bm{\beta} \in \mathbb{R}^{(b+1)c},$  then the OLS estimator for $\bm{\beta}$ can be written as $\widehat{\bm{\beta}}=(Y^* Y)^{-1}Y^* \bm{x}$, where $\bm{x}=(x_{b+1}, \cdots, x_n)^* \in \mathbb{R}^{n-b}$ and $Y$ is the design matrix.  After estimating $a_{jk}'s,$ $\phi_j(i/n)$ is estimated using (\ref{eq_phiform}). Specifically, 
\begin{equation}\label{eq_phijest}
\widehat{\phi}_j(\frac{i}{n})=\widehat{\bm{\beta}}^*\mathbb{B}_j(\frac{i}{n}),
\end{equation} 
where $\mathbb{B}_j(i/n):=\mathbb{B}_{j,b}(i/n) \in \mathbb{R}^{(b+1)c}$  has $(b+1)$ blocks and the $j$-th block is $\mathbf{B}(\frac{i}{n}), \ j=0,1,2,\cdots,b,$ and zeros otherwise. Next, we provide an example to list some commonly used basis functions.  We also refer to \cite[Section 2.3]{CXH} for a more detailed discussion.  
 
 }


\begin{exam} \label{exam_basis}
(1). Normalized Fourier basis.  For $x \in [0,1],$ consider the following trigonometric polynomials 
\begin{equation*}
\Big\{1, \sqrt{2} \cos(2 k \pi x), \ \sqrt{2} \sin (2 k \pi x), \cdots \Big\}, k \in \mathbb{N}.
\end{equation*}
{We note that the classical trigonometric basis function is well suited for approximating periodic functions on $[0, 1]$. }

(2). Normalized Legendre polynomials \cite{BWbook}. The Legendre polynomial of degree $n$ can be obtained using Rodrigue's formula
\begin{equation*}
P_n(x)=\frac{1}{2^n n!} \frac{d^n}{dx^n} (x^2-1)^n,\ -1 \leq x \leq 1.
\end{equation*}
In this paper, we use the normalized Legendre polynomial 
\begin{equation*}
P_n^*(x)=
\begin{cases}
1, & n=0; \\
\sqrt{\frac{2n+1}{2}}P_n(2x-1), &, n>0.
\end{cases}
\end{equation*}
The coefficients of  the Legendre  polynomials can be obtained using the R package \texttt{mpoly} {and hence  they are easy to implement in R. 

(3).  Daubechies orthogonal wavelet \cite{ ID98,ID92}. For $N \in \mathbb{N},$ a Daubechies (mother) wavelet of class $D-N$ is a function $\psi \in L^2(\mathbb{R})$ defined by 
\begin{equation*}
\psi(x):=\sqrt{2} \sum_{k=1}^{2N-1} (-1)^k h_{2N-1-k} \varphi(2x-k),
\end{equation*}
where $h_0,h_1,\cdots,h_{2N-1} \in \mathbb{R}$ are the constant (high pass) filter coefficients satisfying the conditions
$\sum_{k=0}^{N-1} h_{2k}=\frac{1}{\sqrt{2}}=\sum_{k=0}^{N-1} h_{2k+1},$
as well as, for $l=0,1,\cdots,N-1$
\begin{equation*}
\sum_{k=2l}^{2N-1+2l} h_k h_{k-2l}=
\begin{cases}
1, & l =0 ,\\
0, & l \neq 0.
\end{cases}
\end{equation*} 
And $\varphi(x)$ is the scaling (father) wavelet function is supported on $[0,2N-1)$ and satisfies the recursion equation
$\varphi(x)=\sqrt{2} \sum_{k=0}^{2N-1} h_k \varphi(2x-k),$ as well as the normalization $\int_{\mathbb{R}} \varphi(x) dx=1$ and
$ \int_{\mathbb{R}} \varphi(2x-k) \varphi(2x-l)dx=0, \ k \neq l.$
Note that the filter coefficients can be efficiently computed as listed in \cite{ID92}.  The order $N$, on the one hand, decides the support of our wavelet; on the other hand, provides the regularity condition in the sense that
\begin{equation*}
\int_{\mathbb{R}} x^j \psi(x)dx=0, \ j=0,\cdots,N, \  \text{where} \ N \geq d. 
\end{equation*}
We will employ Daubechies wavelet with a sufficiently high order when forecasting in our simulations and data analysis. The basis functions can be either generated using the library \texttt{PyWavelets} in Python \footnote{For visualization for the families of Daubechies wavelet functions, we refer to \url{http://wavelets.pybytes.com}, where the library \texttt{PyWavelets} is also introduced there.} or the \texttt{wavefun} in the \texttt{Wavelet Toolbox} of Matlab. In the present paper, to construct a sequence of orthogonal wavelet, we will follow the dyadic construction of \cite{ID98}. For a given $J_n$ and $J_0,$ we will consider the following periodized wavelets on $[0,1]$
\begin{equation}\label{eq_constructone}
\Big\{ \varphi_{J_0 k}(x), \ 0 \leq k \leq 2^{J_0}-1; 
\psi_{jk}(x), \ J_0 \leq j \leq J_n-1, 0 \leq k \leq 2^{j}-1  \Big\},\ \mbox{ where}
\end{equation} 
\begin{equation*}
\varphi_{J_0 k}(x)=2^{J_0/2} \sum_{l \in \mathbb{Z}} \varphi(2^{J_0}x+2^{J_0}l-k) , \
\psi_{j k}(x)=2^{j/2} \sum_{l \in \mathbb{Z}} \psi(2^{j}x+2^{j}l-k),
\end{equation*}
or, equivalently \cite{MR1085487}
\begin{equation}\label{eq_meyerorthogonal}
\Big\{ \varphi_{J_n k}(x), \  0 \leq k \leq  2^{J_n-1} \Big\}.
\end{equation}
}
\end{exam}

In light of (\ref{eq_phijest}), with the estimates $\widehat{\phi}_j(\cdot), j=0,1,2,\cdots,b,$ we forecast $x_{n+1}$ using
\begin{equation*}
\widehat{\mathsf{x}}_{n+1}^b=\widehat{\phi}_0(1)+\sum_{j=1}^b \widehat{\mathsf{\phi}}_j(1) x_{n+1-j}.
\end{equation*}

Next, we shall discuss the estimation of the MSE of forecast, i.e., the variance of $\{\epsilon_{n+1}\}$.  Denote the series of estimated forecast error $\{\widehat{\epsilon}^b_i\}$ by
$\widehat{\epsilon}_i^b:=x_i-\sum_{j=1}^b \widehat{\mathsf{\phi}}_j(i/n) x_{i-j}.$
Let the variance of $\{\epsilon_i\}$ be $\{\sigma_i^2\}.$  Similar to \cite[Lemma 3.11]{DZ1}, we find that there exists a smooth function $\varphi(\cdot) \in C^d([0,1])$ such that for some constant $C>0,$
\begin{equation}\label{eq_vvvffff}
\sup_{i>b} |\sigma_i^2-\varphi(\frac{i}{n})| \leq C n^{-1+5(1+\epsilon)/(2\tau)}. 
\end{equation}
Therefore, we shall use sieve expansion to estimate the smooth function $\varphi(\cdot).$ Similar to (\ref{eq_phiform}), we have 
\begin{equation*}
\varphi(\frac{i}{n})=\sum_{k=1}^c \mathfrak{b}_k \alpha_k(\frac{i}{n})+O(c^{-d}). 
\end{equation*}
Furthermore,  by equation (3.14) of \cite{DZ1}, write 
\begin{equation*}
(\widehat{\epsilon}_i^b)^2=\sum_{k=1}^c \mathfrak{b}_k \alpha_k(\frac{i}{n})+\nu_i+O_{\mathbb{P}}\Big(n^{(1+\epsilon)/\tau}(\zeta_c \frac{\log n}{\sqrt{n}}+n^{-d \alpha_1})\Big), \ i \geq b,
\end{equation*}
where $\{\nu_i\}$ is a centred sequence of locally stationary time series satisfying Assumptions \ref{phy_generalts}, \ref{assu_pdc}, \ref{assum_local}, \ref{assu_smoothtrend} and \ref{assu_smmothness}. Consequently, we can use an OLS with $(\widehat{\epsilon}_i^b)^2$ being the response and $\alpha_k(\frac{i}{n})$, $k=1,\cdots, c$ being the explanatory variables  to estimate $\mathfrak{b}_k,$ which are denoted as $\widehat{\mathfrak{b}}_k, k=1,2,\cdots,c.$  
Finally, we estimate 
$$\widehat{\varphi}(i/n)=\sum_{k=1}^c \widehat{\mathfrak{b}}_k \alpha_k(i/n).$$

We are now ready to state the asymptotic behaviour of the estimated coefficients and MSE of (\ref{eq_forecast}). Denote $\zeta_c=\sup_i|| \mathbf{B}(i/n) ||.$ Recall (\ref{eq_defnc}). Note that for the commonly used sieve basis functions, we have $\zeta_c=O(n^{\alpha_1^*}),$ where $\alpha_1^*=\frac{1}{2} \alpha_1$ for the Fourier basis and orthogonal wavelet, and $\alpha_1^*=\alpha_1$ for Legendre polynomial.

\begin{thm}\label{thm_finalresult}
Suppose Assumptions \ref{phy_generalts}, \ref{assu_pdc}, \ref{assum_local},
\ref{assu_smmothness} and \ref{assu_basis} hold true. Then we have 
\begin{equation*}
\sup_{i>b, 0 \leq j \leq b} \left| \phi_j(\frac{i}{n})-\widehat{\phi}_j(\frac{i}{n}) \right|=O_{\mathbb{P}}\left( \zeta_c \sqrt{\frac{\log n}{n}}+n^{-d\alpha_1} \right). 
\end{equation*}
Furthermore, by using the Fourier basis or orthogonal wavelets, if $c=O((n/(\log n))^{1/(2d+1)})$, we have  
\begin{equation}\label{eq:rate}
\sup_{i>b,0 \leq
 j \leq b} \left| \phi_j(\frac{i}{n})-\widehat{\phi}_j(\frac{i}{n}) \right|=O_{\mathbb{P}}\left(\left(n/(\log n) \right)^{-d/(2d+1)}\right). 
\end{equation}
Specifically, when $i=n,$ we obtain the uniform convergence rate for the coefficients of (\ref{eq_forecast}), Moreover, for the MSE of forecast, we have 
\begin{equation*}
\left |{\sigma_{n+1}^2}-\widehat{\varphi}(1)\right |=O_{\mathbb{P}}\Big(n^{(1+\epsilon)/\tau}(\zeta_c \frac{\log n}{\sqrt{n}}+n^{-d \alpha_1})+n^{-1+5(1+\epsilon)/(2\tau)}\Big). 
\end{equation*}
\end{thm} 
The rate in (\ref{eq:rate}) achieves  globally minimax rate for nonparametric function estimation in the sense of \cite{stone1982}.   


{

}

\subsection{Sieve estimation of the PACF of locally stationary time series}\label{sec:estpacf}
In this section, we discuss the estimation of the PACF for locally stationary time series using the method of sieves.  
By a discussion similar to (\ref{eq_phiform}), we have that
\begin{equation}\label{eq_rhocxh}
\rho_j(\frac{i}{n})=\sum_{k=1}^c d_{jk} \alpha_k(\frac{i}{n})+O(c^{-d}), \ i>b. 
\end{equation}
Therefore, similar to (\ref{eq_choeq}), we can write
\begin{equation*}
x_i=\sum_{s=0}^j \sum_{k=1}^c d_{sk} z_{ks}+\epsilon_{i,j}+O_{\mathbb{P}}(n^{-1+5(1+\epsilon)/(2\tau)}+jc^{-d}), \ i>b,
\end{equation*}
where $\epsilon_{i,j}$ is defined in (\ref{eq_jorderlocal}), $z_{ks}(i/n)=\alpha_k(i/n)x_{i-s}$ for $j \geq 1$ and $z_{k0}=\alpha_k(i/n).$ 
 Let $Y_j$ be the $(n-b) \times (j+1)c$ rectangular matrix whose $i$-th row is $\bm{x}_{i,j} \otimes \mathbf{B}(l/n),$  where $\bm{x}_{i,j}=(1,x_{i-1}, \cdots, x_{i-j}) \in \mathbb{R}^{j+1},$ $\mathbf{B}(i/n)=(\alpha_1(i/n), \cdots, \alpha_c(i/n)) \in \mathbb{R}^c$. We put all the $d_{sk}'s,  s=0,1,\cdots, j, \ k=1,2,\cdots, c$ into a vector $\bm{d}_j \in \mathbb{R}^{(j+1)c}.$ The OLS estimator for $\bm{d}_j$ can be written as $\bm{\widehat{d}}_j=(Y_j^* Y_j)^{-1} Y_j^* \bm{x},$ where $\bm{x}=(x_{b+1}, \cdots, x_n)^* \in \mathbb{R}^{n-b}.$ 
 
 Denote the sieve estimator of PACF as 
\begin{equation*}
\widehat{\rho}_j(\frac{i}{n})=\sum_{k=1}^c \widehat{d}_{jk} \alpha_k(\frac{i}{n}). 
\end{equation*}
%
We have the following results. 

%
%
\begin{thm} \label{thm_consistency}
Suppose Assumptions \ref{phy_generalts}, \ref{assu_pdc}, \ref{assum_local}, \ref{assu_smoothtrend},
\ref{assu_smmothness} and \ref{assu_basis} hold true. Then we have 
\begin{equation*}
\sup_{i>b,j \leq b} \left| \rho_j(\frac{i}{n})-\widehat{\rho}_j(\frac{i}{n}) \right|=O_{\mathbb{P}}\left( \zeta_c \sqrt{\frac{\log n}{n}}+n^{-d\alpha_1} \right). 
\end{equation*}
Furthermore, by using the Fourier basis and orthogonal wavelets, if $c=O((n/(\log n))^{1/(2d+1)})$, we have  
\begin{equation*}
\sup_{i>b,j \leq b} \left| \rho_j(\frac{i}{n})-\widehat{\rho}_j(\frac{i}{n}) \right|=O_{\mathbb{P}}\left(\left(n/(\log n) \right)^{-d/(2d+1)}\right). 
\end{equation*}
\end{thm}



{

}



%

\section{Test of stability for locally stationary time series prediction}\label{sec:test}


As we mentioned in the introduction, it is of great practical importance to check whether the optimal forecasting coefficients are indeed time-varying. In this section, we propose a test of stability for the best linear prediction
 based on the estimated AR coefficients from Section \ref{subsec:arcoeff}.


As $\phi_0(\cdot)$ is related to the trend of the time series and in some real applications the trend is removed before performing forecasting, we shall first test the stability of $\phi_j(\cdot)$,  $1 \leq j \leq b.$ Furthermore, as we observe from Theorem \ref{lem_reducedtest} below, test of stability for $\phi_j(\cdot)$,  $1 \leq j \leq b$ is asymptotically equivalent to testing correlation stationarity of $x_i$ which may be of separate interest. Formally, the null hypothesis we would like to test is
\begin{equation}\label{eq_testnonzero1}
\mathbf{H}_{0}: \ \phi_j(\cdot) \ \text{is a  constant function} \ \text{on} \ [0,1], \ j=1,2,\cdots,b.   
\end{equation} 
%
%
%
Before providing the test statistic for $\mathbf{H}_0$, we shall first investigate the interesting insight that $\mathbf{H}_0$ is asymptotically equivalent to testing whether $\{x_i\}_{i=1}^n$ is correlation stationary, i.e., there exists some function $\varrho$ such that
\begin{equation}\label{defn_ho}
\mathbf{H}^{\prime }_{0}: \ \operatorname{Corr} (x_i, x_j)=\varrho(|i-j|),
\end{equation}
where $\text{Corr}(x_i, x_j)$ stands for the correlation between $x_i$ and $x_j.$
We formalize the above statements in Theorem \ref{lem_reducedtest}. 
\begin{thm}\label{lem_reducedtest}
Suppose Assumptions \ref{phy_generalts}, \ref{assu_pdc}, \ref{assum_local} and \ref{assu_smmothness} hold true. 
For  $j \leq b$, on the one hand, when $\mathbf{H}_0^{\prime}$ holds true, then $\phi_j(\frac{i}{n})=\phi_j$ which is independent of time. On the other hand, when $\phi_j(\frac{i}{n})=\phi_j, j=1,2,\cdots, b$, there exists some smooth function $\varrho,$ such that 
\begin{equation*}
\operatorname{Corr} (x_i, x_{i+j})=\varrho_{|i-j|}+O(n^{-1+(1+\epsilon)/\tau}).
\end{equation*}   
\end{thm}

In some cases, practitioners and researchers may be interested  in testing whether all optimal forecast coefficient functions $\phi_j(\cdot)$, $j=0,1,\cdots, b$ do not change over time. That is equivalent to testing whether both the trend and the correlation structure of the time series stay constant over time. 
In this case, one will test 
\begin{equation}\label{eq_nonzeromeangeneral}
\mathbf{H}_{0,g}: \ \phi_j(\cdot) \ \text{is a constant function}  \ \text{on} \ [0,1], \ j=0,1,\cdots, b. 
\end{equation}

\subsection{Test statistics and asymptotic normality}
In this subsection, we propose test statistics for  $\mathbf{H}_0$ in (\ref{eq_testnonzero1}) and $\mathbf{H}_{0,g}$ in (\ref{eq_nonzeromeangeneral}).   Recall (\ref{eq_phijest}). We focus our discussion on $\mathbf{H}_0$ and briefly discuss $\mathbf{H}_{0,g}$ in the end. To test $\mathbf{H}_0,$ we use the following statistic 
\begin{equation} \label{eq_defnntori}
T=\sum_{j=1}^b \int_0^1(\widehat{\phi}_j(t)-\overline{\widehat{\phi}}_j)^2 dt,  \  \overline{\widehat{\phi}}_j=\int_0^1 \widehat{\phi}_j(t) dt.
\end{equation}  
{\color{green}  

}

Next, we show that the study of the statistic $T$ reduces to the investigation of a weighted quadratic form of high dimensional non-stationary time series. Denote 
$\bar{B}=\int_0^1 \mathbf{B}(t)dt$ and $W=I-\bar{B} \bar{B}^*.$ Let $\mathbf{W}$ be a $(b+1)c \times (b+1)c$ dimensional diagonal block matrix with diagonal block $W$ and $\mathbf{I}_{bc}$ be a $(b+1)c \times (b+1)c$ dimensional diagonal matrix whose non-zero entries are ones and  in the lower $bc \times bc$ major part. Recall $\bm{x}_i=(1, x_{i-1}, \cdots, x_{i-b})^*.$ Let $p=(b+1)c.$ We denote the sequence of $p$-dimensional vectors $\bm{z}_i$ by
\begin{equation}\label{eq_zikronecker}
\bm{z}_i=\bm{h}_i \otimes \mathbf{B}(\frac{i}{n}) \in \mathbb{R}^p, \ \bm{h}_i=\bm{x}_i \epsilon_i.
\end{equation}
\begin{lem}\label{lem_reducedquadratic}
Denote
$\mathbf{X}=\frac{1}{\sqrt{n}} \sum_{i=b+1}^n \bm{z}_i^*,$ 
and  the $p \times p$ matrix $\Gamma$ by
$
\Gamma=\overline{\Sigma}^{-1} \mathbf{I}_{bc} \mathbf{W} \overline{\Sigma}^{-1}, 
$
where 
\begin{equation}\label{eq_overlinesigma}
\overline{\Sigma}=\begin{pmatrix}
\mathbf{I}_c & \bm{0} \\
\bm{0} & \Sigma
\end{pmatrix}, \ \Sigma=\int_0^1 \Sigma^b(t) \otimes (\mathbf{B}(t) \mathbf{B}^*(t))dt.
\end{equation}
Then under the assumptions of Theorem \ref{thm_finalresult}, we have 
\begin{equation}\label{eq_tfinalexpress}
nT=\mathbf{X}^* \Gamma \mathbf{X}+o_{\mathbb{P}}(1). 
\end{equation}
\end{lem}

From Lemma \ref{lem_reducedquadratic}, we find that it suffices to establish the distribution of $\mathbf{X}^* \Gamma \mathbf{X}.$ To this end, we shall establish a Gaussian approximation result for this quadratic form of high-dimensional non-stationary time series $\{\bm{z}_i\}$. 
Note that when $i>b,$ $\bm{h}_i$ is a locally stationary time series 
\begin{equation}\label{eq_defnh}
\bm{h}_i=\mathbf{U}(\frac{i}{n}, \mathcal{F}_i), \ i>b. 
\end{equation}

Choose a sequence of centred Gaussian random vectors $\{\bm{v}_i\}_{i=b+1}^n$ which preserves the covariance structure of $\{\bm{h}_i\}_{i=b+1}^n$ and define 
$\bm{g}_i=\bm{v}_i \otimes \mathbf{B}(\frac{i}{n}).$
Denote 
\begin{equation*}
\mathbf{Y}=\frac{1}{\sqrt{n}} \sum_{i=b+1}^n \bm{g}_i^*.
\end{equation*}
We will control the Kolmogorov distance 
\begin{equation}\label{eq_kolomogorv}
\mathcal{K}(\mathbf{X}, \mathbf{Y})=\sup_{x \in \mathbb{R}} \left| \mathbb{P} \Big( \mathbf{X}^*\Gamma \mathbf{X} \leq x \Big)-\mathbb{P} \Big( \mathbf{Y}^*\Gamma \mathbf{Y} \leq x \Big) \right|.
\end{equation}
Denote 
\begin{equation*}
\xi_c:=\sup_{1\leq i \leq c} \sup_{t \in [0,1]} \Big |\alpha_i(t) \Big |.
\end{equation*}
It is notable that $\xi_c$ can be well controlled for commonly used basis functions. For instance, $\xi_c=O(1)$ for the Fourier basis and the normalized orthogonal polynomials; $\xi_c=O(\sqrt{c})$ for orthogonal wavelet.  The following theorem provides a bound for $\mathcal{K}(\mathbf{X}, \mathbf{Y}).$
\begin{thm} \label{thm_gaussian} Under Assumptions \ref{phy_generalts}, \ref{assu_pdc},  \ref{assum_local}, \ref{assu_smoothtrend}, \ref{assu_smmothness} and \ref{assu_basis}, there exists a constant $C>0$ and a small constant $\delta>0,$ such that 
\begin{align*}
\mathcal{K}(\mathbf{X}, \mathbf{Y}) \leq C \Big(\frac{\xi_c}{M_z}+& p^{\frac{7}{4}} n^{-1/2}M_z^3 M^2+M^{\frac{-q\tau+1}{2q+1}}\xi_c^{(q+1)/(2q+1)} p^{\frac{q+1}{2q+1}} n^{\frac{\delta q}{2q+1}}  \\
& +p^{1/4}\xi_c^{1/2}\left(   (p\xi_c M^{-\tau+1}+p\xi_c^q M_z^{-(q-2)}) \right)^{1/2} +n^{-\delta} \Big),
\end{align*}
where $M_z, M \rightarrow \infty$ when $n \rightarrow \infty.$
\end{thm}

{As indicated by Theorem \ref{thm_gaussian}, since $\tau>0$ is large, when $\xi_c=O(1)$ and $q>0$ is large enough, we can allow $p=n^{2/7-\delta_1},$ where $\delta_1>0$ is a sufficiently small constant.} Asymptotic normality of $nT$ can be readily derived by the above Gaussian approximation. Denote the long-run covariance matrix for $\{\bm{h}_i\}$ as
\begin{equation}\label{eq_longrunh}
\Omega(t)=\sum_{j=-\infty}^{\infty} \text{Cov} \Big(\mathbf{U}(t, \mathcal{F}_0), \mathbf{U}(t, \mathcal{F}_j) \Big),
\end{equation}
and the aggregated covariance matrix as 
$\Omega=\int_0^1 \Omega(t) \otimes \Big( \mathbf{B}(t) \mathbf{B}^*(t) \Big)dt.$ 
$\Omega$ can be regarded as the integrated long-run covariance matrix of $\{\bm{h}_i\}.$ For $k \in \mathbb{N}, $ we define
\begin{equation} \label{eq_f}
f_k=\Big( \text{Tr}[ \Omega^{1/2} \Gamma \Omega^{1/2} ]^k \Big)^{1/k}.
\end{equation}
\begin{prop} \label{prop_normal} Under Assumptions \ref{phy_generalts}, \ref{assu_pdc},  \ref{assum_local}, \ref{assu_smoothtrend}, \ref{assu_smmothness} and \ref{assu_basis}, when $\mathbf{H}_0$ holds true, we have 
\begin{equation*}
\frac{nT-f_1}{f_2} \Rightarrow \mathcal{N}(0,2).
\end{equation*}
\end{prop}
{
We now discuss the power of the test under the following local alternative. For a given $\alpha,$
\begin{equation*}
\mathbf{H}_a:   \sum_{j=1}^{\infty} \int_0^1 \Big( \phi_j(t)-\bar{\phi}_j\Big)^2 dt>C_{\alpha} \frac{\sqrt{bc}}{n},
\end{equation*}
where $\bar{\phi}_j=\int_0^1 \phi_j(t)dt$ and 
$C_{\alpha} \equiv C_{\alpha}(n) \rightarrow \infty $  as $n \rightarrow \infty.$ For instance, we can choose
$C_{\alpha}> n^{\kappa} \mathcal{Z}_{1-\alpha}, \ \kappa>0,$ where $\mathcal{Z}_{1-\alpha}$ is the $(1-\alpha)\%$ quantile of the standard Gaussian distribution.

\begin{prop} \label{prop_power} Under Assumptions \ref{phy_generalts}, \ref{assu_pdc}, \ref{assum_local}, \ref{assu_smoothtrend}, \ref{assu_smmothness} and \ref{assu_basis}, when $\mathbf{H}_a$ holds true, we have 
\begin{equation*}
\frac{nT-f_1-n\sum_{j=1}^{\infty} \int_0^1 \Big( \phi_j(t)-\bar{\phi}_j \Big)^2 dt}{f_2} \Rightarrow \mathcal{N}(0,2). 
\end{equation*}
\end{prop}
}

The above proposition states that under  $\mathbf{H}_a,$ the power of our test will asymptotically be 1, i.e., 
\begin{equation*}
\mathbb{P} \Big( \left| \frac{nT-f_1}{f_2} \right| \geq \sqrt{2} \mathcal{Z}_{1-\alpha} \Big) \rightarrow 1, \ n \rightarrow \infty.
\end{equation*}

{Finally, we briefly discuss how to test $\mathbf{H}_{0,g}$. 
To test $\mathbf{H}_{0,g}$ in (\ref{eq_nonzeromeangeneral}), we shall use 
\begin{equation*} 
T_g=\sum_{j=0}^b \int_0^1(\widehat{\phi}_j(t)-\overline{\widehat{\phi}}_j)^2 dt,  \  \overline{\widehat{\phi}}_j=\int_0^1 \widehat{\phi}_j(t) dt.
\end{equation*} 
It can be further written as 
\begin{equation*}
nT_g=\mathbf{X}^* \Gamma_g \mathbf{X}+o_{\mathbb{P}}(1), \ \Gamma_g=\overline{\Sigma}^{-1}  \mathbf{W} \overline{\Sigma}^{-1},
\end{equation*}
where we recall (\ref{eq_overlinesigma}). 
%
By Theorem \ref{thm_gaussian}, we can prove similar results to $n T_g$ as in Propositions \ref{prop_normal} and \ref{prop_power}.  We shall omit further details. 
}

\subsection{Robust bootstrap procedure}\label{sec_bootstrapping}
It is difficult to directly use Proposition \ref{prop_normal} to carry out the stability test since the quantities $f_1$ and $f_2$ are hard to estimate. Additionally, the high-dimensional Gaussian quadratic form  $\mathbf{Y}^*\Gamma \mathbf{Y}$ converges at a slow rate.  To overcome these difficulties, we extend the strategy of \cite{ZZ1} and use a high-dimensional mulitplier bootstrap statistic to mimic the distributions of $nT$ and $nT_g$. We focus on the discussion of $nT.$ Recall that
\begin{equation} \label{eq_defnnt}
nT= \Big( \frac{1}{\sqrt{n}} \sum_{i=b+1}^n \bm{z}_i^* \Big) \Gamma \Big( \frac{1}{\sqrt{n}} \sum_{i=b+1}^n \bm{z}_i \Big).
\end{equation}
Recall (\ref{eq_zikronecker}).  Denote
\begin{equation}\label{eq_defnphi}
\Phi=\frac{1}{\sqrt{n-m-b+1}\sqrt{m}} \sum_{i=b+1}^{n-m} \Big[ \Big(\sum_{j=i}^{i+m} \bm{h}_i \Big) \otimes \Big( \mathbf{B}(\frac{i}{n}) \Big) \Big]  R_i,
\end{equation}
where $R_i, i=b+1, \cdots, n-m$ are i.i.d. standard Gaussian random variables. Denote the bootstrap quadratic form 
\begin{equation}\label{eq_defnmathcalt}
\mathcal{T}:=\Phi^* \widehat{\Gamma} \Phi,
\end{equation}
where $\widehat{\Gamma}:= \widehat{\Sigma}^{-1} \mathbf{I}_{bc} \mathbf{W} \widehat{\Sigma}^{-1}$ with $\widehat{\Sigma}=\frac{1}{n} Y^* Y.$ Since $\{\epsilon_i\}$ cannot be observed directly, we shall use the residuals
\begin{equation*}
\widehat{\epsilon}^b_i:=x_i-\widehat{\phi}_0(\frac{i}{n})-\sum_{j=1}^b \widehat{\phi}_j(\frac{i}{n})x_{i-j}.
\end{equation*}
Accordingly, define $\bm{\widehat{h}}_i,\widehat{\Phi}$ and $\widehat{\mathcal T}$ by replacing $\epsilon_i$ in $\bm{h}_i, \Phi$ and $\mathcal T$ with $\widehat\epsilon_i^b$, respectively.


We claim that $\widehat{\mathcal{T}}$ mimics the distribution of $nT$ asymptotically. Before formally introduce our results, we first introduce the following assumption, which states that $p$ diverges in a moderate way.  Let $m=O(n^{\alpha_2}).$ .

\begin{assu} \label{assu_parameter} 
For $\alpha_2 \in (0,1),$ we assume that 
\begin{equation*}
\frac{1}{\tau}+2\alpha_1^*+\frac{\alpha_2}{2}-\frac{1}{2} \alpha_1<\frac{1}{2}, \ \frac{1}{\tau}+2\alpha_1^*-\alpha_2-\frac{1}{2} \alpha_1<0.
\end{equation*}
Furthermore, we assume that Assumption \ref{assum_local} holds with $q>4.$
\end{assu} 
\begin{rem}
The above assumption is equivalent to 
\begin{equation}\label{eq_assumimply}
\sqrt{b}\zeta_c^2 c^{-1/2} \Big( \sqrt{\frac{m}{n}}+\frac{1}{m} \Big)=o(1).
\end{equation}
Therefore, when $\alpha_1^*=\frac{1}{2} \alpha_1,$ the above assumption can be read as 
\begin{equation*}
\sqrt{p}   \Big( \sqrt{\frac{m}{n}}+\frac{1}{m} \Big)=o(1).
\end{equation*}
Hence, in the optimal case when $m=O(n^{1/3}),$ Assumption \ref{assu_parameter} allows one to choose $p \ll n^{2/3}.$ 
\end{rem}

\begin{thm}\label{thm_bootstrapping}
Under Assumptions \ref{phy_generalts}, \ref{assu_pdc}, \ref{assum_local}, \ref{assu_smoothtrend}, \ref{assu_smmothness}, \ref{assu_smoothtrend}, \ref{assu_basis} and \ref{assu_parameter}, when $\mathbf{H}_0$ holds true,  there exists some set $\mathcal{A}_n$ such that $\mathbb{P}(\mathcal{A}_n)=1-o(1)$ and under the event $\mathcal{A}_n,$ we have that conditional on the data $\{x_i\}_{i=b+1}^n,$
\begin{equation*}
\sup_{x \in \mathbb{R}} \left| \mathbb{P} \left( \frac{\widehat{\mathcal{T}}-f_1}{\sqrt{2}f_2} \leq x \right)-\mathbb{P} \left( \Psi \leq x \right) \right|=o(1),
\end{equation*} 
where $\Psi \sim \mathcal{N}(0,1)$ is a standard normal random variable. 
\end{thm}
For the detailed construction of $\mathcal{A}_n,$ we refer the reader to the proof of the above theorem. A theoretical discussion of  the accuracy of the bootstrap can be found in Section \ref{sec_appendix_one}.



%

{

Finally, the following steps are proposed for practical implementation of the bootstrap: 
\begin{enumerate}
\item Select the tuning parameters $b$, $c$ and $m$ by the methods demonstrated in Section \ref{sec:choiceparameter}. 
\item Compute $\widehat{\Sigma}^{-1}$ using $n(Y^*Y)^{-1}$ and the residuals $\{\widehat{\epsilon}^b_i\}_{i=b+1}^n.$
\item  Generate B (say 1000) i.i.d. copies of $\{\Phi^{(k)}\}_{k=1}^B.$ Compute $\widehat{\mathcal{T}}_k, k=1,2,\cdots, B$ correspondingly.
\item Let $\widehat{\mathcal{T}}_{(1)} \leq \widehat{\mathcal{T}}_{(2)} \leq \cdots \leq \widehat{\mathcal{T}}_{(B)}$ be the order statistics of $\widehat{\mathcal{T}}_k, k=1,2,\cdots, B.$ Reject $\mathbf{H}_0$ at the level $\alpha$ if $nT>\widehat{\mathcal{T}}_{(B(1-\alpha))},$ where $\lfloor x \rfloor$ denotes the largest integer smaller or equal to $x.$  Let $B^*=\max\{r: \widehat{\mathcal{T}}_{r} \leq nT\}.$ The $p$-value of the test can be computed as $1-\frac{B^*}{B}.$
\end{enumerate}


%

\subsection{{Test of PACF for locally stationary time series}}\label{sec_estimationpacf} 


This subsection is devoted to testing whether a group of PACF are uniformly zero across time which is important when selecting a preliminary order of an AR model.  
We observe  from Lemma \ref{lem_pacftruncation} and Theorem \ref{thm_consistency} that the following statistic should be small under $\mathbf{H}^{\rho,b_0,b_1}_0$ : $\rho_j(\cdot)\equiv 0$, $j=b_1,b_1+1,\cdots,b_0$,
\begin{equation*}
T_{\rho}=\sum_{j=b_1}^{b_0} \int_0^1 \widehat{\rho}_j^2(t) dt, \ b<b_1<b_0,
\end{equation*} 
where $b_0$ is a given sufficiently large lag. Consequently, $T_{\rho}$ can be used to test $\mathbf{H}_0^{\rho,b_0,b_1}$. However, in order to obtain the value of $\mathcal{T}_{\rho},$ we need to do $(b_0-b_1+1)$ high-dimensional OLS regressions which is computationally intensive. The following lemma suggests that we can simply use 
\begin{equation*}
T_{\phi}=\sum_{j=b_1}^{b_0} \int_0^1 \widehat{\phi}^2_j(t)dt, \ b<b_1<b_0,
\end{equation*}
as our test statistic, where only one high-dimensional OLS regression is need in order to obtain its value. 
\begin{lem}\label{lem_controltbtrho} For $b_0>b_1>b,$ we have that 
\begin{equation*}
T_{\phi}=T_{\rho}+O_{\mathbb{P}}(n^{-1+(4+3\epsilon)/\tau}).
\end{equation*}
\end{lem}
Similar to the discussion of  Theorem \ref{thm_bootstrapping}, $\mathcal{T}_{\phi}$ is normally distributed and so does $\mathcal{T}_{\rho}$. This is summarized as the following lemma. {Denote $\bm{w}_i=\bm{x}_{b_0, i} \epsilon_i, \ i>b_0,$ where $\bm{x}_{b_0,i}=(1,x_{i-1}, \cdots, x_{i-b_0})^* \in \mathbb{R}^{b_0+1}.$  Similar to the discussion of (\ref{eq_defnh}),  write 
\begin{equation}\label{eq_defnwi}
\bm{w}_i=\mathbf{V}(\frac{i}{n}, \mathcal{F}_i), i>b_0.
\end{equation}
Denote the long-run covariance matrix $\Pi(t)$
as
$
\Pi(t)=\sum_{j=-\infty}^{\infty} \text{Cov} \Big( \mathbf{V}(t, \mathcal{F}_0), \mathbf{V}(t, \mathcal{F}_j) \Big),
$
and the integrated long-run covariance matrix as 
$
\Pi= \int_{0}^1 \Pi(t) \otimes (\mathbf{B}(t) \mathbf{B}^*(t)) dt. 
$
For $k \in \mathbb{N},$ define 
\begin{equation*}
g_k=(\text{Tr}[\Pi^{1/2} \overline{\Sigma}_0^{-1} \mathbf{M} \overline{\Sigma}_0^{-1} \Pi^{1/2}]^k)^{1/k}, \ \overline{\Sigma}_0=\begin{pmatrix}
\mathbf{I}_c & \bm{0} \\
\bm{0} & \Sigma_0
\end{pmatrix},
\end{equation*}
where $\Sigma_0:=\int_0^1 \Sigma^{b_0}(t) \otimes (\mathbf{B}(t) \mathbf{B}^*(t))dt$ and $\mathbf{M}$ is a diagonal block matrix whose lower $(b_0-b_1+1)c \times (b_0-b_1+1)c$ major part are identity matrices and zeros otherwise. 
\begin{lem}\label{lem_distributionpacf} Suppose $b_0>b_1>b.$  Under Assumptions \ref{phy_generalts}, \ref{assu_pdc},  \ref{assum_local}, \ref{assu_smoothtrend}, \ref{assu_smmothness} and \ref{assu_basis}, when $(4+3\epsilon)/\tau<\frac{1}{2},$ we have 
\begin{equation*}
\frac{n T_{\rho}-g_1}{g_2} \Rightarrow \mathcal{N}(0,2). 
\end{equation*}
\end{lem}
}



%

\subsection{Choices of tuning parameters}\label{sec:choiceparameter}

In this subsection, we discuss how to choose the parameters in both the forecasting and testing procedures. We start with tuning parameter selection for forecasting. As we have seen from (\ref{eq_forecast}) and (\ref{eq_phiform}), we need to choose two important parameters in order to get an accurate prediction: $b$ and $c.$ We use a data-driven procedure proposed in \cite{bishop2013pattern} to choose such parameters. 

For a given integer $l,$ say $l=\lfloor 3 \log_2 n \rfloor,$ we divide the time series into two parts: the training part $\{x_i\}_{i=1}^{n-l}$ and the validation part $\{x_i\}_{i=n-l+1}^n.$  With some preliminary initial pair $(b,c)$, we propose a sequence of candidate pairs  $(b_i, c_j), \ i=1,2,\cdots, u, \ j=1,2,\cdots, v,$ in an appropriate neighbourhood of $(b,c)$ where $u, v$ are some given integers. For each pair of the choices $(b_i, c_j),$  we fit a time-varying AR($b_i$) model (i.e., $b=b_i$ in (\ref{eq_forecast})) with $c_j$ sieve basis expansion using the training data set.  Then using the fitted model, we forecast the time series in the validation part of the time series.  Let $\widehat x_{n-l+1,ij}, \cdots, \widehat x_{n,ij}$ be the forecast of $x_{n-l+1},..., x_n,$ respectively using the parameter pair $(b_i, c_j)$. Then we choose the pair $(b_{i_0},c_{j_0})$ with the minimum sample MSE of forecast, i.e.,
 \begin{equation*} 
({i_0},{j_0}):= \argmin_{((i,j): 1 \leq i \leq u, 1 \leq j \leq v)} \frac{1}{l}\sum_{k=n-l+1}^n (x_k-\widehat x_{k,ij})^2.
 \end{equation*}
 We will discuss how to choose some initial values of $b$ and $c$ in the testing procedure.

Next, we  discuss how to choose the parameters $b,c$ and $m$ for the stability tests. In particular, the parameters $b$ and $c$ selected for the testing also serve as suitable preliminary initial values for the forecasting as we discussed above. Next, we will make use of the PACF to find a suitable $b$ where the sample PACF are uniformly insignificant after lag $b$. 
 
In light of Lemma \ref{lem_distributionpacf}, we can follow the bootstrapping procedure as discussed in the end of Section \ref{sec_bootstrapping} by replacing $\bm{h}_i$ with $\bm{w}_i$ (c.f. (\ref{eq_defnwi})) in (\ref{eq_defnphi}) to perform the test for $nT_{\rho}$. 
For a sufficiently large $b_0$ and a given nominal level $\alpha,$ denote 
\begin{equation}\label{eq_widewidewidewidewidewidehatb}
\widehat{b}=\max_{b_1<b_0}\{b_1<b_0: {\mathbf{H}_0^{\rho, b_0, b_1}} \ \text{is rejected}  \}.
\end{equation}
Then we can use $\widehat{b}$ for our test. Observe that the value of $\widehat{b}$ can be roughly determined by the PACF plot which is helpful in terms of reducing computational complexity. Also note that all PACF after lag $\widehat{b}$ are uniformly statistically insignificant and hence can be treated as 0.   


Then we discuss the choice of $c$ using the criterion of  cross-validation \cite{BH}. The key difference is that our observations are not i.i.d. samples, so we need to slightly modify the procedure. For a given large value $\theta$ such that
\begin{equation}\label{eq_hassum}
\theta \rightarrow \infty, \ \frac{\theta}{n} \rightarrow 0.
\end{equation} 
Denote $\{\widehat{\phi}_{j,c}^{\theta}(t), \ j=1,2,\cdots,b\}$ as the estimation using the data points $\{x_{b+1}, \cdots, x_{n-\theta}\}$ and $c$ basis functions. Denote cross-validation rule as 
\begin{equation*}
\text{CV}(c)=\frac{1}{\theta} \sum_{k=1}^{\theta} \widehat{\epsilon}^2_{k,c}, \ \widehat{\epsilon}_{k,c}=x_{n-k}-\widehat{\phi}_{0,c}^\theta-\sum_{j=1}^b \widehat{\phi}_{j,c}^{\theta}(\frac{n-k}{n})x_{n-k-j}.
\end{equation*}
 Therefore, we choose choose the estimate of $c$ using 
\begin{equation}\label{eq_chosencmethod}
\widehat{c}:=\argmin_{c \leq c_0} \text{CV}(c), \ c_0 \ \text{is a pre-chosen large value}.
\end{equation}

{
Finally, we discuss how to choose $m$ for practical implementation. In \cite{ZZ1}, the author used the minimum volatility (MV) method to choose the window size $m$ for the scalar covariance function. The MV method does not depend on the specific form of the underlying time series dependence structure and hence is robust to misspecification of
the latter structure \cite{politis1999subsampling}. The MV method utilizes the fact that the covariance structure of $\widehat{\Omega}$ becomes stable when the
block size $m$ is in an appropriate range, where $\widehat{\Omega}=E[\Phi\Phi^*|(x_1,\cdots,x_n)]=$ is defined as 
{ 
\begin{equation}\label{eq_widehatomega}
\widehat{\Omega}:=\frac{1}{(n-m-b+1)m} \sum_{i=b+1}^{n-m} \Big[ \Big(\sum_{j=i}^{i+m} \bm{h}_i \Big) \otimes \Big( \mathbf{B}(\frac{i}{n}) \Big) \Big] \times \Big[ \Big(\sum_{j=i}^{i+m} \bm{h}_i \Big) \otimes \Big( \mathbf{B}(\frac{i}{n}) \Big) \Big]^*.
\end{equation}
}    Therefore, it desires to minimize the standard errors of the latter covariance structure in a suitable range of candidate $m$'s.

In detail, for a give large value $m_{n_0}$ and a neighbourhood control parameter $h_0>0,$  we can choose a sequence of window sizes $m_{-h_0+1}<\cdots<m_1< m_2<\cdots<m_{n_0}<\cdots<m_{n_0+h_0}$  and obtain $\widehat{\Omega}_{m_j}$ by replacing $m$ with $m_j$ in (\ref{eq_defnphi}), $j=-h_0+1,2, \cdots, n_0+h_0.$ For each $m_j, j=1,2,\cdots, m_{n_0},$ we calculate the matrix norm error of $\widehat{\Omega}_{m_j}$ in the $h_0$-neighborhood, i.e., 
\begin{equation*}
\mathsf{se}(m_j):=\mathsf{se}(\{ \widehat{\Omega}_{m_{j+k}}\}_{k=-h_0}^{h_0})=\left[\frac{1}{2h_0} \sum_{k=-h_0}^{h_0} \| \overline{\widehat{\Omega}}_{m_j}-\widehat{\Omega}_{m_j+k} \|^2 \right]^{1/2},
\end{equation*}
where $\overline{\widehat{\Omega}}_{m_j}=\sum_{k=-h_0}^{h_0} \widehat{\Omega}_{m_j+k} /(2h_0+1).$
Therefore, we choose the estimate of $m$ using 
\begin{equation*}
\widehat{m}:=\argmin_{m_1 \leq m \leq m_{n_0}} \mathsf{se}(m).
\end{equation*}
Note that in \cite{ZZ1} the author used $h_0=3$ and we also adopt this choice in the current paper. 
}

\section{Simulation studies}\label{sec:simu}
 In this section, we perform extensive Monte Carlo simulations to study the finite-sample accuracy of the nonparametric sieve forecasting method and the finite sample accuracy and power of the stability test and compare them with those of some existing methods in the literature. 
\subsection{Simulation setup}\label{simu_intro}
We consider four different types of non-stationary time series models: two linear time series models, a two-regime model, a Markov switching model and a bilinear model.  
\begin{enumerate}
\item Linear AR model: Consider the following time-varying AR(2) model
\begin{equation*}
x_i=\sum_{j=1}^2 a_j(\frac{i}{n}) x_{i-j}+\epsilon_i, \ \epsilon_i=\Big(0.4+0.4\Big|\sin(2\pi\frac{i}{n})\Big|) \eta_i,  
\end{equation*} 
where $\eta_i, i=1,2,\cdots, n,$ are i.i.d. random variables whose distributions will be specified when we finish introducing the models. It is elementary to see that when $a_j(\frac{i}{n}), j=1,2,$ are constants, the prediction is stable. 
{
\item Linear MA model: Consider the following time-varying MA(2) model
\begin{equation*}
x_i=\sum_{j=1}^2 a_j(\frac{i}{n}) \epsilon_{i-j}+\epsilon_i, \  \epsilon_i=\Big(0.4+0.4\Big|\sin(2\pi\frac{i}{n})\Big|) \eta_i. 
\end{equation*}
}
\item Two-regime model: Consider the following  self-exciting threshold auto-regressive (SETAR) model \cite{FY,HT2011}
\begin{equation*}
x_i=
\begin{cases}
a_1(\frac{i}{n})x_{i-1}+\epsilon_i, \ x_{i-1} \geq 0, \\
a_2(\frac{i}{n}) x_{i-1}+\epsilon_i, \ x_{i-1}<0. 
\end{cases}
\epsilon_i=\Big(0.4+0.4\Big|\sin(2\pi\frac{i}{n})\Big|) \eta_i.
\end{equation*}
It is easy to check that the SETAR model is stable if  { $a_j(\frac{i}{n}), \  j=1,2,$} are constants and bounded by one.
\item Markov two-regime switching model:  Consider the following Markov switching AR(1) model 
\begin{equation*}
x_i=
\begin{cases}
a_1(\frac{i}{n})x_{i-1}+\epsilon_i, \ s_i=0, \\
a_2(\frac{i}{n}) x_{i-1}+\epsilon_i, \ s_i=1.
\end{cases}
\epsilon_i=\Big(0.4+0.4\Big|\sin(2\pi\frac{i}{n})\Big|) \eta_i,
\end{equation*}
where the unobserved state variable $s_i$ is a discrete Markov chain taking values $0$ and $1,$ with transition probabilities $p_{00}=\frac{2}{3}, \ p_{01}=\frac{1}{3}, \ p_{10}=p_{11}=\frac{1}{2}.$  It is easy to check that the above model is stable if the
functions $a_j(\cdot), j=1,2,$ are constants and bounded by one \cite{REQ}. In the simulations, the initial state is chosen to be 1. 
\item Simple bilinear model: Consider the first order bilinear model
\begin{equation*}
x_i=\Big( a_1(\frac{i}{n})\epsilon_{i-1}+a_2(\frac{i}{n}) \Big)x_{i-1}+\epsilon_i , \ \epsilon_i=\Big(0.4+0.4\Big|\sin(2\pi\frac{i}{n})\Big|) \eta_i.
\end{equation*}
It is known from \cite{FY} that when the functions $a_j(\cdot),j=1,2,$ are constants and bounded by one, $x_i$ has an ARMA representation and hence stable.  

\end{enumerate} 

In the simulations below, we record our results based on 1,000 repetitions and for the bootstrapping  procedure described in the end of Section \ref{sec_bootstrapping}, we choose $B=1,000.$ {For the choices of random variables $\eta_i, i=1,2,\cdots,$ we set $\eta_i$ to be student-$t$ distribution with degree of $5$, i.e., $t$(5) for models 1-2 and standard normal random variables for models 3-5. }

\subsection{Prediction of locally stationary time series} \label{sec:predictioncompare}

In this section, we study the prediction accuracy of our adaptive sieve forecast (\ref{eq_forecast}) by comparing it with some state-of-the-art methods. Specifically, we compare with the Tapered Yule-Walker estimate (TTVAR) in \cite{RSP}, the non-decimated wavelet estimate (LSW) in \cite{FBS}, the model switching method (SNSTS) in \cite{KPF},  the best linear prediction using the previous samples (SBLP) \footnote{The prediction is based on the stationary assumption and an ARIMA model.}, the best linear prediction using $b$ recent samples (PBLP) and our adaptive sieve forecast (\ref{eq_forecast}). We implement TTVAR with constant taper function $g \equiv 1$ and the bandwidth is selected according to \cite[Corollary 4.2]{RSP}. For the wavelet method, we use the matlab codes from the first author's website (see \url{http://stats.lse.ac.uk/fryzlewicz/flsw/flsw.html}) and for the model switching method, we use the R package \texttt{forecastSNSTS}. For our sieve method, we use the orthogonal wavelets (\ref{eq_meyerorthogonal}) with Daubechies-9 wavelet and the data-driven approach described in Section \ref{sec:choiceparameter} to choose $b$ and $c.$ 

In Table \ref{table_compare_prediction}, we record the mean square error over 1,000 simulations  for one-step ahead prediction of the models 1-5 in Section \ref{simu_intro}. {Specifically,  we use
\begin{equation*}
a_1(\frac{i}{n}) \equiv 0.4, \ a_2(\frac{i}{n})=0.2+ \delta \sin(2 \pi \frac{i}{n}), 
\end{equation*}
where  $\delta=0.35$ for models 1-2 and $\delta=0.5$ for models 3-5.} It can be seen that our sieve method outperforms the other methods in literature for five models in both sample sizes $n=256$ and $n=512$. The forecasting accuracy improvement is more significant for non-AR type models such as the MA and bilinear models.

\begin{table}[ht]
\begin{center}
\setlength\arrayrulewidth{1pt}
\renewcommand{\arraystretch}{1.5}
{\fontsize{10}{10}\selectfont 
\begin{tabular}{|c|ccccccc|lllllll|}
\hline
Model & \multicolumn{1}{c|}{TTVAR} & \multicolumn{1}{c|}{LSW} & \multicolumn{1}{c|}{SNSTS} & \multicolumn{1}{c|}{SBLP} & \multicolumn{1}{c|}{PBLP} & \multicolumn{1}{c|}{Sieve} & Improvement  \\ \hline
      & \multicolumn{7}{c|}{$n$=256}                                                                                                                                                                                                                                                                                         \\ \hline
1     &   0.24                                  &     0.21                     &     0.45                      &                     0.284      &                          0.24 & {\bf 0.189} & 10 $\%$  \\
2     &                 0.28                    &      0.27                   &       0.28                   &                  0.273         &                 0.283        &  {\bf 0.22} & 18.5 $\%$  \\ 
3     &     0.21                                 &                          0.185 & 0.198                         &                 0.241          &                          0.194 &  {\bf 0.178} & 3.8 $\%$  \\
4     &        0.207                              &                          0.195 & 0.2                         &                      0.247     &                   0.199       & {\bf 0.187} & 4.1 $\%$     \\
5     &                 0.22                    &      0.22                   &       0.24                   &                  0.246         &                 0.273        &  {\bf 0.176} & 20 $\%$  \\ 
\hline
      & \multicolumn{7}{c|}{$n$=512}                                                                                                                                                                                                                                                                                         \\ \hline
1     &                        0.21             &                          0.2 & 0.2                          &                   0.233      &                         0.209 &  {\bf 0.181} & 9.5 $\%$\\
2     &                 0.26                   &      0.26                   &       0.264                   &                  0.276         &                 0.283        &  {\bf 0.196} & 24.62 $\%$ \\ 
3     &    0.207                                &                          0.183 & 0.192                         &                0.213           &                      0.194   & {\bf 0.18} & 1.7 $\%$   \\
4     &      0.205                               &                          0.175 & 0.188                           &                   0.211        &         0.181              & {\bf 0.17 } & 2.86 $\%$   \\
5     &       0.23                               &                          0.21 & 0.24                          &       0.23                    &                        0.22 &  {\bf  0.183} & 12.86 $\%$    \\
 \hline
\end{tabular}
}
\end{center}
\caption{Comparison of prediction accuracy for models 1-5 using different methods.  We highlight the smallest mean square errors and record the percentage of improvement of our method compared with the next best method.  }
\label{table_compare_prediction}
\end{table}

\subsection{Accuracy and power of the stability test}\label{sec_poer}
In this section, we study the performance of the proposed test (\ref{eq_testnonzero1}). First, we study the finite sample accuracy of our test under correlation stationarity when  
\begin{equation}\label{eq_testingcases}
a_1(\frac{i}{n}) = a_2(\frac{i}{n}) \equiv 0.4.
\end{equation}
Observe that the simulated time series are not covariance stationary as the marginal variances change smoothly over time. We choose the values of $b,c$ and $m$ according to the methods described in Section \ref{sec:choiceparameter}. It can be seen from Table \ref{table_typei} that our bootstrap testing procedure behaves reasonably accurate for all three types of sieve basis functions even for a smaller sample size $n=256.$ 

Second, we study the power of the tests and report the results in Table \ref{table_power} when the underlying time series is not correlation stationary. Specifically, we use
\begin{equation}\label{eq_testingcasesalternative}
a_1(\frac{i}{n}) \equiv 0.4, \ a_2(\frac{i}{n})=0.2+ \delta \sin(2 \pi \frac{i}{n}), 
\end{equation}
for the models 1-5 in Section \ref{simu_intro}. It can be seen that the simulated powers are reasonably good even for smaller $\delta$ and the sample size, and the results will be improved when $\delta$ and the sample size increase. Additionally, the power performances of the three types of sieve basis functions are similar in general.
\begin{table}[ht]
\begin{center}
\setlength\arrayrulewidth{1pt}
\renewcommand{\arraystretch}{1.5}
{\fontsize{10}{10}\selectfont 
\begin{tabular}{|c|ccccc|ccccc|}
\hline
      & \multicolumn{5}{c|}{$\alpha=0.1$}                                                                                                                       & \multicolumn{5}{c|}{$\alpha=0.05$}                                                                                                                        \\ \hline
Basis/Model & \multicolumn{1}{c|}{1} & \multicolumn{1}{c|}{2} & \multicolumn{1}{c|}{3} & \multicolumn{1}{c|}{4} & \multicolumn{1}{c|}{5}  & \multicolumn{1}{c|}{1} & \multicolumn{1}{c|}{2} & \multicolumn{1}{c|}{3} & \multicolumn{1}{c|}{4} & \multicolumn{1}{c|}{5} \\ 
\hline
     & \multicolumn{10}{c|}{$n$=256}                                                                                                                                                                                                                                                                                          \\
   \hline
Fourier     &          0.132  & 0.11                          &                          0.12 &         0.13                  &                          0.11 &                    0.067   & 0.07    & 0.06   &                                     0.04 &       0.06                    \\
Legendre    &     0.091       & 0.136                          &                          0.13 &   0.12                        &           0.13                &              0.06    & 0.059         &  0.041  &                                     0.07 &                            0.07 \\
Daubechies-9    &  0.132 & 0.12   & 0.11                        &      0.133             &         0.132                 & 0.063                 &0.067 &  0.059  &          0.068                            &                 0.065   \\
\hline
      & \multicolumn{10}{c|}{$n$=512}                                                                                                                                                                                                                                                                                         \\
       \hline
Fourier     &            0.09      & 0.13                 &                          0.11 &      0.13                     &  0.127  &                         0.05 & 0.06 & 0.067   &    0.068                                  &              0.069                 \\
Legendre     &    0.09    & 0.094    &                  0.092                                &          0.12                 &                       0.118   & 0.04   & 0.058                       & 0.07    &        0.043                              &                            0.057 \\
Daubechies-9     &  0.091  & 0.11    &      0.098                   &                         0.11   & 0.118                          &          0.048               &  0.052 & 0.054 &                                 0.053  &                  0.054           \\
 \hline
\end{tabular}
}
\end{center}
\caption{Simulated type I errors using the setup (\ref{eq_testingcases}). 
}
\label{table_typei}
\end{table}

\begin{table}[ht]
\begin{center}
\setlength\arrayrulewidth{1pt}
\renewcommand{\arraystretch}{1.5}
{\fontsize{10}{10}\selectfont 
\begin{tabular}{|c|ccccc|ccccc|}
\hline
      & \multicolumn{5}{c|}{$\delta=0.2/0.5$}                                                                                                                       & \multicolumn{5}{c|}{$\delta=0.35/0.7$}                                                                                                                        \\ \hline
Basis/Model & \multicolumn{1}{c|}{1} & \multicolumn{1}{c|}{2} &  \multicolumn{1}{c|}{3} & \multicolumn{1}{c|}{4} & \multicolumn{1}{c|}{5}  & \multicolumn{1}{c|}{1} & \multicolumn{1}{c|}{2} &  \multicolumn{1}{c|}{3} & \multicolumn{1}{c|}{4} & \multicolumn{1}{c|}{5} \\ 
\hline
      & \multicolumn{10}{c|}{$n$=256}                                                                                                                                                                                                                                                                                          \\
   \hline
Fourier     & 0.84     & 0.86                              &                          0.84 &                  0.837          &                      0.94 &                        0.97 &  0.97 & 0.96  & 0.99                                      &              0.98              \\
Legendre     &         0.8   & 0.806                          &                          0.81 &               0.84           & 0.83                          & 0.97    & 0.968                       &  0.95  &                                     0.97 &        0.91                     \\
Daubechies-9    &     0.81 & 0.81 & 0.86                        &                       0.81   & 0.81                          &           0.97       & 0.96 & 0.983  & 0.98                                    &        0.98           \\
\hline
      & \multicolumn{10}{c|}{$n$=512}                                                                                                                                                                                                                                                                                         \\
       \hline
Fourier     &  0.91         & 0.9                       &                           0.96&          0.9              &  0.93                      &           0.96               &  0.97  &  0.973 &    0.98                                &                  0.97             \\
Legendre    &     0.9         & 0.91                       &                          0.92&                        0.893   & 0.91                           &           0.94    & 0.95          &0.98  &        0.97                             &        0.96   \\
Daubechies-9     &             0.87  & 0.88                       &                      0.93     &  0.91                          &                       0.91    &  0.96                         & 0.99 & 0.97  &        0.97                              &  0.96                             \\
 \hline
\end{tabular}
}
\end{center}
\caption{Simulated power under the setup (\ref{eq_testingcasesalternative}) using nominal level $0.1.$ For models 1-2, we consider the cases $\delta=0.2 $ and $\delta=0.35$, whereas for models 3-5, we use $\delta=0.5$ and $\delta=0.7.$ 
}
\label{table_power}
\end{table}

\subsection{Comparison with tests for covariance stationarity}
 
In this subsection, we compare our method with some existing works on the tests of covariance stationarity: the $\mathcal{L}^2$ distance method in \cite{DPV}, the discrete Fourier transform method in \cite{DR} and the Haar wavelet periodogram method in \cite{GN}. The first method is easy to implement; for the second method, we use the codes from the author's website (see \url{https://www.stat.tamu.edu/~suhasini/test_papers/DFT_covariance_lagl.R}); and for the third method, we employ the R package
 \texttt{locits}, which is contributed by the author.  For the purpose of comparison of accuracy, besides the five models considered in Section \ref{simu_intro}, 
 we also consider the following two strictly stationary time series. 
\begin{enumerate}
\item[6.] Linear time series: stationary ARMA(1,1) process. We consider the following  process
\begin{equation*}
x_i-0.5x_{i-1}=\epsilon_i+0.5 \epsilon_{i-1}, 
\end{equation*}
where $\epsilon_i$ are i.i.d. $\mathcal{N}(0,1)$ random variables.
\item[7.] Nonlinear time series: stationary SETAR. We consider the following model 
\begin{equation*}
x_i=
\begin{cases} 
0.4 x_{i-1}+\epsilon_i, & x_{i-1} \geq 0, \\
0.5 x_{i-1}+\epsilon_i, & x_{i-1}<0,
\end{cases}
\end{equation*}  
where $\epsilon_i$ are i.i.d. $\mathcal{N}(0,1)$ random variables. 
\end{enumerate}  
Furthermore, for the comparison of power, we  consider the following two non-stationary time series whose errors have constant variances.
\begin{enumerate}
\item[$6^{\#}$.] Non-stationary linear time series. We consider the following process
\begin{equation*}
x_i=\delta \sin (4\pi\frac{i}{n}) x_{i-1}+\epsilon_i,
\end{equation*} 
where $\epsilon_i, i=1,2,\cdots,n,$ are i.i.d. standard normal random variables. 
\item[$7^{\#}$.] Non-stationary nonlinear time series.  We consider the following process
\begin{equation*}
x_i=
\begin{cases}
\delta \sin(4 \pi \frac{i}{n}) x_{i-1}+\epsilon_i,& 1 \leq i \leq 0.75n, \\
0.4 x_{i-1}+\epsilon_i, & 0.75n<i \leq n \ \text{and} \ x_{i-1} \geq 0, \\
0.3 x_{i-1}+\epsilon_i,  & 0.75n<i \leq n \ \text{and} \ x_{i-1} < 0, \\
\end{cases}
\end{equation*}
where $\epsilon_i, i=1,2,\cdots,n,$ are i.i.d. standard normal random variables. 
\end{enumerate}

In the simulations below, we report the type I error rates under the nominal levels $0.05$ and $0.1$ for the above seven models in Table \ref{table_compare_typeone}, where for models 1-5 we use the setup (\ref{eq_testingcases}). Our simulation results are based on 1,000 repetitions, where $\mathcal{L}^2$  refers to the $\mathcal{L}^2$ distance method, DFT 1-3 refer to the discrete Fourier method using the imagery  part, real part, both imagery and real parts of the discrete Fourier transform method, HWT is the Haar wavelet periodogram method and RB is our robust bootstrap method using orthogonal wavelets constructed by (\ref{eq_meyerorthogonal}) with Daubechies-9 wavelet. 

Since HWT needs the length to be a power of two, we set the length of time series to be 256 and 512. For the $\mathcal{L}^2$ test, we use $M=8, N=32$ for $n=256$ and $M=8, N=64$ for $n=512.$ For the DFT, we choose the lag to be $0$ as suggested by the authors in \cite{DR}. Since the mean of model 5 is non-zero, we test its first order difference for the methods mentioned above. Moreover, we report the power of the above tests under certain alternatives in Table \ref{table_comprare_power} for models $6^\#-7^\#$ and models 1-5 under the setup (\ref{eq_testingcasesalternative}).

{

}

\begin{table}[ht]
\begin{center}
\setlength\arrayrulewidth{1pt}
\renewcommand{\arraystretch}{1.5}
\newcolumntype{L}[1]{>{\raggedright\arraybackslash}p{#1}}
{\fontsize{10}{10}\selectfont 
\begin{tabular}{|p{0.9cm}|p{0.75cm}p{0.7cm}p{0.7cm}p{0.7cm}p{0.7cm}p{0.8cm}|p{0.75cm}p{0.7cm}p{0.7cm}p{0.7cm}p{0.7cm}p{0.8cm}|}
\hline
      & \multicolumn{6}{c|}{$\alpha=0.1$}                                                                                                                       & \multicolumn{6}{c|}{$\alpha=0.05$}                                                                                                                        \\ \hline
Model & \multicolumn{1}{l|}{$\mathcal{L}^2$} & \multicolumn{1}{l|}{DFT1} & \multicolumn{1}{l|}{DFT2} & \multicolumn{1}{l|}{DFT3} & \multicolumn{1}{l|}{HWT} & RB & \multicolumn{1}{l|}{$\mathcal{L}^2$} & \multicolumn{1}{l|}{DFT1} & \multicolumn{1}{l|}{DFT2} & \multicolumn{1}{l|}{DFT3} & \multicolumn{1}{l|}{HWT} & RB \\ \hline
      & \multicolumn{12}{c|}{$n$=256}                                                                                                                                                                                                                                                                                         \\ \hline
1     &          0.08                            &                          0.148 &         0.057                  &                        0.13   &                          0.18 & 0.132   &                                     0.024 &                          0.067 &     0.017                      &         0.063                  &                          0.083& 0.063   \\
2     &      0.081                               &                          0.097 &    0.068                       &                          0.12 &       0.085                    &    0.12  &           0.038                           &                          0.04 &     0.07                     &          0.057                &               0.028         &  0.067   \\
3     &     0.171                                 &                          0.183 &   0.04                        &                          0.137 &                          0.227&  0.11  &                                     0.087 &                          0.103 &    0.011                      &           0.033               &                          0.093 &   0.059 \\
4     &     0.2                                 &                        0.163   &      0.05                    &                          0.12 &                   0.176       & 0.133   &                                     0.077 &                    0.087       &  0.013 &              0.034             &               0.113           & 0.068    \\
5     &   0.46                                  &                          0.293 &         0.077                  &                          0.19 &                          0.153 & 0.132  &                                     0.29 &                       0.21    &        0.03                   &             0.14              &       0.12 & 0.065   \\ 
6     &         0.11                          &                          0.105 &           0.096                &                          0.09 &                 0.087          &  0.088  &   0.047                                   &       0.053                    &        0.053                 &             0.039              &               0.052          &  0.057   \\
7     &      0.051                               &                          0.097 &    0.08                       &                          0.092 &       0.085                    &    0.127  &           0.018                           &                          0.04 &     0.06                     &          0.047                &               0.038         &  0.061   \\
\hline
      & \multicolumn{12}{c|}{$n$=512}                                                                                                                                                                                                                                                                                         \\ \hline
1     &   0.087                                   &                          0.127 &       0.03                    &                          0.13 &     0.237                     &  0.091  &                                     0.023 &      0.1                     &                          0.02 &                 0.043          &                         0.137 &  0.048  \\
2     &      0.051                             &                          0.096 &    0.085                       &                          0.093 &       0.075                    &    0.11  &           0.026                           &                          0.036 &     0.067                     &          0.044                &               0.033         &  0.052   \\
3     &            0.26                          &                          0.16 &      0.04                     &                          0.117 &       0.243                  & 0.098   &                                    0.127  &     0.1                      &                          0.007 &      0.037                     &                         0.14 & 0.054   \\
4     &  0.287                                    &                          0.167 &             0.027              &                          0.09 &         0.247                 & 0.11   &                                     0.177 &        0.103                   &                          0.013 &           0.073                &                         0.163 & 0.053    \\
5     &     0.64                                 &                          0.303 &        0.087                   &                          0.283 &          0.35                &   0.118 &                                     0.413 &      0.26                     &                          0.063 &      0.167                     &                         0.23 & 0.054   \\
6    &    0.11                                 &                          0.093 &        0.084                   &                          0.088 &      0.088                     &   0.092 &                                     0.035 &         0.046                 &   0.047      &                        0.048 &    0.053    &  0.048  \\
7     &      0.051                               & 0.087                           &          0.113             &      0.083                     &       0.093                    &  0.092  &                                     0.013 &  0.037                         &                          0.047 &       0.043                    &                        0.04 & 0.051   \\
 \hline
\end{tabular}
}
\end{center}
\caption{Comparison of accuracy for models 1-7 using different methods. 
}
\label{table_compare_typeone}
\end{table}

We first discuss the results for models 6-7 since they are not only correlation stationary but also covariance stationary. 
It can be seen from Table \ref{table_compare_typeone} that all the methods including our RB achieve a reasonable level of accuracy for the linear model 6. However, for the nonlinear model 7, we conclude from Table \ref{table_compare_typeone} that the $\mathcal{L}^2$ method loses its accuracy due to the fact that the latter test is designed only for linear models. Regarding the power in Table \ref{table_comprare_power}, for model $6^\#$, when the sample size and $\delta$ are smaller, only our RB method is powerful. When $n=256$ and $\delta$ increases, the $\mathcal{L}^2$ test starts to become powerful. Further, when both the sample size and $\delta$ increase, the HWT method becomes powerful. Similar discussion holds for model $7^\#.$ Therefore, we conclude that, when the marginal variance of the time series stays constant, even though other methods in the literature may be accurate for the purpose of testing for correlation stationarity, our RB method is generally more powerful when the sample size is moderate and/or the departure from stationary is small.

\begin{table}[ht]
\begin{center}
\setlength\arrayrulewidth{1pt}
\renewcommand{\arraystretch}{1.5}
\newcolumntype{L}[1]{>{\raggedright\arraybackslash}p{#1}}
{\fontsize{10}{10}\selectfont 
\begin{tabular}{|p{0.9cm}|p{0.75cm}p{0.7cm}p{0.7cm}p{0.7cm}p{0.7cm}p{0.8cm}|p{0.75cm}p{0.7cm}p{0.7cm}p{0.7cm}p{0.7cm}p{0.8cm}|}
\hline
      & \multicolumn{6}{c|}{$\delta=0.2/0.5$}                                                                                                                       & \multicolumn{6}{c|}{$\delta=0.35/0.7$}                                                                                                                        \\ \hline
Model & \multicolumn{1}{l|}{$\mathcal{L}^2$} & \multicolumn{1}{l|}{DFT1} & \multicolumn{1}{l|}{DFT2} & \multicolumn{1}{l|}{DFT3} & \multicolumn{1}{l|}{HWT} & RB & \multicolumn{1}{l|}{$\mathcal{L}^2$} & \multicolumn{1}{l|}{DFT1} & \multicolumn{1}{l|}{DFT2} & \multicolumn{1}{l|}{DFT3} & \multicolumn{1}{l|}{HWT} & RB \\ \hline
      & \multicolumn{12}{c|}{n=256}                                                                                                                                                                                                                                                                                         \\ \hline
1     &      0.263                                &                          0.14 &        0.03                   &                          0.07 &                  0.3         & 0.81  &                             0.503       &                          0.113 &               0.053            &       0.089                   &                          0.4 & 0.97   \\
2     &      0.183                               &                          0.497 &    0.08                       &                          0.092 &       0.585                    &    0.81  &           0.68                           &                          0.14 &     0.06                     &          0.047                &               0.38         &  0.96   \\
3     &         0.44                             &                          0.153 &           0.04                &                          0.16 &                   0.393       &  0.86  &                          0.7           &                          0.14 &               0.05            &        0.09                   &                         0.64 &    0.983 \\
4     &    0.603                                  &                         0.16 &         0.04                  &                          0.203 &            0.44              &  0.81 &                             0.86        &                 0.2          &          0.07                 &           0.12               &       0.647                  &  0.98  \\
5     &      0.92                                &                          0.243 &             0.143              &                          0.24 &        0.57                 &    0.81 &                           0.997           &                          0.347 &            0.193               &       0.397                    &                         0.797 &  0.98   \\
$6^{\#}$     &      0.697                                &                          0.12 &             0.093             &                          0.11 &        0.327                  & 0.86   &                                    0.923 &  0.16                         &            0.15               &       0.15                    &                         0.563 & 0.94     \\
$7^{\#}$     &      0.463                                &                          0.137 &             0.107             &                          0.133 &        0.273                 &  0.85   &                           0.81          &                          0.193 &            0.203             &       0.223                   &                         0.483 &  0.96   \\
 \hline
      & \multicolumn{12}{c|}{n=512}                                                                                                                                                                                                                                                                                         \\ \hline
1     &     0.477                                 &                          0.173 &        0.04                   &                          0.08 &            0.52              &  0.87  &                                     0.857 &         0.137                  &                          0.03 &         0.1                  &                         0.75 &  0.96  \\
2     &      0.51                               &                          0.297 &    0.082                       &                          0.092 &       0.385                    &    0.88  &           0.918                           &                          0.24 &     0.06                     &          0.047                &               0.838         &  0.99   \\
3     &        0.657                              &                          0.24 &          0.05                 &                          0.083 &                  0.61        & 0.93   &                                     0.96 &       0.17                    &                          0.24 &      0.113                     &                         0.95 &  0.97  \\
4     &   0.84                                   &                          0.23 &                    0.043       &                          0.143 &               0.773           &0.91    &                                     0.987 &      0.293                     &                          0.053 &             0.19              &                         0.97 &0.97    \\
5     &          0.963                           &                          0.297 &  0.127                         &                          0.263 &  0.87                         &  0.91  & 0.983                                      &          0.523                 &     0.24                      &       0.478 &    0.994 & 0.96  \\
$6^{\#}$     & 0.847                                 &                          0.147 &             0.087          &                          0.103 &          0.67              &  0.88   &                                     0.95 &               0.13            &                      0.09 &                         0.133 &                         0.963 &  0.95   \\
$7^{\#}$     &    0.69                               &                          0.14 &         0.13                  &                          0.217 &       0.383                &  0.91 &                                     0.953 &            0.3             &                          0.313 &           0.383                &                         0.823 &  0.943    \\ 
 \hline
\end{tabular}
}
\end{center}
\caption{ Comparison of power at nominal level 0.1 using different methods. 
}
\label{table_comprare_power}
\end{table}

Next, we study models 1-5 from Section \ref{simu_intro}. None of these models is covariance stationary. For the type I error rates, we use the setting (\ref{eq_testingcases}) where  all the models are correlation stationary. For the power, we use the setup (\ref{eq_testingcasesalternative}). We find that DFT-3 is accurate for models 1-4 but with low power across all the models. Moreover, the $\mathcal{L}^2$ test seems to have a high power for models 3-5. But this is at the cost of blown-up type I error rates. This inaccuracy increases when the sample size becomes larger. For the HWT method, even though its power becomes larger when the sample size and $\delta$ increase, it also loses its accuracy. Finally, for all the models 1-5, our RB method both obtain high accuracy and power. In summary, most of the existing tests for covariance stationarity are not suitable for the purpose of testing for correlation stationarity. From our  simulation studies, our robust bootstrap method performs well for the latter purpose.

\section{Empirical illustrations}\label{sec:realdata}

\subsection{Global temperature data} 
In this first application, we study the global temperature time series using the dataset Global component of Climate at a Glance (GCAG). As explained on the website of National Oceanic and Atmospheric Administration (NOAA) \footnote{\url{https://www.ncdc.noaa.gov/cag/global/data-info}}, GCAG comes from the Global Historical Climatology Network-Monthly (GHCN-M) Data Set and International Comprehensive Ocean-Atmosphere Data Set (ICOADS), which have data from 1880 to the present. These two datasets are blended into a single product to produce the combined global land and ocean temperature anomalies.
The term \textit{temperature anomaly} means a departure from a reference value or long-term average. 

 The available time series of global-scale temperature anomalies are calculated with respect to the 20th century average \cite{SRPL}, while the mapping tool displays global-scale temperature anomalies with respect to the 1981-2016 based period. ( see \url{https://datahub.io/core/global-temp#readme} for the dataset). This dataset is a global-scale climate diagnostic tool and provides a big picture overview of average global temperatures compared to a reference value.

\begin{figure}[H]
\centering
\includegraphics[width=10cm,height=4cm]{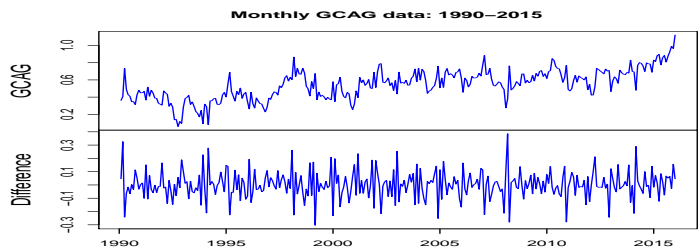}
\caption{Monthly (1990-2015) global temperature using data set GCAG. }
\label{annal_timeseries}
\end{figure}
We study the monthly time series from this dataset for the time period  1990-2015 (Figure \ref{annal_timeseries}). 
As indicated from the above figure, the  global temperature has an increasing trend and we consider its first order difference.

\begin{figure}[ht]
\centering
\includegraphics[width=13cm,height=7cm]{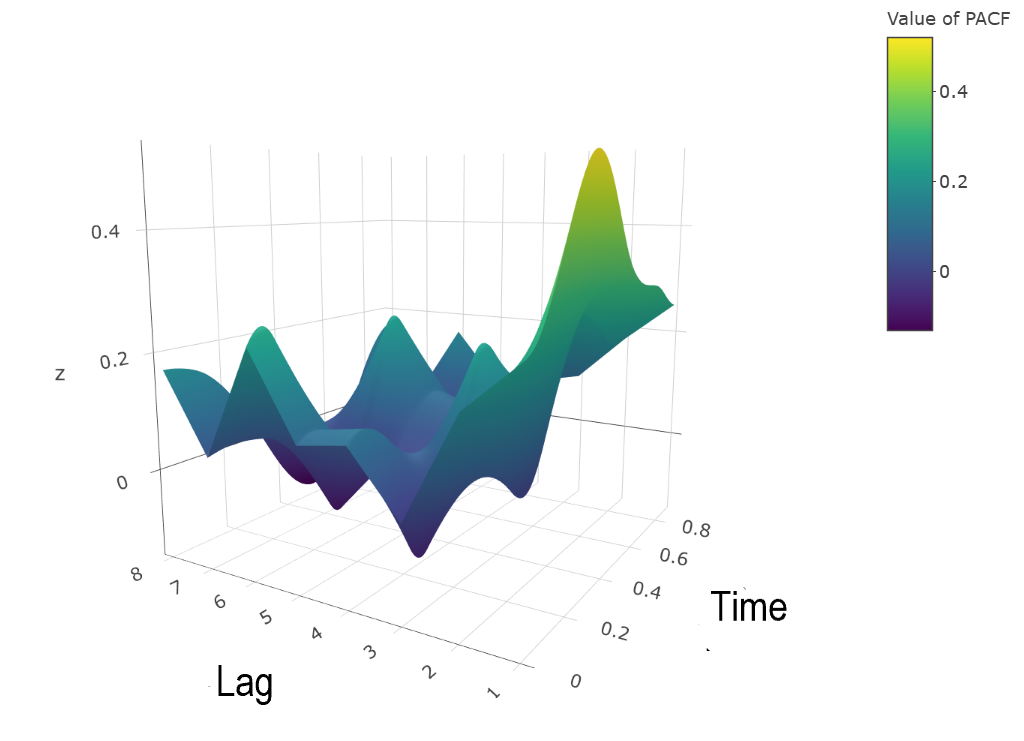}
\caption{PACF plots for the first order difference of the Monthly (1990--2015)  global temperature dataset.  It can be seen that the first few lags of the PACF are larger than the other lags at any time. Here we use the orthogonal wavelets (\ref{eq_meyerorthogonal}) with Daubechies-9 wavelet with $J_n=3.$}
\label{pacf_exam1}
\end{figure}

{
Then we apply the methodologies described in Sections \ref{sec:predict} and \ref{sec:test} to study the time series.  We first employ the methods from Section \ref{sec:test} to test whether this time series is correlation stationary.  There are three parameters, $b, c$ and $m$ needed to be properly chosen.  Especially, we make use of the PACF defined in  Definition \ref{defn_pacf} to choose $b$ (c.f. Section \ref{sec_estimationpacf}). In Figure \ref{pacf_exam1}, we make a 3-D plot of the PACF for the time series (first order difference) between 1990 and 2015. It can be seen that the temporal dependence of this time series decays uniformly in time. For the sieve basis functions, we use the orthogonal wavelets constructed by (\ref{eq_meyerorthogonal}) with Daubechies-9 wavelet. The tuning parameters $b,c$ and $m$ are chosen according to Section \ref{sec:choiceparameter} which yields $b=5$, $J_n=3$ (i.e. $c=8$) and $m=10$. We apply the bootstrap procedure described in the end of Section \ref{sec_bootstrapping} to test the stationarity of the correlation and find that the $p$-value is $0.026$. We hence conclude that the prediction is unstable during this time period. 

}

{
Next, we use time series 1990-2015  as the training dataset to study the prediction performance over the year 2016, i.e., we do a one-step ahead prediction for each month of 2016 and take the average of the square error.  We use the data-driven approach as described in Section \ref{sec:choiceparameter} to choose $b=6$ and $J_n=3.$ The MSE of our prediction 
is  $0.381.$  We compare this result with the methods mentioned in  Section \ref{sec:predictioncompare} and record the results in Table \ref{table:tem}. We find that our prediction performs better than the other methods. Especially, we get a $16.6 \%$ improvement compared to simply fitting a stationary model using all the time series from 1990 to 2015 (SBLP).    

\begin{table}[H]
\center{
\begin{threeparttable}
\begin{tabular}{cccccc}
\hline
Method & {\bf Sieve} & TTVAR & LSW & SNSTS & SBLP \\ \hline
MSE  &  {\bf 0.381}       &    0.3913        &   0.3851            &      0.3969        &  0.45706             \\ 
\hline
\end{tabular}
\end{threeparttable}}
\caption{Comparison of prediction accuracy for GCAG. We refer to Section \ref{sec:predictioncompare} for the short-hand notation of the names of the methods. 
}\label{table:tem}
\end{table}
}

\subsection{Stock return data of Nigerian Breweries}
In the second application, we study the stock return data of the Nigerian Breweries (NB) Plc. This stock is traded in Nigerian Stock Exchange (NSE). Regarding on market
returns, the brewery industry in Nigerian has done pretty well in
outperforming Brazil, Russia, India, and China (BRIC)
and emerging markets by a wide margin over the past
ten years. Nigerian Breweries Plc is the  largest brewing company in Nigeria, which  mainly serves the Nigerian market and also exports to other parts of West Africa. 
The data can be found on the website of morningstar  
(see  \url{http://performance.morningstar.com/stock/performance-return.action?p=price_history_page&t=NIBR&region=nga&culture=en-US}). We are interested in predicting the volatility of the NB stock. We shall study the absolute value of the  daily log-return of the stock for the latter purpose.

\begin{figure}[ht]
\centering
\includegraphics[width=10cm,height=4cm]{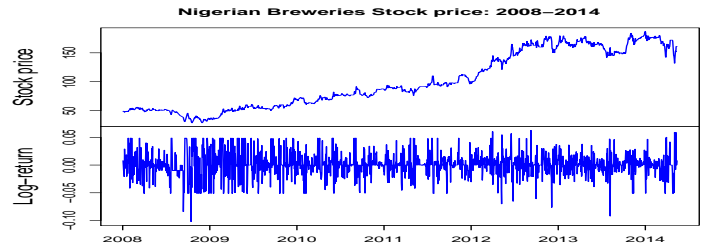}
\caption{Nigerian Breweries stock return from 2008 to 2014. }
\label{0814_timeseries}
\end{figure}
{
We perform our analysis on the time period 2008-2014 (Figure \ref{0814_timeseries}). This time series  contains  the data of the 2008 global financial crisis and its post period. As said in the report from the Heritage Foundation \cite{Report11}, "the economy is experiencing the slowest recovery in 70 
years" and even till 2014, the economy does not fully recover.    
}

\begin{figure}[h]
\centering
\includegraphics[width=13cm,height=7cm]{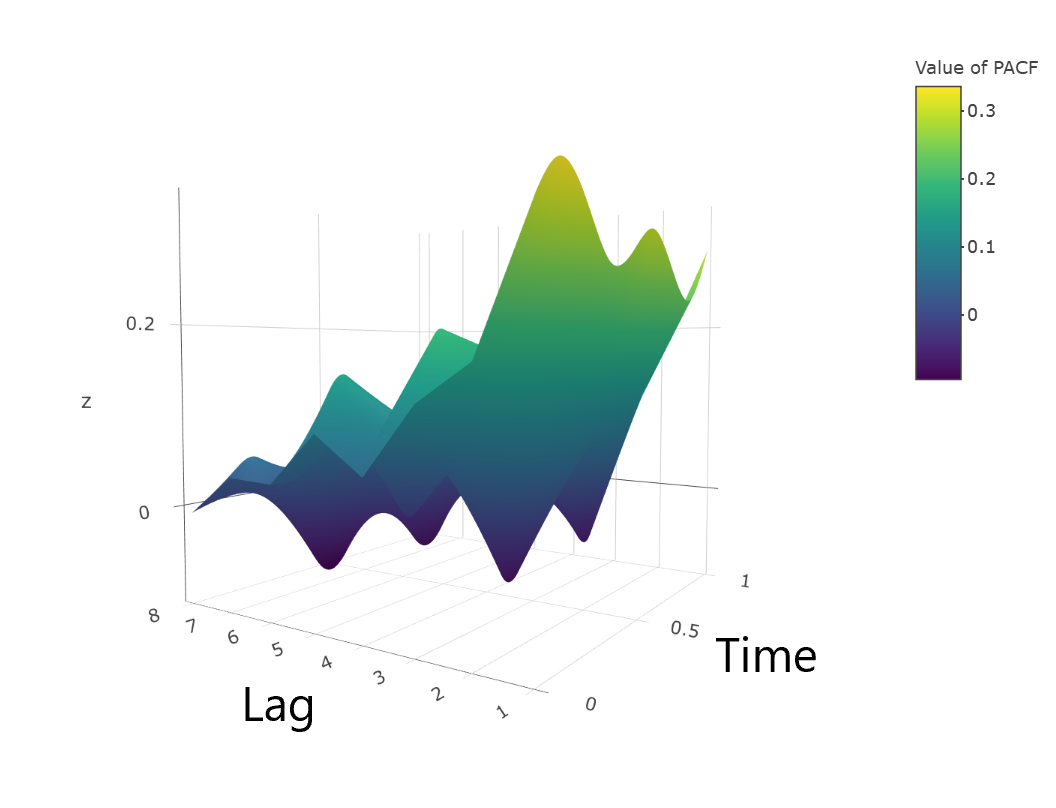}
\caption{PACF plots for the absolute values of the first order differences of the logarithms  of the Nigerian Breweries stock return datatset (2008-2014).  
}
\label{pacf_exam2}
\end{figure}

{Then we apply the methodologies described in Sections \ref{sec:predict} and \ref{sec:test} for the absolute values of log-return time series. It is clear that we need to fit a mean curve for this model. In Figure \ref{pacf_exam2}, we make a 3-D plot of the PACF for the time series between 2008 and 2014. It can be seen that the temporal dependence of this  time series decays uniformly in time. Then we test the stability of the best linear prediction as described in Section \ref{sec:test}. For the sieve basis functions, we use the orthogonal wavelets constructed by (\ref{eq_meyerorthogonal}) with Daubechies-9 wavelet. We choose the parameters $b,c$ and $m$ based on the discussion of Section \ref{sec:choiceparameter} which yields $b=7,$ $J_n=5$ (i.e., $c=32$) and $m=18$. We apply the bootstrap procedure described in the end of Section \ref{sec_bootstrapping} and find that the $p$-value is $0.0825$. We hence conclude that the prediction is likely to be unstable during this time period. }

{
Next, we use the time series 2008-2014 as the training dataset to study the prediction performance over the first month of 2015. We employ the data-driven approach from Section \ref{sec:choiceparameter} to choose $b=5$ and $J_n=3.$
The MSE is  $0.194.$  We compare this result with the methods mentioned in  Section \ref{sec:predictioncompare} and record the results in Table \ref{table:fin}. We find that our prediction performs better than the other methods. Especially,  we get a $24.5 \%$ improvement compared to simply fitting a stationary model using all the time series from 2008 to 2014.    
\begin{table}[H]
\center{
\begin{threeparttable}
\begin{tabular}{cccccc}
\hline
Method & {\bf Sieve} & TTVAR & LSW & SNSTS & SBLP \\ \hline
MSE  &  {\bf 0.194}       &    0.198        &   0.198            &      0.202        &  0.257             \\ 
\hline
\end{tabular}
\end{threeparttable}}
\caption{Comparison of prediction accuracy for GCAG. We refer to Section \ref{sec:predictioncompare} for the short-hand notations of the names of the methods. For SBLP,  we use all the time series from 2008 to 2014 to fit a stationary ARMA model. }\label{table:fin}
\end{table}

Finally, we further study the absolute value of the stock return from 2012 to 2014. We apply our bootstrap procedure described in the end of Section \ref{sec_bootstrapping} to test correlation stationarity of the time series. We select $b=6,$ $J_n=4$ (i.e., $c=16$) and $m=12$ for this sub-series and find that the $p$-value is $0.599$. We hence conclude that the prediction is stable during this time period. Therefore, we fit a best stationary ARMA model to this sub-series and do the prediction. This yields an MSE of 0.195. We find that our sieve method is still slightly better. The result from this sub-series shows an interesting trade-off between forecasting using a shorter and stationary time series and a longer but non-stationary series. The forecast model of the shorter stationary period can be estimated at a faster rate but at the expense of a smaller sample size.  The opposite happens to the longer non-stationary period. Note that 2012-2014 is nearly half as long as 2008-2014 and hence the length of the shorter stationary period is substantial compared to that of the long period. In this case we see that the forecasting accuracy using the short period is comparable to that of the longer period. In many applications where the data generating mechanism is constantly changing, the stable period is typically very short and a nonparametric model for the longer period is preferred. Finally, we emphasize that the correlation stationarity test is an important tool to decide a period of prediction stability.


}

\vspace{3pt}


%

{\footnotesize
\bibliographystyle{abbrvnat}
\setcitestyle{authoryear}

\bibliography{corrtest}
}

\newpage
\begin{center}
{\Large Supplementary material for \\
Globally optimal and adaptive short-term forecasting of locally stationary time series and a test for its stability}
\end{center}
This supplementary material contains further explanation,  auxiliary lemmas and technical proofs for the main results of the paper.

\begin{appendix}

\renewcommand{\theequation}{S.\arabic{equation}}
\renewcommand{\thetable}{S.\arabic{table}}
\renewcommand{\thefigure}{S.\arabic{figure}}
\renewcommand{\thesection}{S.\arabic{section}}
\renewcommand{\thelem}{S.\arabic{lemma}}

\section{A few further remarks}\label{sec_appendix_one}
{ First, we will need the following further assumption in the paper. 

\begin{assu}\label{assu_basis} We assume that the following assumptions hold true for the sieve basis functions and parameters: \\
(1). For any $k=1,2,\cdots, b, $ denote $\Sigma^k(t) \in \mathbb{R}^{k \times k}$ whose $(i,j)$-th entry is  $\Sigma^k_{ij}(t)=\gamma(t, |i-j|),$ we assume that the eigenvalues of 
$$ \int_0^1 \Sigma^k(t) \otimes \left( \mathbf{B}(t) \mathbf{B}^*(t) \right),$$
are bounded above and also away from zero by a universal constant $\kappa>0$.\\
(2).  There exist constants $\omega_1, \omega_2 \geq 0,$ for some constant $C>0,$ we have 
\begin{equation*}
\sup_t || \nabla \mathbf{B}(t) || \leq C n^{\omega_1} c^{\omega_2}.
\end{equation*}
(3). We assume that for $\tau$ defined in Assumption \ref{phy_generalts}, $d$ defined in Assumption \ref{assu_smmothness} and $\alpha_1$ defined in (\ref{eq_defnc}), there exists a large constant $C>2,$ such that
\begin{equation*}
\frac{C}{\tau}+\alpha_1<1 \ \text{and} \ d \alpha_1>2.
\end{equation*}
\end{assu}
{We mention that the above assumptions are mild and easy to check. First, (1) of Assumption \ref{assu_basis} guarantees the invertibility of the design matrix $Y$ and the existence of the OLS solution. It can be easily verified that for the linear non-stationary process (\ref{ex_linear}), (1) will be satisfied if $\sup_t \sum |a_j(t)|<1.$ (See \cite[Lemma 3.4]{DZ1} for detailed discussion.)  
Second, (2) is a mild regularity condition on the sieve basis functions and satisfied by the commonly  used basis functions. We refer the readers to \cite[Assumption 4]{CC} for further details. (3)  can be easily satisfied by choosing $C<\tau$
and $\alpha_1$ accordingly. When the physical dependence is of exponential decay,
we only need  $d \alpha_1>2.$ We refer the readers to \cite[Assumption 3.5]{DZ1} for more details. }

Second, the accuracy of the robust bootstrap in Section \ref{sec_bootstrapping} is determined by the closeness of its conditional covariance structure to that of $\Omega$. Following  \cite[Section 4.1.1]{ZZ1}, we shall use
\begin{equation}\label{eq_omegadefn}
\mathcal{L}(m)=\left| \left|\widehat{\Omega}-\Omega\right| \right|,
\end{equation}
where $\widehat{\Omega}$ is defined in (\ref{eq_widehatomega}), to quantify the latter closeness. The following theorem establishes the bound for $\mathcal{L}(m)$. Its proof will be given in Section \ref{sec_mainproof}.

\begin{thm}[Optimal choice of $m$] \label{thm_choicem} Under the assumptions of Theorem \ref{thm_bootstrapping}, we have 
\begin{equation*}
\mathcal{L}(m)=O \left( b \zeta_c^2 \Big( \sqrt{\frac{m}{n}}+\frac{1}{m} \Big)  \right).
\end{equation*}
Consequently, the optimal choice is $\widehat{m}=O(n^{1/3}).$
\end{thm}
Note that compared to \cite[Theorem 4]{ZZ1}, the difference from Theorem \ref{thm_choicem} is that we get an extra factor $b\zeta_c^2$ due to the high dimensionality. For instance, when we use the Fourier basis, normalized Chebyshev orthogonal polynomials and orthogonal wavelet, we shall have that $b\zeta_c^2=p,$ which is the dimension of  $\bm{z}_i$ defined in (\ref{eq_zikronecker}). However, it will not influence the optimal choice of $m$.

} 

\section{Some auxiliary lemmas}

In this section, we collect some preliminary lemmas which will be used for our technical proofs.  First of all, we collect a result which provides a deterministic bound for the spectrum of a  square matrix.  Let  $A=(a_{ij})$ be a complex $ n\times n$ matrix. For  $1 \leq i \leq n,$ let  $R_{i}=\sum _{{j\neq {i}}}\left|a_{{ij}}\right| $ be the sum of the absolute values of the non-diagonal entries in the  $i$-th row. Let  $ D(a_{ii},R_{i})\subseteq \mathbb {C} $ be a closed disc centered at $a_{ii}$ with radius  $R_{i}$. Such a disc is called a \emph{Gershgorin disc.}
\begin{lem}[Gershgorin circle theorem]\label{lem_disc} Every eigenvalue of $ A=(a_{ij})$ lies within at least one of the Gershgorin discs  $D(a_{ii},R_{i})$, where $R_i=\sum_{j\ne i}|a_{ij}|$.
\end{lem}

{ The next lemma provides a lower bound for the eigenvalues of a Toeplitz matrix in terms of its associated spectral density function.  Since the autocovariance matrix of any stationary time series is a Toeplitz matrix, we can use the following lemma to bound the smallest eigenvalue of the autocovariance matrix. It will be used in the proof of Proposition \ref{prop_pdc} and can be found in \cite[Lemma 1]{XW}. 
\begin{lem}\label{lem_spectralbound}
Let $h$ be a continuous function on $[-\pi, \pi].$ Denote by $\underline{h}$ and $\overline{h}$ its minimum and maximum, respectively. Define $a_k=\int_{-\pi}^{\pi} h(\theta) e^{- \mathrm{i} k \theta} d \theta$ and the $T \times T$ matrix $\Gamma_T=(a_{s-t})_{1 \leq s, t \leq T}.$ Then 
\begin{equation*}
 2 \pi \underline{h} \leq \lambda_{\min}(\Gamma_T) \leq \lambda_{\max}(\Gamma_T) \leq 2 \pi \overline{h}. 
\end{equation*}
\end{lem}


%

}

The following lemma indicates that, under suitable condition, the inverse of a banded matrix can also be approximated by another banded-like matrix. It will be used in the proof of Theorem \ref{lem_phibound} and can be found in \cite[Proposition 2.2]{DMS}.  We say that $A$ is $m$-banded if $$A_{ij}=0, \ \text{if} \ |i-j|>m/2.$$
\begin{lem}\label{lem_band} Let $A$ be a positive definite, $m$-banded, bounded and bounded invertible matrix.  Let $[a,b]$ be the smallest interval containing the spectrum of $A.$ Set $r=b/a, q=(\sqrt{r}-1)/(\sqrt{r}+1)$ and set $C_0=(1+r^{1/2})^2/(2ar)$ and $\lambda=q^{2/m}.$ Then we have 
\begin{equation*}
|(A^{-1})_{ij}| \leq C \lambda^{|i-j|},
\end{equation*}
where 
\begin{equation*}
C:=C(a,r)=\max\{a^{-1}, C_0\}. 
\end{equation*}

\end{lem}

The following lemma provides an upper bound for the error of solutions of perturbed linear system. It can be found in the standard numerical analysis literature, for instance see \cite{num_paper}. It will be used in the proof of Theorem \ref{lem_phibound} and Lemma \ref{lem_controltbtrho}. 

\begin{lem}\label{lem_nuem} Consider a matrix $A$ and vectors $x,v$ satisfying the linear system
\begin{equation*}
Ax=v.
\end{equation*}
Recall that the conditional number of $A$ is defined as 
\begin{equation*}
\kappa(A)=\frac{\lambda_{\max}(A)}{\lambda_{\min}(A)}.
\end{equation*} 
If we add perturbations on both $A$ and $v$ such that  
\begin{equation*}
(A+\Delta A)(x+\Delta x)=v+\Delta v.
\end{equation*}
Assuming that the linear system is well-conditioned, i.e.,  the conditional number $\kappa(A)$  satisfies that, for some constant $C>0,$  
\begin{equation*}
\frac{\kappa(A)}{1-\kappa(A) \frac{|| \Delta A||}{A}} \leq C,
\end{equation*}
then we have that
\begin{equation*}
\frac{|| \Delta x ||}{|| x ||} \leq C \left( \frac{|| \Delta A||}{|| A||}+ \frac{|| \Delta v ||}{|| v||} \right).
\end{equation*}
\end{lem}

The following lemma provides Gaussian approximation result on convex sets for the sum of an $m$-dependent sequence, which is \cite[Theorem 2.1]{FX}. It will be used in the proof of Theorem \ref{thm_gaussian}. 
\begin{lem}\label{lem_xf} Let $W=\sum_{i=1}^n X_i$ be a sum of $d$-dimensional random vectors such that $\mathbb{E}(X_i)=0$ and $\operatorname{Cov} (W)=\Sigma_w.$ Suppose $W$ can be decomposed as follows: 
\begin{enumerate}
\item $\forall i \in [n], \ \exists i \in N_i \subset [n],$ such that $W-X_{N_i}$ is independent of $X_i$, where $[n]=\{1,\cdots,n\}.$
\item $\forall i \in [n], j \in N_i, \ \exists N_i \subset N_{ij} \subset [n],$ such that $W-X_{N_{ij}}$ is independent of $\{X_i, X_j\}.$
\item  $\forall i \in [n], j \in N_i, \ k \in N_{ij}, \ \exists N_{ij} \subset N_{ijk} \subset [n]$ such that $W-X_{N_{ijk}}$ is independent of $\{X_i, X_j, X_k\}.$
\end{enumerate}
Suppose further that for each $i \in [n], j \in N_i, k \in N_{ij},$
\begin{equation*}
|X_i| \leq \beta, |N_i| \leq n_1, |N_{ij}| \leq n_2, |N_{ijk}| \leq n_3,
\end{equation*}
where $|\cdot|$ is the Euclidean norm of a vector. Then there exists a universal constant $C$ such that 
\begin{equation*}
\mathcal{K}(W,Z) \leq C d^{1/4} n || \Sigma^{-1/2}||^3 \beta^3 n_1(n_2+\frac{n_3}{d}),
\end{equation*}
where $Z$ is a $d$-dimensional Gaussian random vector preserving the covariance structure of $W$ and $\mathcal{K}(\cdot, \cdot)$ is  the Kolmogorov distance defined in (\ref{eq_kolomogorv}). 
\end{lem}

{
The following lemma offers a control for the summation of Chi-square random variables, which will be employed in the proof of Theorem \ref{thm_gaussian}. It can be found in \cite[Lemma 7.2]{XZW}.
\begin{lem} \label{lem_chisquare} Let $a_1 \geq a_2 \geq \cdots \geq a_p \geq 0$ such that $\sum_{i=1}^p a^2_i=1;$ let $\eta_i$ be i.i.d. $\chi_1^2$ random variables. Then for all $h>0,$ we have 
\begin{equation*}
\sup_t \mathbb{P}(t \leq \sum_{k=1}^p a_k \eta_k \leq t+h) \leq \sqrt{h} \sqrt{4/\pi}.
\end{equation*} 
\end{lem}

}

Next, we collect some preliminary results. The first part of  the following lemma shows that the covariance function (\ref{eq_defncov}) decays polynomially fast under suitable assumptions. It can be found in  \cite[Lemma 2.6]{DZ1}.  The second part shows that the sample covariance matrix and its inverse will converge to some deterministic limits. Its proof is similar to \cite[Eq. (B.6)]{DZ1} and we omit the detail here. 

\begin{lem}\label{lem_coll} (1). Suppose  Assumptions \ref{phy_generalts}, \ref{assu_pdc}, \ref{assum_local} and \ref{assu_smmothness} hold true. Then there exists some constant $C>0,$ such that 
\begin{equation*}
\sup_t |\gamma(t,j)| \leq Cj^{-\tau}, \ j \geq 1. 
\end{equation*}
(2). Recall (\ref{eq_overlinesigma}).  Suppose  Assumptions \ref{phy_generalts}, \ref{assu_pdc}, \ref{assum_local}, \ref{assu_smmothness} and \ref{assu_basis} hold true. Then we have that 
\begin{equation*}
|| \widehat{\Sigma}-\Sigma ||=O_{\mathbb{P}} \Big( \frac{\zeta_c \log n}{\sqrt{n}} \Big),
\end{equation*}   
where $\widehat{\Sigma}=n^{-1}Y^*Y.$

\end{lem}

Finally, we collect the concentration inequalities for  non-stationary process using the physical dependence measure.  It is the key ingredient for the proof of most of the theorems and lemmas. 
It can be found in  \cite[Lemma 6]{ZZ1}. 

\begin{lem}\label{lem_con} Let $x_i=G_i(\mathcal{F}_i),$ where $G_i(\cdot)$ is a measurable function and  $\mathcal{F}_i=(\cdots, \eta_{i-1}, \eta_i)$ and $\eta_i, \ i  \in \mathbb{Z}$ are i.i.d  random variables. Suppose that $\mathbb{E}x_i=0$ and $\max_i \mathbb{E}|x_i|^q<\infty$ for some $q>1.$ For some $k>0,$ let $\delta_x(k):=\max_{ 1 \leq i \leq n} \norm{G_i(\mathcal{F}_i)-G_i(\mathcal{F}_{i,i-k})}_q.$ We further let $\delta_x(k)=0$ if $k<0.$ Write $\gamma_k=\sum_{i=0}^k \delta_x(i).$ Let $S_i=\sum_{j=1}^i x_j.$ \\
(i). For $q'=\min(2,q),$
\begin{equation*}
\norm{S_n}_q^{q'} \leq C_q \sum_{i=-n}^{\infty} (\gamma_{i+n}-\gamma_i)^{q'}.
\end{equation*}
(ii). If $\Delta:=\sum_{j=0}^{\infty} \delta_x(j) <\infty,$ we then have 
\begin{equation*}
\norm{\max_{1 \leq i \leq n}|S_i|}_q \leq C_q n^{1/q'} \Delta. 
\end{equation*}
In (i) and (ii), $C_q$ are generic finite constants which only depend on $q$ and can vary from place to place.  
\end{lem}

\section{Proof of main results }\label{sec_mainproof}
This section is devoted to the technical proof of the paper. 

\begin{proof}[Proof of Lemma \ref{eq_sufficinetconditionupdc}] The proof follows from a direct computation using Lemma \ref{lem_disc}. 

\end{proof}

{
\begin{proof}[Proof of Theorem \ref{lem_phibound}] We start with the proof of (\ref{eq_phibound1}). 
Till the end of the proof, we focus our discussion on each fixed $j.$  First, when $j=O(1),$ the result holds true immediately. We focus our discussion on the case when $j$ diverges with $n.$ Since $i>j,$ $i$ also diverges with $n.$  Denote $\bm{\phi_i}=(\phi_{i1}, \cdots, \phi_{i,i-1})^*.$ By Yule-Walker's equation, we have 
\begin{equation}\label{eq_system1}
\bm{\phi_i}=\Gamma_i^{-1} \bm{\gamma}_i,
\end{equation}
where $\Gamma_i=\text{Cov}(\bm{x}_{i-1}, \bm{x}_{i-1})$ and $\bm{\gamma}_i=\text{Cov}(\bm{x}_{i-1}, x_i) $ with $\bm{x}_{i-1}=(x_{i-1}, \cdots, x_1)^*.$  We denote the $(i-1) \times (i-1)$ symmetric banded matrix $\Gamma_i^s \equiv \Gamma_{i}^s(j)$  by
\begin{equation*}
(\Gamma_i^s)_{kl}=
\begin{cases} 
(\Gamma_i)_{kl}, & |k-l| \leq j^{1/(1+\epsilon)}; \\
0, & \text{otherwise}.
\end{cases}
\end{equation*}
By Assumption \ref{phy_generalts} and  a discussion similar to  \cite[Proposition 4]{ZZ2}, we conclude that for some constant $C>0,$
\begin{equation}\label{eq_generaauto}
\sup_{k,l}|\operatorname{Cov}(G_{k,n}(\mathcal{F}_k), G_{l,n}(\mathcal{F}_l) )| \leq C|k-l|^{-\tau}.
\end{equation}
Therefore, by {a discussion similar to (\ref{eq_boundlargeusesmall}) and the UPDC in Assumption \ref{assu_pdc}, we have $\lambda_{\min}(\Gamma_i^s) \geq \kappa-j^{-(\tau-1)/(1+\epsilon)}$ for all $i>j.$ } Similarly, we can show that $\lambda_{\max}(\Gamma_i^s) \leq C$ for some constant $C>0.$  This shows that the support of the spectrum of $\Gamma_i^s$ is bounded from both above and below by some constants when $n \geq (\frac{2}{\kappa})^{C(1+\epsilon)/(\tau-1)}$. By Lemma \ref{lem_band}, we conclude that for some $\delta \in (0,1)$ and some constant $C>0,$ we have 
\begin{equation}\label{eq_bandin}
\left| (\Gamma_i^s)^{-1}_{kl}\right| \leq C \delta^{|k-l|/j^{1/(1+\epsilon)}}.
\end{equation} 
By Cauchy-Schwarz inequality and Lemma \ref{lem_disc}, when $n$ is large enough, for some constant $C>0,$ we have that  
\begin{align}\label{eq_bandedapprox}
\left| \left| \Gamma_i^{-1}\bm{\gamma}_i-(\Gamma_i^s)^{-1} \bm{\gamma}_i \right| \right|  & \leq  
\| \Gamma_i-\Gamma_i^s \| \| \Gamma_i^{-1} \| \| (\Gamma_i^s)^{-1} \| \| \bm{\gamma}_i \| \nonumber \\
 & \leq C j^{-(\tau-1)/(1+\epsilon)},
\end{align}
where we use (\ref{eq_generaauto}), the UPDC in Assumption \ref{assu_pdc} and the fact $\lambda_{\min}(\Gamma_i^s) \geq \kappa-j^{-(\tau-1)/(1+\epsilon)}.$   
 Denote $\bm{\phi}_i^s=(\phi_{i1}^s, \cdots, \phi_{i,i-1}^s)$ such that 
\begin{equation*}
\bm{\phi}_i^s=(\Gamma_i^s)^{-1} \bm{\gamma}_i.
\end{equation*}
Then we get immediately from (\ref{eq_bandedapprox}) that 
\begin{equation}\label{eq_bbbbbbb111111}
|\phi_{ij}-\phi_{ij}^s| \leq Cj^{-(\tau-1)/(1+\epsilon)}.
\end{equation}
Hence, it suffices to control $\phi_{ij}^s.$ Note that $\phi_{ij}^s=\sum_{k=1}^{i-1} (\Gamma^s_i)_{jk}^{-1} \gamma_{ik},$ where $\gamma_{ik}=\text{Cov}(x_i, x_{i-k}).$ By (\ref{eq_generaauto}) and (\ref{eq_bandin}), when $n$ is large such that $i-1-j^{1/(1+\epsilon)+\epsilon/2(1+\epsilon)}>j^{1/(1+\epsilon)+\epsilon/2(1+\epsilon)},$ we have 
\begin{align*}
|\phi_{ij}^s| \leq C \sum_{k=1}^{i-1} \delta^{(i-1-k)/\sqrt{j}} k^{-\tau} & \leq C_1 \left( \sum_{k=1}^{i-1-j^{2/3}} \delta^{j^{\epsilon/2(1+\epsilon)}}+ \sum_{k=j^{1/(1+\epsilon)+\epsilon/2(1+\epsilon)}}^{i-1} k^{-\tau} \right)  \\
& \leq C_2 j^{-(1+\epsilon/2)(\tau-1)/(1+\epsilon)}.
\end{align*}
Together with (\ref{eq_bbbbbbb111111}), we conclude our proof of (\ref{eq_phibound1}).  

Then we proceed to prove the first term of (\ref{eq_phibound2}) using Lemma \ref{lem_nuem}. For the convenience of our discussion, we denote the $(k,l)$-entry of $\Gamma_i$ as $\Gamma_i(k,l).$
For $i>b,$ we denote the $(i-1) \times (i-1)$ block matrix $\Gamma_i^b$ and the block vector $\bm{\gamma}_i^b \in \mathbb{R}^{i-1}$ via
\begin{equation*}
\Gamma_{i}^b=
\begin{bmatrix}
\operatorname{Cov}(\bm{x}_i^b, \bm{x}_i^b) & \bm{E}_1 \\
\bm{E}_3 & \bm{E}_2
\end{bmatrix}, \ 
\bm{\gamma}_{i}^b=(\operatorname{Cov}(\bm{x}_i^b, x_i), \bm{0}),
\end{equation*}
where $\bm{x}_i^b=(x_{i-1}, \cdots, x_{i-b})^*$ and $\bm{E}_i, i=1,2,$ are defined as
\begin{equation}\label{eq_defne1e2}
\bm{E}_1=\operatorname{Cov}( \bm{x}_i^m , \bm{x}^b_i) \in \mathbb{R}^{b \times (i-b-1)}, \  \bm{E}_2=\operatorname{Cov}(\bm{x}_i^m, \bm{x}_i^m) \in \mathbb{R}^{(i-b-1) \times (i-b-1)},
\end{equation}
and $\bm{x}_i^m=(x_{i-b-1}, \cdots, x_1)^*.$ Moreover, $\bm{E}_3=(\bm{E}_3(k,l)) \in \mathbb{R}^{(i-b-1) \times b}.$ Note that $b+1 \leq k \leq i-1, 1 \leq l \leq b.$ For some constant $\varsigma>2,$ $\bm{E}_3$ is denoted as 
\begin{equation}\label{eq_e3}
\bm{E}_3(k,l)=
\begin{cases}
\Gamma_i(k,l) & |k-l| \leq b/\varsigma,\\
0 & \ \text{otherwise}.  
\end{cases}
\end{equation}

Denote $\bm{\phi}_i^b=(\phi^b_{i1}, \cdots, \phi^b_{ib}, \bm{0}) \in \mathbb{R}^{i-1}.$ We have that 
\begin{equation}\label{eq_system2}
\Gamma_i^b\bm{\phi}_i^b=\bm{\gamma_i}^b-\Delta \bm{\gamma}_i, 
\end{equation}
where $ \Delta\bm{\gamma}_i$ is defined as 
\begin{equation*}
\Delta \bm{\gamma}_i=(\bm{E}_3 \widetilde{\bm{\phi}}_i^b, \bm{0}), \ \widetilde{\bm{\phi}}_i^b=(\phi_{i1}^b, \cdots, \phi_{ib}^b)^*.
\end{equation*}
Since  
\begin{equation*}
\max_{1 \leq j \leq b}|\phi_{ij}-\phi_{ij}^b|\leq \| \bm{\phi}_i-\bm{\phi}_i^b \|,
\end{equation*}
it suffices to provide an upper bound for $\| \bm{\phi}_i-\bm{\phi}_i^b \|.$ Now we employ Lemma \ref{lem_nuem} with $A=\Gamma_i, \Delta A=\Gamma_i^b-\Gamma_i, x=\bm{\phi}_i, \Delta x=\bm{\phi}_i^b-\bm{\phi}_i, v=\bm{\gamma}_i, \Delta v=\bm{\gamma}_i^b-\bm{\gamma}_i-\Delta \bm{\gamma}_i$ to the systems (\ref{eq_system1}) and (\ref{eq_system2}). By the UPDC in Assumption \ref{assu_pdc}, for some constant $C>0,$ we find that $\kappa(A) \leq C.$ By Lemma \ref{lem_disc} and (\ref{eq_generaauto}), we find that for some constant $C>0,$ we have 
\begin{equation*}
\| \Delta A \|\leq (b/\varsigma)^{-\tau+1} \leq C b^{-\tau+1}.
\end{equation*}
Moreover, note that
\begin{equation*}
\| \Delta v \| \leq \| \bm{\gamma}_i^b-\bm{\gamma}_i \|+\| \Delta \bm{\gamma}_i \|.
\end{equation*}
The first term of the right-hand side of the above equation can be bounded by $C b^{-\tau+1}$ using (\ref{eq_generaauto}). For the second term, by a discussion similar to (\ref{eq_phibound1}), we find that $|\phi_{ij}^b| \leq C j^{-(\tau-1)/(1+\epsilon)}.$ Involving the definition of $\bm{E}_3,$ we obtain that for some constant $C_1>0$
\begin{equation*}
\| \Delta \bm{\gamma}_i \| \leq C\min\{b/\varsigma,i-b-1\}(b/\varsigma)^{-(\tau-1)/(1+\epsilon)+1} \leq C_1 b^{-(\tau-1)/(1+\epsilon)+2}.  
\end{equation*}
Consequently, we have that 
\begin{equation}\label{eq_l2norm}
\| \bm{\phi}_i-\bm{\phi}_i^b \| \leq C b^{-(\tau-1)/(1+\epsilon)+2}=C n^{-1+(3+2\epsilon)/\tau}.
\end{equation} 
This finishes our proof of (\ref{eq_phibound1}). 

Finally, we prove the second term of (\ref{eq_phibound2}). Recall $\phi_{i0}$ is defined as 
\begin{equation}\label{eq_noncenteredintercept}
\phi_{i0}:=\mu_i-\sum_{j=1}^{i-1} \phi_{ij} \mu_{i-j}.
\end{equation}
and 
 $\phi_{i0}^b=\mu_i-\sum_{j=1}^b\phi_{ij}^b \mu_{i-j}, $ where $\mu_i=\mathbb{E}x_i, i=1,2,\cdots, n,$ is  the sequence of trends of $\{x_i\}.$  We have 
\begin{equation*}
\phi_{i0}-\phi_{i0}^b=\sum_{j=1}^b(\phi_{ij}^b-\phi_{ij})\mu_{i-j}-\sum_{j=b+1}^{i-1} \phi_{ij} \mu_{i-j}.
\end{equation*}
The first term of the right-hand side of the above equation can be bounded by $Cn^{-1+(3.5+2.5\epsilon)/\tau}$ using (\ref{eq_l2norm}) and Cauchy-Schwarz inequality and the second term can be bounded by $Cn^{-1+(2+\epsilon)/\tau}$ using (\ref{eq_phibound1}). This concludes our proof. 


%
\end{proof}
}

\begin{proof}[Proof of Proposition \ref{prop_dependent}] First of all, when $i \leq
 b,$ it holds by setting $\phi_{ij}$ to be the coefficients of best linear prediction. When $i>b, $  by (\ref{eq_arapproxiamtionnoncenter}), we decompose that 
 \begin{equation*}
 x_i=\phi_{i0}+\sum_{j=1}^b \phi_{ij}x_{i-j}+\epsilon_i+\sum_{j=b+1}^{i-1} \phi_{ij} x_{i-j}.
 \end{equation*}
Moreover, by Theorem \ref{lem_phibound},  under the assumption (\ref{eq_boundx}), we find that 
\begin{equation*}
\sum_{j=b+1}^{i-1} \phi_{ij} x_{i-j}=O_{\mathbb{P}}(n^{-1+(2+\epsilon)/\tau}). 
\end{equation*}
This concludes our proof. 
\end{proof}

\begin{proof}[Proof of Theorem \ref{thm_arrepresent}] Recall (\ref{defn_xistart}).
Clearly, $\{x_i^*\}$ is an AR($b$) process when $i>b.$ For $i=b+1,$ we have that
\begin{equation*}
x_i-x_i^*=0.
\end{equation*}
Suppose (\ref{eq_induction}) holds true for $k>b+1,$ then for $k+1,$ we have  
\begin{align*}
x_{k+1}-x_{k+1}^* & =\sum_{j=1}^b \phi_{ij} (x_{k+1-j}-x^*_{k+1-j})+\sum_{j=b+1}^{k} \phi_{ij} x_{k+1-j} \\
& =O_{\mathbb{P}}(n^{-1+(2+\epsilon)/\tau}), 
\end{align*} 
where in the second step we use induction and Theorem \ref{lem_phibound} that $\sum_{j=1}^b |\phi_{ij}|<\infty$.
\end{proof}

\begin{proof}[Proof of Lemma \ref{lem_pacftruncation}]
 Denote $\bm{\phi}_{i,j}=(\phi_{i1,j}, \cdots, \phi_{ij,j})^*.$ By  Yule-Walker's equation, we have
\begin{equation}\label{eq_yulewalkergeneralnon}
\bm{\phi}_{i,j}=\Omega_{i,j} \bm{\gamma}_{i,j}.
\end{equation}
Here $\Omega_{i,j}=[\text{Cov}(\bm{x}_i^j, \bm{x}_i^j)]^{-1}$ and $\bm{\gamma}_{i,j}=\text{Cov}(\bm{x}_i^j,x_i),$ where $\bm{x}_i^j:=(x_{i-1}, \cdots, x_{i-j})^*$. Moreover, $\phi_{i0,j}$ is defined similar to (\ref{eq_noncenteredintercept}). Then it is easy to see that
the proof is similar to that of Theorem \ref{lem_phibound} except we need to change the dimension from $i-1$ to $j$.
\end{proof}

{
\begin{proof}[Proof of Proposition \ref{prop_pdc}]
Denote the covariance matrix of $(x_1, \cdots, x_n)$ as $\Sigma \equiv \Sigma_n.$  For a given $d_n=O(n^f), \ \frac{1}{\tau}<f<\frac{1}{2},$ we define the banded matrix $\Sigma^{d_n}$ such that 
\begin{equation*}
\Sigma_{ij}^{d_n}= 
\begin{cases}
\Sigma_{ij}, & \text{if} \ |i-j| \leq d_n; \\
0 , & \text{otherwise}. 
\end{cases}
\end{equation*} 
Throughout the proof, we let $\lambda_n$ be the smallest eigenvalue of $\Sigma$ and $\mu_n$ be that of $\Sigma^{d_n}.$ By Lemmas \ref{lem_coll} and \ref{lem_disc}, we have 
\begin{equation}\label{eq_boundlargeusesmall}
\lambda_n=\mu_n+o(1). 
\end{equation}
Therefore, it is equivalent to study the UPDC for $\Sigma^{d_n}.$ We now consider a longer time series $\{x_i\}_{i=-d_n}^{n+d_n},$ where we use the convention $x_i=G(0, \mathcal{F}_i)$ if $i<0$ and $x_i=G(1, \mathcal{F}_i)$ if $i>n.$ We will need the following lemma to prove the sufficiency.
\begin{lem} \label{lem_bandedbound} Let $\Sigma_i^{d_n}$ be the covariance matrix of $(x_i, x_{i+1}, \cdots, x_{i+d_n}).$ Then for  all $-d_n \leq i \leq n,$ let $\lambda_{d_n}(\Sigma_i^{d_n})$ be the smallest eigenvalue of $\Sigma_i^{d_n}.$ Then if the spectral density (\ref{eq_spetraldensity}) is bounded from below, we have that for some constant $\varsigma>0,$
\begin{equation*}
\lambda_{d_n} (\Sigma_i^{d_n}) \geq \varsigma>0, \  \text{for all} \ i. 
\end{equation*}
\end{lem}
\begin{proof}
Without loss of generality, we set $i=0.$ Consider the stationary process such that $x_i^0=G(0, \mathcal{F}_i).$ By Lemma \ref{lem_spectralbound}, when the spectral density is bounded below, we find that 
\begin{equation}\label{eq_boundone}
\lambda_{d_n}(\text{Cov}(x_i^0, \cdots, x_{d_n}^0)) \geq \varsigma>0,
\end{equation}
for any $d_n.$ On the other hand, when $1 \leq i, j \leq d_n,$ for some constant $C>0,$ we have 
\begin{equation*}
\left| \text{Cov}(x_i, x_j)-\text{Cov} (G(0, \mathcal{F}_i), G(0, \mathcal{F}_j)) \right| \leq C \min \left( \frac{\max(i,j)}{n}, |i-j|^{-\tau} 
\right).
\end{equation*}
As a consequence, by Lemma \ref{lem_disc}, we find that
\begin{equation*}
|\lambda_{d_n}(\Sigma_i^{d_n})-\lambda_{d_n}(\text{Cov}(x_i^0, \cdots, x_{d_n}^0))| \leq C \frac{d_n^2}{n}.
\end{equation*}
Together with (\ref{eq_boundone}), we finish the proof.
\end{proof}
With the above preparation, we proceed with the final proof. We start with the sufficiency part. For any non-zero vector $\bm{a}=(a_1, \cdots, a_{i+2d_n})^* \in \mathbb{R}^{i+2d_n},\ i=-d_n, \cdots, n,$ denote 
\begin{equation*}
F(\bm{a}, i):=\sum_{k=1}^{i+d_n} \sum_{l=1}^{i+d_n} a_k (\Sigma_i^{d_n})_{k,l} a_l.
\end{equation*}  
By Lemma \ref{lem_bandedbound}, we find that 
\begin{equation}\label{eq_fbound}
F(\bm{a}, i) \geq \varsigma \sum_{l=i}^{i+d_n} a_l^2.  
\end{equation}
Now we let the first and last $d_n$ entries of $\bm{a}$ be zeros. Then using a discussion similar to (\ref{eq_fbound}), we find that 
\begin{equation}\label{eq_finalbound2}
\frac{1}{d_n} \sum_{i=-d_n}^n F(\bm{a}, i) \geq \varsigma \sum_{l=1}^n a_l^2. 
\end{equation}
Furthermore, by Lemma \ref{lem_disc}, it is easy to see that for some constant $C>0,$
\begin{equation*}
\left| \frac{1}{d_n} \sum_{i=-d_n}^n F(\bm{a}, i)-\sum_{k=1}^n \sum_{l=1}^n a_k \Sigma^{d_n}_{kl} a_l \right| \leq \frac{C}{d_n} \sum_{k=1}^n a_k^2. 
\end{equation*}
Together with (\ref{eq_finalbound2}), we find that 
\begin{equation*}
\sum_{k=1}^n \sum_{l=1}^n a_k \Sigma^{d_n}_{kl} a_l \geq \frac{\varsigma}{2} \sum_{l=1}^n a_l^2,
\end{equation*}
when $n$ is large enough. This shows that $\Sigma^{d_n}$ satisfies PDC and hence finishes the proof of the sufficient part. Next we briefly discuss the proof of necessity. We make use of the structure of $\Sigma^{d_n}.$ For any given $t_i:=\frac{i}{n}$ and $\omega,$ denote 
\begin{equation*}
f_{n}(t_i, \omega)=\frac{1}{2 \pi n} \sum_{k,l=1}^n e^{-\mathrm{i} k \omega} \gamma(t_i, k-l)  e^{\mathrm{i} l \omega}.
\end{equation*}
It is easy to see that (for instance see a similar discussion in \cite[Corollary 4.3.2]{BD})
\begin{equation}\label{eq_ffi}
f_n(t_i, \omega)= f(t_i, \omega)+o(1).
\end{equation} 
Furthermore, we denote 
\begin{equation*}
g_n(t_i,\omega)=\frac{1}{2 \pi n} \sum_{k,l=1}^n e^{-\mathrm{i} k \omega}  \Sigma^{d_n}_{i, |k-l|} e^{\mathrm{i} l \omega}.
\end{equation*}
By the assumption (\ref{assum_lip}) and Lemma \ref{lem_coll}, we find that 
\begin{equation}\label{eq_gf}
g_n(t_i, \omega)=f_n(t_i, \omega)+o(1).
\end{equation}
Using the structure of $\Sigma^{d_n}$ and the assumption that $\Sigma^{d_n}$ satisfies UPDC, we find that for some constant $\kappa>0,$
\begin{equation*}
g_n(t_i, \omega) \geq \kappa. 
\end{equation*}
In light of  (\ref{eq_ffi}) and (\ref{eq_gf}), we find that $f(t_i, \omega) \geq \kappa.$ Finally, we can conclude our proof using the continuity of $f(t, \omega)$ in $t$. This concludes our proof.   
\end{proof}

}

%

{
\begin{proof}[Proof of Theorem \ref{thm_locallynonzero}] The smoothness of the functions $\phi_{j}(t), 0 \leq j \leq b$ follows from term by term differentiation by using Assumption \ref{assu_smoothtrend} and Theorem \ref{lem_phibound}. 

Next, note that under Assumption \ref{assu_smoothtrend}, we  can write (\ref{eq_noncenteredintercept}) as
\begin{equation}\label{defn_psidiscts}
\phi_{0i}=\mu(i/n)-\sum_{j=1}^{i-1} \phi_{ij} \mu((i-j)/n).
\end{equation}
The first equation is proved in \cite[Lemma 2.8]{DZ1} and the second equation follows from the first equation,  the smoothness of $\mu (t)$ such that $|\mu(i/n)-\mu((i-j)/n)| \leq C b/n,$ when $|i-j| \leq b$ for some constant $C>0$ and the first term of (\ref{eq_phibound1}).

\end{proof}
}

{
\begin{proof}[Proof of Corollary \ref{cor_localboundmse}] The proof is similar to those of Theorem \ref{thm_arrepresent} and we omit further details here. 
\end{proof}

}

\begin{proof}[Proof of Lemma \ref{lem_pacf}] The smoothness of $\rho_j(t)$ follows from Assumption \ref{assu_smmothness} and a discussion similar to the proof of Theorem \ref{thm_locallynonzero}. For the second part, we first note that  $\rho_{i,j}=\phi_{ij,j}$, $\rho_{j}(i/n)=\phi_{j,j}(i/n)$ and 
\begin{equation*}
\phi_{ij,j}-\phi_{j,j}(\frac{i}{n})=\mathbf{e}_j^* \Omega_{i,j}(\bm{\gamma}_{i,j}-\widetilde{\bm{\gamma}}_{i,j})+\mathbf{e}_j^* \Omega_{i,j} (\widetilde{\Gamma}_{i,j}-\Gamma_{i,j}) \widetilde{\Omega}_{i,j} \widetilde{\bm{\gamma}}_{i,j}, 
\end{equation*}
where $\mathbf{e}_j=(0,\cdots,0,1)^*$ and  we recall (\ref{eq_yulewalkergeneralnon}) and (\ref{eq_yulewalkereqlocalstationary}). For the first term of the right-hand side of the above equation, by Cauchy-Schwarz inequality, we have 
\begin{equation*}
\left| \mathbf{e}_j^* \Omega_{i,j} (\bm{\gamma}_{i,j}-\widetilde{\bm{\gamma}}_{i,j})  \right|^2 \leq \lambda_{\max} \Big( \Omega_{i,j} \Omega_{i,j}^* \Big) || \bm{\gamma}_{i,j}-\widetilde{\bm{\gamma}}_{i,j}  ||_2^2.
\end{equation*} 
First, by Assumption  \ref{assu_pdc}, we find that for some constant $C>0,$ we have 
\begin{equation*}
\lambda_{\max}\left( \Omega_{i,j} \Omega_{i,j}^* \right) \leq C.
\end{equation*}
Second, by Assumptions \ref{phy_generalts} and \ref{assum_local}, together with (\ref{eq_generaauto}) and Lemma \ref{lem_coll}, we conclude that 
\begin{align}
\left| \mathbf{e}_j^* \Omega_{i,j} (\bm{\gamma}_{i,j}-\widetilde{\bm{\gamma}}_{i,j})  \right|^2 & \leq C \left( \sum_{k=1}^{\min\{j,b\}} (\gamma_{i,j}(k)-\widetilde{\gamma}_{i,j}(k))^2+\sum_{k=\min\{j,b\}+1}^{\max\{j,b\}} (\gamma_{i,j}(k)-\widetilde{\gamma}_{i,j}(k))^2 \right) \label{eq_firstline} \\
& \leq C n^{-2+3(1+\epsilon)/\tau}. \nonumber
\end{align}
Here, for the first term of the right-hand side of (\ref{eq_firstline}), we use  (\ref{assum_lip}) to obtain that 
\begin{equation*}
|\gamma_{i,j}(k)-\widetilde{\gamma}_{i,j}(k)| \leq n^{-1+(1+\epsilon)/\tau}, 
\end{equation*}
and use (\ref{eq_generaauto}) to control the second term of (\ref{eq_firstline}).  Similarly, we can show that 
\begin{equation*}
\left| \mathbf{e}_j^* \Omega_{i,j} (\widetilde{\Gamma}_{i,j}-\Gamma_{i,j}) \widetilde{\Omega}_{i,j} \widetilde{\bm{\gamma}}_{i,j} \right|^2 \leq C n^{-2+3(1+\epsilon)/\tau}.
\end{equation*}
This concludes our proof. 
\end{proof}

\begin{proof}[Proof of Theorems \ref{thm_prediction}] 
Note that by adding and subtracting $\widehat{x}_{n+1},$ we have 
\begin{align*}
\mathbb{E}(x_{n+1}-\widehat{x}_{n+1}^b)^2=\mathbb{E}(x_{n+1}-\widehat{x}_{n+1})^2+\mathbb{E}(\widehat{x}_{n+1}-\widehat{x}_{n+1}^b)^2+2\mathbb{E}(x_{n+1}-\widehat{x}_{n+1})(\widehat{x}_{n+1}-\widehat{x}_{n+1}^b).
\end{align*}
It suffices to control the second and third terms of the above equations. First, 
\begin{equation}\label{eq_expandifference}
\widehat{x}_{n+1}-\widehat{x}_{n+1}^b=\sum_{j=1}^b(\phi_{nj}-\phi_j(1))x_{n+1-j}+\sum_{j=b+1}^n
\phi_{nj} x_{n+1-j}.
\end{equation}
Therefore, by Theorem \ref{thm_locallynonzero}, (\ref{eq_boundx}) and (\ref{eq_phibound1}), we find that there exists some constant $C>0$ such that 
\begin{equation}\label{eq_aaa}
\mathbb{E}(\widehat{x}_{n+1}-\widehat{x}_{n+1}^b)^2 \leq C n^{-2+5(1+\epsilon)/\tau}.
\end{equation}
Second,  since $\widehat{x}_{n+1}$ is the best linear forecasting based on $\{x_1, \cdots, x_n\}$, then $x_{n+1}-\widehat{x}_{n+1}$ is uncorrelated with  any linear combination of $\{x_1, \cdots, x_n\}$. Together with (\ref{eq_expandifference}), we readily obtain that 
\begin{equation*}
\mathbb{E}(x_{n+1}-\widehat{x}_{n+1})(\widehat{x}_{n+1}-\widehat{x}_{n+1}^b)=0.
\end{equation*}
This completes our proof. 
\end{proof}

\begin{proof}[Proof of Theorems \ref{thm_finalresult} and \ref{thm_consistency}]
For the proof of Theorem \ref{thm_finalresult}, the first part of follows from a discussion similar to \cite[Theorem 3.7 and Corollary 3.8]{DZ1}. The only difference is that  in the statements of the aforementioned results, the above statements are only proved for $j \geq 1.$ Since the case when $j=0$ holds with a similar discussion, we omit the details of the proof here. Indeed the only difference is that our design matrix $Y $ is the $(n-b) \times (b+1)c$ rectangular matrix whose $i$-th row is $\bm{x}_{i} \otimes \mathbf{B}(\frac{i}{n}).$  Here $\bm{x}_{i}=(1,x_{i-1}, \cdots, x_{i-b}) \in \mathbb{R}^{b+1},$ $\mathbf{B}(i/n)=(\alpha_1(\frac{i}{n}), \cdots, \alpha_c(\frac{i}{n})) \in \mathbb{R}^c$ and $\otimes$ is the Kronecker product. The second part follows from a discussion similar to \cite[Theorem 3.2]{DZ1}, the first part of the theorem, (\ref{eq_vvvffff}) and the smoothness of $\varphi(\cdot)$.  

 The proof of Theorem \ref{thm_consistency} is the similar to that of Theorem \ref{thm_finalresult} except that our $Y^* \in \mathbb{R}^{(j+1)c  \times (n-b)}.$ However, as $j \leq b,$ the discussion can be applied to our case directly. 
\end{proof}



%
%
%
%
%
%

\begin{proof}[Proof of Theorem \ref{lem_reducedtest}]
 By Lemma \ref{lem_coll} and Assumption \ref{assum_local}, we find that there exists some constant $C>0,$ such that 
\begin{equation*}
\sup_i \left| \operatorname{Corr}(x_i, x_j) \right| \leq C |i-j|^{-\tau}, \ i \neq j. 
\end{equation*}
Therefore, we only consider the correlation when $|i-j| \leq b.$ Indeed, due to Assumption \ref{assum_local}, for some constant $C>0,$ we have 
\begin{equation*}
\sup_{1 \leq i, j \leq b} \left| \operatorname{Corr}(x_i, x_{j})-\operatorname{Corr}(x_{b+i}, x_{b+j}) \right| \leq Cn^{-1+(1+\epsilon)/\tau}.
\end{equation*}  
Therefore, it suffices to test the stationarity  for the correlation of $x_i$ and $x_j,$ where $i,j>b$ and $|i-j| \leq b.$  First, by the smoothness of $G(\cdot, \cdot)$, we observe that for any $i>b,$ we can write 
\begin{equation*}
\operatorname{Var} \left(G(\frac{i}{n}, \mathcal{F}_{i-k})\right)=\sigma^2(\frac{i}{n}), \  1 \leq k \leq b.  
\end{equation*}
As a consequence, using Yule-Walker's equation, we have 
\begin{equation*}
\widetilde{\bm{\phi}}^b(\frac{i}{n})=\widetilde{\mathrm{P}}^{-1} \widetilde{\bm{\rho}}_b, 
\end{equation*}
where $\widetilde{\mathrm{P}}$ is the correlation matrix of $\widetilde{\bm{x}}_{i-1}=(\widetilde{x}_{i-1}, \cdots, \widetilde{x}_{i-b})^*$ and $\widetilde{\bm{\rho}}_b$ is the correlation vector of $\widetilde{x}_i$ and $\widetilde{\bm{x}}_{i-1}.$ Here we recall $\widetilde{\bm{x}}_{i-1,k}=G(\frac{i}{n},\mathcal{F}_{i-k}), \ k=1,2,\cdots,b,$ where $\widetilde{\bm{x}}_{i-1,k}$ is the $k$-th entry of $\widetilde{\bm{x}}_{i-1}.$ Hence, under $\mathbf{H}_0,$ we conclude that $\widetilde{\bm{\phi}}^b$ is independent of $\frac{i}{n}.$ Second, we have that 
\begin{equation*}
\operatorname{Corr} (x_i, x_{i+j})=\frac{ \sum_{k=1}^b \phi_k \operatorname{Cov}(x_i,x_{i+j-k}) }{\operatorname{Var} x_{i+j} }+O(n^{-1+(1+\epsilon)/\tau}),
\end{equation*}
where we use the stochastic Lipschitz continuity and the expression (\ref{eq_arapproxiamtionnoncenter}) of $x_{i+j}$. Using the fact that 
\begin{equation*}
\operatorname{Cov}(x_i,x_{i+j-k})=\eta_{|k-j|} \operatorname{Var} x_{i+j}+O(n^{-1+(1+\epsilon)/\tau}), 
\end{equation*}
where $\eta_{|k-j|}=\operatorname{Corr}\left(G(\frac{i+j}{n}, \mathcal{F}_i), G(\frac{i+j}{n},\mathcal{F}_{i+j-k})\right)$ is the correlation function for a stationary time series. By Theorem \ref{lem_phibound}, we have that $\sum |\phi_j|<\infty,$
we can choose $\varrho_{j}=\sum_{j=k}^b \phi_k \eta_{|k-j|}.$ This concludes our proof.  
\end{proof}

\begin{proof}[Proof of Lemma \ref{lem_reducedquadratic}] 
Under the null assumption $\mathbf{H}_0$ that $\phi_j(t)$ are identical in $t$,  we have
\begin{equation*}
(\widehat{\phi}_j(t)-\overline{\widehat{\phi}}_j)^2=\left(\widehat{\phi}_j(t)-\phi_j(t)-\left(\int_0^1 (\widehat{\phi}_j(s)-\phi_j(s) )ds\right)\right)^2.
\end{equation*}
By (\ref{eq_phiform}), we can  write 
\begin{equation*}
T=\sum_{j=1}^b \left( \bm{\beta}_j^*-\widehat{\bm{\beta}}_j^* \right) W \Big( \bm{\beta}_j-\widehat{\bm{\beta}}_j \Big)+O(bc^{-d}), \ W=\Big(I-\bar{B} \bar{B}^*\Big),
\end{equation*}
where $\bm{\beta}_j \in \mathbb{R}^c$ satisfies that $\bm{\beta}_{jk}=\bm{\beta}_{jc+k}, \ 1 \leq k \leq c.$ 
It is well-known that the OLS estimator satisfies 
\begin{equation}\label{beta_est}
\widehat{\bm{\beta}}=\bm{\beta}+\left(\frac{Y^*Y}{n} \right)^{-1}\frac{Y^* \bm{\epsilon}}{n}, \ \bm{\epsilon}=(\epsilon_{b+1}, \cdots, \epsilon_n)^*.
\end{equation}
 By (\ref{beta_est}), we find that $nT$ is a quadratic form in terms of $\frac{1}{\sqrt{n}} \sum_{i=b+1}^n \bm{z}_i^*.$ 
We find that 
\begin{equation*}
nT=\mathbf{X}^* \left( \frac{Y^*Y}{n} \right)^{-1} \mathbf{I}_{bc} \mathbf{W} \left( \frac{Y^*Y}{n} \right)^{-1} \mathbf{X}+O(bc^{-d}).
\end{equation*}
By (2) of Lemma \ref{lem_coll} and (1) and (3) of Assumption \ref{assu_basis}, we can conclude our proof.  

\end{proof}

\begin{proof}[Proof of Theorem \ref{thm_gaussian}]  Denote 
\begin{equation*}
A_x:=\Big\{ \mathbf{W} \in \mathbb{R}^p: \mathbf{W}^* \Gamma \mathbf{W} \leq x \Big\}, 
\end{equation*}
and $\mathcal{A}=\{x \in \mathbb{R}:  A_x \}.$ It is easy to check that  $A_x$  is convex as $\Gamma$ is positive semi-definite. By definition, we have 
\begin{align}\label{eq_defncalk}
\mathcal{K}(\mathbf{X}, \mathbf{Y})=\sup_x \Big| \mathbb{P} \Big( \mathbf{X} \in A_x \Big)-\mathbb{P} \Big( \mathbf{Y} \in A_x \Big) \Big|=\sup_{A \in \mathcal{A}} \Big| \mathbb{P} \Big( \mathbf{X} \in A \Big)-\mathbb{P} \Big( \mathbf{Y} \in A \Big) \Big|,
\end{align}
where $A$ is a convex set. Given a large constant $M \equiv M(n),$ denote
\begin{equation*}
\bm{h}_i^M= \mathbb{E}(\bm{h}_i | \eta_{i-M}, \cdots, \eta_i), \ i=b+1, \cdots,n,
\end{equation*}
and $\bm{z}_i^M=\bm{h}_i^M \otimes \mathbf{B}(\frac{i}{n})=(z_{i1}^M, \cdots, z_{ip}^M)^*, p=(b+1)c.$ Then we can define $\mathbf{X}^M$ accordingly and then $\mathbf{Y}^M$ can be defined similarly. Note that  in Lemma \ref{lem_xf}, we have $n_1=n_2=n_3=M.$ Next we provide a truncation for the $M$-dependent sequence.      Now we choose $M_z$ for $\gamma \in (0,1),$ such that
\begin{equation*}
\mathbb{P} \Big( \max_{b+1 \leq i \leq n} \max_{1 \leq j \leq p} |z^M_{ij}| \geq M_z\Big) \leq \gamma.
\end{equation*}
Denote the set
\begin{equation*}
\mathcal{B}(M_z):=\Big\{ \max_{b+1 \leq i \leq n} \max_{1 \leq j \leq p} |z^M_{ij}| \leq M_z\Big\},
\end{equation*}
and $\mathbf{X}=(X_1, \cdots, X_p).$ Similarly, we can define its $M$-dependent approximation as $\mathbf{X}^M$ and truncated version as $\overline{\mathbf{X}}^M.$ We decompose the probability by  
\begin{align}
\mathcal{K}(\mathbf{X}^M, \mathbf{Y})&=
\mathcal{K}(\mathbf{X}^M, \mathbf{Y} \cap \mathcal{B}(M_z)) +\mathcal{K}(\mathbf{X}^M, \mathbf{Y} \cap \mathcal{B}^c(M_z)) \nonumber \\
& \leq \mathcal{K}(\overline{\mathbf{X}}^M, \mathbf{Y})+C \gamma, \label{eq_finalshow}
\end{align}
where $C>0$ is some constant. 
%
Note that on $\mathcal{B}(M_z),$
\begin{equation*}
\Big|\frac{1}{\sqrt{n}}\bm{z}^M_i \Big|=\frac{1}{\sqrt{n}}\Big| \bm{h}^M_i \otimes \mathbf{B}(\frac{i}{n}) \Big| \leq \frac{C \sqrt{p} M_z}{\sqrt{n}}. 
\end{equation*}
Denote $\widetilde{\mathbf{Y}}^M$ as the Gaussian random vector with the same covariance structure with $\overline{\mathbf{X}}^M.$ By Lemma \ref{lem_xf}, we conclude that 
\begin{equation*}
\mathcal{K}(\overline{\mathbf{X}}^M, \widetilde{\mathbf{Y}}^M) \leq C p^{\frac{7}{4}} n^{-1/2}M_z^3 M^2.
\end{equation*}
In light of (\ref{eq_finalshow}), it suffices to control the difference of the covariance matrices between $\overline{\mathbf{X}}^M$ and $\mathbf{X}.$ We first recall the following fact: if $Y \geq 0$ and $\zeta>0$ then 
\begin{equation}\label{eq_positiveexpecation.}
\mathbb{E} Y^\zeta=\int_0^\infty \zeta y^{\zeta-1} \mathbb{P}(Y>y)dy. 
\end{equation}
First, we show that the covariance matrices between $\overline{\mathbf{X}}^M$ and $\mathbf{X}^M$ are close. For $i=1,2,\cdots,p,$
\begin{equation}\label{eq_variancedecomposition}
\operatorname{Var}(\overline{X}_i^M)-\operatorname{Var}(X_i^M)=\mathbb{E}(\overline{X}_i^M)^2-\mathbb{E} (X_i^M)^2+(\mathbb{E}(\overline{X}_i^M-X_i^M))(\mathbb{E}(\overline{X}_i^M+X_i^M)).
\end{equation}
Note that
\begin{align*}
\mathbb{E}(\overline{X}_i^M-X_i^M) & \leq \mathbb{E}|\overline{X}_i^M-X_i^M| \leq M_z \mathbb{P}(|X_i^M|>M_z)+ \int_{M_z}^{\infty} \mathbb{P}(|X_i^M|>y)dy  \\
& \leq C\int_{M_z}^{\infty} \frac{1}{y^q} dy=C (q-1)M_z^{-(q-1)},
\end{align*}
where in the second inequality we use (\ref{eq_positiveexpecation.}), in the third inequality we use Markov inequality and in the last equality we use the fact $q>2$. By Cauchy-Schwarz inequality, we can show analogously that for some constant $C>0$ 
\begin{equation*}
\left( \mathbb{E}(\overline{X}_i^M)^2-\mathbb{E} (X_i^M)^2 \right)=\mathbb{E} \left( (X_i^M)^2 \mathbf{1}(|X_i^M|>M_z) \right) \leq C \xi_c^{q} M_z^{-(q-2)},
\end{equation*}
where we use the fact that 
\begin{equation*}
\mathbf{1}\left(|X_i^M|>M_z \right) \leq \frac{|X_i^M|^{q-2}}{M_z^{q-2}}.
\end{equation*}
This implies that for some constant $C>0$ 
\begin{equation*}
\operatorname{Var}(\overline{X}_i^M)-\operatorname{Var}(X_i^M) \leq C \xi_c^q M_z^{-(q-2)}.
\end{equation*}
Similarly, we can show that 
\begin{equation*}
\operatorname{Cov}(\overline{X}_i^M,\overline{X}_j^M)-\operatorname{Cov}(X_i^M, X_j^M) \leq C \xi_c^q M_z^{-(q-1)}.
\end{equation*}
Together with Lemma \ref{lem_disc}, we find that 
\begin{equation*}
\| \operatorname{Cov}(\mathbf{X}^M)- \operatorname{Cov}(\overline{\mathbf{X}}^M)  \|\leq C \xi_c^q pM_z^{-(q-2)}.
\end{equation*}
Second, we control the difference between $\mathbf{X}^M$ and $\mathbf{X}.$ By \cite[Lemma A.1]{LL} (or Lemma \ref{lem_con}), we have 
\begin{equation}\label{eq_bbbd}
\mathbb{E} \left(|X_j-X_j^M|^q \right)^{2/q} \leq C \Theta^2_{M,j,q}.
\end{equation}
By Assumption \ref{phy_generalts}, we conclude that 
\begin{equation}\label{eq_bbbd1}
\Theta_{M,j,q} \leq C \xi_c M^{-\tau+1}.
\end{equation}
Consequently, by Jenson's inequality, we have that 
\begin{equation}\label{eq_mommentboundfinal}
\mathbb{E}|X_j-X_j^M| \leq C \xi_c M^{-\tau+1}, \ \mathbb{E}|X_j-X_j^M|^2 \leq C \xi_c^2 M^{-2\tau+2}. 
\end{equation}
Therefore, we have that for some constant $C>0,$
\begin{equation*}
\operatorname{Var}(X_i^M)-\operatorname{Var}(X_i^M) \leq C \xi_c M^{-\tau+1},
\end{equation*}
where we use a discussion similar to (\ref{eq_variancedecomposition}). Similarly, we can show that 
\begin{equation*}
\left|\operatorname{Cov}(X_i^M,X_j^M)-\operatorname{Cov}(X_i, X_j) \right| \leq C \xi_c M^{-\tau+1}.
\end{equation*}
Together with Lemma \ref{lem_disc},  we find that 
\begin{equation*}
\| \operatorname{Cov}(\mathbf{X})- \operatorname{Cov}(\mathbf{X}^M)  \|\leq C p\xi_c M^{-\tau+1}.
\end{equation*}
As a result, we conclude that 
\begin{equation}\label{eq_covariance}
\| \operatorname{Cov}(\mathbf{X})- \operatorname{Cov}(\overline{\mathbf{X}}^M)  \|\leq C(p\xi_c M^{-\tau+1}+p \xi_c^q M_z^{-(q-2)}).
\end{equation}
We decompose that
\begin{equation}\label{eq_ffffff}
\mathbb{P}(\mathbf{Y} \Gamma \mathbf{Y}^* \leq x)-\mathbf{P}(\widetilde{\mathbf{Y}} \Gamma \widetilde{\mathbf{Y}}^* \leq x)=\mathbb{P}(\mathbf{Y} \Gamma \mathbf{Y}^* \leq x)-\mathbb{P}(\mathbf{Y} \Gamma \mathbf{Y}^* \leq x+\mathcal{D}(\mathbf{Y}, \widetilde{\mathbf{Y}})),
\end{equation}
where $\mathcal{D}(\mathbf{Y}, \widetilde{\mathbf{Y}})$ is defined as 
\begin{equation*}
\mathcal{D}(\mathbf{Y}, \widetilde{\mathbf{Y}}):=-\widetilde{\mathbf{Y}} \Gamma \widetilde{\mathbf{Y}}^*+\mathbf{Y} \Gamma \mathbf{Y}^* .
\end{equation*}
By (\ref{eq_covariance}) and a decomposition similar to (\ref{eq_decompositionone}), we have 
\begin{equation*}
\|  \mathcal{D}(\mathbf{Y}, \widetilde{\mathbf{Y}}) \| \leq C \sqrt{p} \xi_c (p\xi_c M^{-\tau+1}+p\xi_c^q M_z^{-(q-2)}).
\end{equation*}
By Lemma \ref{lem_chisquare} and (\ref{eq_ffffff}), we find that 
\begin{equation*}
\mathcal{K}(\mathbf{Y}, \widetilde{\mathbf{Y}}^M) \leq C \left( \sqrt{p} \xi_c (p\xi_c M^{-\tau+1}+p\xi_c^qM_z^{-(q-2)}) \right)^{1/2}.
\end{equation*}
Therefore, using the definition of $\mathcal{K}(\cdot,\cdot)$ in (\ref{eq_defncalk}), we conclude that 
\begin{equation*}
\mathcal{K}(\mathbf{X}^M, \mathbf{Y}) \leq C \left(\gamma+p^{\frac{7}{4}} n^{-1/2}M_z^3 M^2+\left(  \sqrt{p}\xi_c (p\xi_c M^{-\tau+1}+p\xi_c^qM_z^{-(q-2)}) \right)^{1/2} \right).
\end{equation*}
By Markov inequality, we can choose $\gamma=O\Big(\frac{\xi_c}{M_z}\Big).$ Finally, we control $\mathcal{K}(\mathbf{X}, \mathbf{X}^M)$ to finish our proof. We first introduce some notations. Denote the physical dependence measure for $z_{kl}$ as $\delta_{kl}^z(s,q)$ and 
\begin{equation*}
\theta_{k,j,q}=\sup_{k} \delta_{kl}^z(s,q), \ \Theta_{s,l,q}=\sum_{o=s}^{\infty} \theta_{o,l,q}.
\end{equation*}
By Assumption \ref{phy_generalts}, we conclude that 
\begin{equation}\label{eq_boundboundbound}
\sup_{1 \leq l \leq p} \Theta_{s,l,q}<\xi_c, \ \sum_{s=1}^{\infty} \sup_{1 \leq l \leq p}s \theta_{s,l,3}<\xi_c.  
\end{equation}
Denote the set
\begin{equation*}
\mathcal{I}(\Delta_M):=\Big \{ \max_{1 \leq j \leq p} \Big| X_j-X_j^{(M)} \Big| \leq \Delta_M \Big\}.
\end{equation*}
We claim that for arbitrary small $\delta>0,$ we can decompose the probability by { 
\begin{align}\label{eq_decompositionkey}
\mathcal{K}(\mathbf{X}^M, \mathbf{X})&=
\mathcal{K}(\mathbf{X}^M, \mathbf{X} \cap \mathcal{I}(\Delta_M)) +\mathcal{K}(\mathbf{X}^M, \mathbf{X}\cap \mathcal{I}^c(\Delta_M))  \\
& \leq C \Big( \sqrt{p \Delta_M \xi_c n^{\delta}} +n^{-\delta}+\mathbb{P}(\mathcal{I}^c(\Delta_M)) \Big), \nonumber
\end{align}
where we use the definition of $\mathcal{K}(\cdot, \cdot)$ to control the second term of the right-hand side of (\ref{eq_decompositionkey}). For the first term, note that
\begin{equation*}
\mathbb{P}\left( (\mathbf{X}^M)^* \Gamma \mathbf{X}^M \leq x \right)-\mathbb{P} \left( \mathbf{X} \Gamma \mathbf{X} \leq x \right)=\mathbb{P}\left( (\mathbf{X}^M)^* \Gamma \mathbf{X}^M \leq x \right)-\mathbb{P}\left( (\mathbf{X}^M)^* \Gamma \mathbf{X}^M \leq x+\mathcal{D}(\mathbf{X}^M, \mathbf{X}) \right),
\end{equation*}
where $\mathcal{D}(\mathbf{X}^M, \mathbf{M})$ is defined as
\begin{equation*}
\mathcal{D}(\mathbf{X}^M,\mathbf{X})=-\mathbf{X}^* \Gamma \mathbf{X}+(\mathbf{X}^M)^* \Gamma \mathbf{X}^M .
\end{equation*}
Further, we have
\begin{align}\label{eq_decompositionone}
\| \mathcal{D}(\mathbf{X}^M, \mathbf{M}) \| \leq \|(\mathbf{X}^M)^* \Gamma (\mathbf{X}^M-\mathbf{X}) \|+\| (\mathbf{X}^M-\mathbf{X})^* \Gamma \mathbf{X}\|. 
\end{align}
Recall (\ref{eq_bbbd}) and (\ref{eq_bbbd1}). Restricted on $\mathcal{I}(\Delta_M),$ by Cauchy-Schwarz inequality, the fact $\Gamma$ is bounded, Lemma \ref{lem_con} with (\ref{eq_boundboundbound}), we find that for some constant $C>0,$ 
\begin{equation*}
\| \mathcal{D}(\mathbf{X}^M, \mathbf{X}) \| \leq  C \sqrt{p} \xi_c (\sqrt{p} \Delta_M)=C p \Delta_M \xi_c.
\end{equation*}
Therefore, conditional on $\mathcal{I}(\Delta_M),$ for some constant $C>0,$ we have 
\begin{align*}
\left|\mathbb{P}\left( (\mathbf{X}^M)^* \Gamma \mathbf{X}^M \leq x \right)-\mathbb{P} \left( \mathbf{X} \Gamma \mathbf{X} \leq x \right)  \right|&  \leq C n^{-\delta} \\
& +\left|\mathbb{P}\left( (\mathbf{X}^M)^* \Gamma \mathbf{X}^M \leq x  \right)-\mathbb{P}\left( (\mathbf{X}^M)^* \Gamma \mathbf{X}^M \leq x+n^{\delta} p \Delta_M \xi_c  \right) \right|. 
\end{align*}
Moreover, we have 
\begin{align*}
\left|\mathbb{P}\left( (\mathbf{X}^M)^* \Gamma \mathbf{X}^M \leq x  \right) -\mathbb{P}\left( (\mathbf{X}^M)^* \Gamma \mathbf{X}^M \leq x+n^{\delta} p \Delta_M \xi_c  \right) \right| \leq & 2 \mathcal{K}(\mathbf{X}^M, \mathbf{Y})+ \mathbb{P}\left( \mathbf{Y}^* \Gamma \mathbf{Y} \leq x+n^{\delta} p \Delta_M \xi_c  \right) \\
& -\mathbb{P}\left( (\mathbf{Y}^* \Gamma \mathbf{Y} \leq x  \right).
\end{align*}
Since $\Gamma$ is positive definite and bounded, by Lemma \ref{lem_chisquare} and the rotation invariance property of Gaussian random vectors, we obtain the bound for the first term of the right-hand side of (\ref{eq_decompositionkey}).}  Next, by Markov inequality and a simple union bound, we have that 
\begin{equation*}
\mathbb{P}(\mathcal{I}^c(\Delta_M)) \leq C \sum_{j=1}^p \frac{\Theta_{M,j,q}^q}{\Delta_M^q}.
\end{equation*}
Consequently, we can control 
\begin{equation*}
\mathcal{K}(\mathbf{X}^M, \mathbf{X}) \leq C \Big( \sqrt{p \Delta_M \xi_c n^{\delta}} +n^{-\delta}+p \xi_c  M^{-q\tau+1}/\Delta^q_M \Big).
\end{equation*}
By optimizing $\Delta_M,$ we conclude that
\begin{equation*}
\mathcal{K}(\mathbf{X}^M, \mathbf{X}) \leq C \Big( M^{\frac{-q\tau+1}{2q+1}}\xi_c^{(q+1)/(2q+1)} p^{\frac{q+1}{2q+1}} n^{\frac{\delta q}{2q+1}} +n^{-\delta}\Big),
\end{equation*}
 This finishes our proof using triangle inequality.
\end{proof}

\begin{proof}[Proof of Proposition \ref{prop_normal}] Denote $r=\text{Rank}( \Omega^{1/2} \Gamma \Omega^{1/2})$ and the eigenvalues of $\Omega^{1/2} \Gamma \Omega^{1/2}$ as $d_1 \geq d_2>\cdots\geq d_r. $ Under (1) of Assumption \ref{assu_basis}, the definition of $\mathbf{W}$ and the fact that 
\begin{equation*}
\lambda_{\min}(A) \lambda_{\min}(B) \leq \lambda_{\min}(AB) \leq \lambda_{\max}(AB) \leq \lambda_{\max}(A) \lambda_{\max}(B),
\end{equation*}
for any given positive semi-definite matrices $A$ and $B$, we conclude that $d_i=O(1),\ i=1,2,\cdots,r.$ For the basis functions we used, we have that $r=O(bc).$ Therefore, we have 
\begin{equation*}
\frac{d_1}{f_2} \rightarrow 0. 
\end{equation*}
Hence, by Theorem \ref{thm_gaussian} and Lindeberg's central limit theorem, we finish our proof. 

\end{proof}

\begin{proof}[Proof of Proposition \ref{prop_power}] 
Denote the statistic $\mathcal{T}$ as
\begin{equation*}
\mathcal{T}:=\sum_{j=1}^b \int_0^1 \Big( \widehat{\phi}_j(t)-\phi_j(t)- \Big( \int_0^1 \widehat{\phi}_j(s)-\phi_j(s) ds \Big) \Big)^2 dt.
\end{equation*}
One one hand, by Proposition \ref{prop_normal}, we have that 
\begin{equation*}
\frac{n \mathcal{T}-f_1}{f_2} \Rightarrow \mathcal{N}(0,2).
\end{equation*}
On the other hand, by an elementary computation, we have 
\begin{equation*}
n\mathcal{T}=nT+n \sum_{j=1}^{b} \int_0^1 \Big(\phi_j(t)-\bar{\phi}_j \Big)^2dt-2n \sum_{j=1}^b \int_0^1 \Big( \phi_j(t)-\bar{\phi}_j \Big)\Big( \widehat{\phi}_j(t)-\bar{\widehat{\phi}}_j \Big) dt. 
\end{equation*}
Furthermore, we can rewrite the above equation as 
\begin{equation*}
n\mathcal{T}=nT-n\sum_{j=1}^b \int_0^1 \left( \phi_j(t)-\bar{\phi}_j \right)^2 dt+2n \sum_{j=1}^b \int_0^1 \left( \phi_j(t)-\bar{\phi}_j \right) \left( \phi_j(t)-\widehat{\phi}_j(t)-(\bar{\phi}_j-\bar{\widehat{\phi}}_j) \right) dt.
\end{equation*}
By (\ref{eq_phiform}), we find that
\begin{equation*}
\int_0^1 \left( \phi_j(t)-\bar{\phi}_j \right) \left( \phi_j(t)-\widehat{\phi}_j(t)-(\bar{\phi}_j-\bar{\widehat{\phi}}_j) \right)dt=\bm{\beta}_j^*\widehat{B}(\bm{\beta}_j-\widehat{\bm{\beta}}_j)+O(bc^{-d}),
\end{equation*}
where $\widehat{B}$ is defined as 
\begin{equation*}
\widehat{B}=\int_0^1 (\mathbf{B}(t)-\bar{B})(\mathbf{B}(t)-\bar{B})^*dt. 
\end{equation*}
It is easy to see that $\| \widehat{B} \|=O(1).$ Therefore, under the alternative hypothesis $\mathbf{H}_a,$ we find that 
\begin{equation*}
\int_0^1 \left( \phi_j(t)-\bar{\phi}_j \right) \left( \phi_j(t)-\widehat{\phi}_j(t)-(\bar{\phi}_j-\bar{\widehat{\phi}}_j) \right)dt=O_{\mathbb{P}} \left( \sqrt{\log n}\frac{(bc)^{1/4}}{n} \right),
\end{equation*}
where we use Theorem \ref{thm_finalresult} and Assumption \ref{assu_basis}. This concludes our proof. 
\end{proof}

\begin{proof}[Proof of Theorem \ref{thm_bootstrapping}] 
We divide our proofs into two steps. In the first step, we show that the result holds for $\mathcal{T}$ defined in (\ref{eq_defnmathcalt}). In the second step, we control the closeness between $\mathcal{T}$ and $\widehat{\mathcal{T}}.$

 We start with the first step  following the proof strategy of \cite[Theorem 3]{ZZ1}. Denote
\begin{equation*}
\Lambda=\frac{1}{(n-m-b)} \sum_{i=b+1}^{n-m} \Upsilon_{i,m} \Upsilon_{i,m}^*,
\end{equation*}
where we use
\begin{equation*}
\Upsilon_{i,m}=\frac{1}{\sqrt{m}}H_i \otimes \mathbf{B}(\frac{i}{n}), \ H_i=\Big(\sum_{j=i}^{i+m} \bm{h}_j \Big) .
\end{equation*}
\begin{lem}\label{lem_a2} Under the assumptions of Theorem \ref{thm_bootstrapping}, we have 
\begin{equation*}
\sup_{b+1 \leq i \leq n-m} \left| \left| \Upsilon_{i,m} \Upsilon_{i,m}^*-\mathbb{E} \Big( \Upsilon_{i,m} \Upsilon_{i,m}^* \Big) \right| \right|=O \Big( b \zeta^2_c \sqrt{m} \Big).
\end{equation*}
\end{lem}
\begin{proof}
Using the basic property of Kronecker product, we find 
\begin{equation*}
\Upsilon_{i,m} \Upsilon_{i,m}^*=\frac{1}{m} \left[ H_i H_i^*  \right] \otimes \left[  \mathbf{B}(\frac{i}{n})  \mathbf{B}^*(\frac{i}{n})\right].
\end{equation*}
As a consequence, we have that 
\begin{equation}\label{eq_upsionone}
\sup_{b+1 \leq i \leq n-m} \left| \left| \Upsilon_{i,m} \Upsilon_{i,m}^*-\mathbb{E} \Big( \Upsilon_{i,m} \Upsilon_{i,m}^* \Big) \right| \right| \leq \sup_{b+1 \leq i \leq n-m} \left| \left| H_i H_i^*-\mathbb{E} \Big( H_i H_i^* \Big) \right| \right| \frac{\zeta^2_c}{m},
\end{equation}
where we use the property of the spectrum of Kronecker product and the fact  $  \mathbf{B}(\frac{i}{n})  \mathbf{B}^*(\frac{i}{n})$ is a rank-one matrix. Now we focus on studying the first entry of $H_i H_i^*,$ which is of the form $w=\Big(\sum_{j=i}^{i+m} x_{j-1}\epsilon_j \Big)^2.$ We first study its physical dependence measure.  Note that $w$ is $\mathcal{F}_{i+m}$ measurable and can be written as $f_i(\mathcal{F}_{i+m}).$ Denote $w(l)=f_i(\mathcal{F}_{i+m,l}).$ By Assumption \ref{phy_generalts} and Lemma \ref{lem_con}, we conclude that 
\begin{equation}\label{eq_priorbound}
\sup_i \left | \left | \sum_{j=i}^{i+m} x_{j-1} \epsilon_j \right| \right|_q=O(\sqrt{m}).
\end{equation} 
Recall that by Jensen's inequality, if $X \in \mathcal{L}^q, q>4$, we have 
\begin{equation}\label{eq_jensonbound}
\mathbb{E}|X|^2 \leq (\mathbb{E}|X|^q)^{2/q}.
\end{equation} 
Therefore, by (\ref{eq_priorbound}), (\ref{eq_jensonbound}) and Minkowski's inequality, we have 
\begin{equation*}
||w-w(l)||=O(\sqrt{m}) \Big( \sum_{j=l-m}^l \delta(j,q) \Big).
\end{equation*}
By Lemma \ref{lem_con} and Assumption \ref{phy_generalts}, we have 
\begin{equation*}
||w-\mathbb{E}w||=O(m^{3/2}).
\end{equation*}
Therefore, by (\ref{eq_upsionone}) and Lemma \ref{lem_disc}, we conclude our proof. 

\end{proof}
Using a discussion similar to the lemma above and Assumption \ref{assu_parameter}, it is easy to conclude that 
\begin{equation}\label{eq_lambdaconvergence}
|| \Lambda-\mathbb{E}(\Lambda) ||=O \Big( b\zeta_c^2 \sqrt{m/n} \Big).
\end{equation}
Next, we show that a stationary time series can be used  to approximate $\mathbb{E} \Big(H_jH_j^* \Big),$ where the stationary time series can closely preserve the long-run covariance matrix (\ref{eq_longrunh}). Recall (\ref{eq_defnh}).  Denote the stationary time series as
\begin{equation*}
\widetilde{\bm{h}}_{i,j}=\mathbf{U}(\frac{i}{n}, \mathcal{F}_{j}), \ i \leq j \leq i+m.
\end{equation*}
Correspondingly, we can define 
\begin{equation*}
\widetilde{\Upsilon}_{i,m}=\frac{1}{\sqrt{m}}\widetilde{H}_i \otimes \mathbf{B}(\frac{i}{n}), \ \widetilde{H}_i=\sum_{j=i}^{i+m} \widetilde{\bm{h}}_{i,j}.
\end{equation*} 
\begin{lem}\label{lem_a3} Under the assumptions of Theorem \ref{thm_bootstrapping}, we have 
\begin{equation*}
\sup_{b+1 \leq i \leq n-m} \left| \left| \mathbb{E} \Big( \Upsilon_{i,m} \Upsilon_{i,m}^*\Big)-\mathbb{E} \Big(\widetilde{\Upsilon}_{i,m} \widetilde{\Upsilon}_{i,m}^* \Big) \right| \right|=O \Big( \Big( \frac{mb^2}{n} \Big)^{1-2/\tau}  b\zeta_c^2 \Big).
\end{equation*}
\end{lem}
\begin{proof}
Similar to (\ref{eq_upsionone}), we have 
\begin{equation*}
\sup_{b+1 \leq i \leq n-m} \left| \left| \mathbb{E} \Big( \Upsilon_{i,m} \Upsilon_{i,m}^*\Big)-\mathbb{E} \Big(\widetilde{\Upsilon}_{i,m} \widetilde{\Upsilon}_{i,m}^* \Big) \right| \right|\leq \sup_{b+1 \leq i \leq n-m} \left| \left| \mathbb{E}(\widetilde{H}_i \widetilde{H}_i^*)-\mathbb{E} ( H_i H_i^* ) \right| \right| \frac{\zeta^2_c}{m}.
\end{equation*}
We also focus on studying the first entry of $ \widetilde{H}_i \widetilde{H}_i^*- H_i H_i^*,$ which is of the form $\Big(\sum_{j=i}^{i+m} \widetilde{x}_{j-1}\widetilde{\epsilon}_j \Big)^2-\Big(\sum_{j=i}^{i+m} x_{j-1}\epsilon_j \Big)^2.$ We first observe that
\begin{align*}
\left| \left| \sum_{j=i}^{i+m} \Big( \widetilde{x}_{j-1} \widetilde{\epsilon}_j-x_{j-1} \epsilon_j \Big) \right| \right|=O\left( \sum_{j=i}^{i+m} x_{j-1}(\widetilde{\epsilon}_{j}-\epsilon_{j}) \right).
\end{align*}
Hence, by Lemma \ref{lem_con} and Assumption \ref{assum_local}, we have 
\begin{equation*}
\left| \left| \sum_{j=i}^{i+m} \Big( \widetilde{x}_{j-1} \widetilde{\epsilon}_j-x_{j-1} \epsilon_j \Big) \right| \right|=O\Big( \sqrt{m} \sum_{j=0}^{\infty}\min \{\frac{m}{n}, \delta(j,2)\} \Big)=O \Big( \sqrt{m} \Big( \frac{m}{n} \Big)^{1-2/\tau} \Big),
\end{equation*}
where we use the fact $\delta(j,2) \leq \delta(j,q).$ Hence, by (\ref{eq_jensonbound}) and Minkowski's inequality,  we have that 
\begin{equation*}
\sup_{i} \left| \left| \Big(\sum_{j=i}^{i+m} \widetilde{x}_{j-1}\widetilde{\epsilon}_j \Big)^2-\Big(\sum_{j=i}^{i+m} x_{j-1}\epsilon_j \Big)^2 \right| \right| =O\Big(m \Big( \frac{m}{n} \Big)^{1-2/\tau}  \Big).
\end{equation*}
This concludes our proof using Lemma \ref{lem_disc}. 
\end{proof}
Furthermore, by \cite[Lemma 4]{ZZ1} and a discussion similar to (\ref{eq_upsionone}), we have 
 \begin{equation*}
 \sup_{b+1 \leq i \leq n-m} \left| \left| \mathbb{E} \Big(\widetilde{\Upsilon}_{i,m} \widetilde{\Upsilon}_{i,m}^* \Big)-\Omega(\frac{i}{n}) \right| \right|=O\Big(\frac{b \zeta_c^2}{m} \Big). 
 \end{equation*}
Hence, by Assumption \ref{assu_smmothness} and  \cite[Theorem 1.1]{HT}, we have
\begin{equation*}
\left|\frac{1}{n-m-b} \sum_{i=b+1}^{n-m} \mathbb{E} \Big(\widetilde{\Upsilon}_{i,m} \widetilde{\Upsilon}_{i,m}^* \Big)-\int_0^1 \Omega(t)dt \right|=O\Big(\frac{b\zeta_c^2}{m}+\frac{1}{(n-m-b)^2}\Big).
\end{equation*}
Therefore, under Assumption \ref{assu_parameter}, by Lemmas \ref{lem_con}, \ref{lem_a2} and \ref{lem_a3}, we have that 
\begin{equation}\label{eq_lamdaomega}
||\Lambda-\Omega||=O\Big( \theta(m) \Big), \ \theta(m)=b \zeta_c^2 \left( \sqrt{\frac{m}{n}}+\frac{1}{\sqrt{n}}\Big(\frac{mb^2}{n} \Big)^{1-2/\tau}+\frac{1}{m} \right).
\end{equation}
It is easy to check that as $\tau>10,$
\begin{equation*}
\frac{1}{\sqrt{n}}\Big(\frac{m}{n} \Big)^{1-2/\tau} \leq \frac{1}{m},
\end{equation*}
where we use the assumption that $m \ll n.$ By definition, conditional on the data, $\Phi$ is normally distributed. Hence, we may write
\begin{equation*}
\Phi \equiv \Lambda^{1/2} \mathbf{G},
\end{equation*}
where $\mathbf{G} \sim \mathcal{N} (0, I_p)$ and $\equiv$ means that they have the same distribution. Define $r=\text{Rank}(\mathbf{W})$ and the eigenvalues of $\Lambda^{1/2} \widehat{\Gamma} \Lambda^{1/2}$ as $\lambda_1 \geq \lambda_2 \geq \cdots \geq \lambda_r>0.$ By (\ref{eq_lambdaconvergence}) and Assumption \ref{assu_basis}, it is easy to see that  $\lambda_i=O(1)$ when conditional on the data. Therefore, by Lindeberg's central limit theorem, we have 
\begin{equation*}
\frac{\mathbf{G}^* \Lambda^{1/2} \widehat{\Gamma} \Lambda^{1/2} \mathbf{G}-\sum_{i=1}^r \lambda_i}{(\sum_{i=1}^r \lambda_i^2)^{1/2}} \Rightarrow \mathcal{N}(0,2). 
\end{equation*}
Recall that $d_1 \geq d_2 \geq \cdots \geq d_r>0$ are the eigenvalues of $\Omega^{1/2} \Gamma \Omega^{1/2}$ and $d_i=O(1).$  Recall that $r=O(bc)$ and denote the set $\mathcal{A} \equiv \mathcal{A}_n$ as 
\begin{align*}
\mathcal{A} \equiv \mathcal{A}_n:=\Big\{|\sum_{i=1}^r (\lambda_i-d_i)| \leq b_n \sqrt{bc}, \ |\sum_{i=1}^r (\lambda_i^2-d_i^2)| \leq c_n \sqrt{bc} \Big\},
\end{align*}
where $b_n, c_n=o(1).$ On the event $\mathcal{A},$ we have that
\begin{align} \label{eq_gapproximation}
\frac{\mathbf{G}^*  \Lambda^{1/2} \widehat{\Gamma} \Lambda^{1/2}-f_1}{f_2}&=\frac{\mathbf{G}^*  \Lambda^{1/2} \widehat{\Gamma} \Lambda^{1/2}-\sum_{i=1}^r \lambda_i+\sum_{i=1}^r \lambda_i-f_1}{(\sum_{i=1}^r \lambda_i^2)^{1/2}} \left(\frac{(\sum_{i=1}^r \lambda_i^2)^{1/2}}{f_2} \right) \nonumber \\
& = \frac{\mathbf{G}^* \Lambda^{1/2} \widehat{\Gamma} \Lambda^{1/2} \mathbf{G}-\sum_{i=1}^r \lambda_i}{(\sum_{i=1}^r \lambda_i^2)^{1/2}}+o_{\mathbb{P}}(1). 
\end{align}
Therefore, we have shown that Theorem \ref{thm_bootstrapping} holds true on the event $\mathcal{A}.$ Under Assumption \ref{assu_parameter}, and a discussion similar to (\ref{eq_lamdaomega}) \footnote{The operator norm and the difference of trace share the same order as we apply Lemma \ref{lem_disc}.}  and  (2) of Lemma \ref{lem_coll}
\begin{equation*}
|| \widehat{\Sigma}-\Sigma ||=O_{\mathbb{P}} \Big( \frac{\zeta_c \log n}{\sqrt{n}} \Big),
\end{equation*}   
we find that $$\mathbb{P}(\mathcal{A})=1-o(1).$$ Hence, we can conclude our proof for $\mathcal{T}$ using Theorem \ref{thm_gaussian}. 

For the second step, by Theorems \ref{lem_phibound} and \ref{thm_finalresult}, we conclude that
\begin{equation*}
\sup_{i>b} |\epsilon_i-\widehat{\epsilon}_i|=O_{\mathbb{P}}(\vartheta(n)), \ \vartheta_n= n^{2/\tau} \Big(\zeta_c \sqrt{\frac{\log n}{n}}+n^{-d\alpha_1} \Big).
\end{equation*}
Denote $\widehat{\Upsilon}_{i,m}$ by replacing $\bm{h}_{i}$ with $\widehat{\bm{h}}_i.$ Similar to the discussion of Lemma \ref{lem_a3}, we conclude that
\begin{equation*}
\sup_{b+1 \leq i \leq n-m} \left| \left|  \Upsilon_{i,m} \Upsilon_{i,m}^*-\widehat{\Upsilon}_{i,m} \widehat{\Upsilon}_{i,m}^*  \right| \right|=O_{\mathbb{P}}(b\zeta_c^2 \vartheta_n).
\end{equation*}
Hence, we have 
\begin{equation*}
||\Lambda-\widehat{\Lambda} ||=O_{\mathbb{P}}\Big(\frac{1}{\sqrt{n}} b \zeta_c^2 \vartheta_n \Big).
\end{equation*}
Using a similar discussion to (\ref{eq_gapproximation}), we can conclude our proof.   

%
%

\end{proof}

%

\begin{proof} [Proof of Theorem \ref{thm_choicem}] The proof follows from (\ref{eq_lamdaomega}) and the assumptions of Theorem \ref{thm_bootstrapping}.
\end{proof}

\begin{proof}[Proof of Lemma \ref{lem_controltbtrho}] 
For $j>b,$ recall that we denote $\Gamma_{i,j}=\text{Cov}(\bm{x}_i^j, \bm{x}_i^j), \bm{\gamma}_{i,j}=\text{Cov}(\bm{x}_i^j,x_i),$ where $\bm{x}_i^j \in \mathbb{R}^{j}.$ We further denote the $j \times j$ matrix $\Gamma_{i,b}^j$ as the following block matrix and $\bm{\gamma}_{i,b}^j$ as the block vector
\begin{equation*}
\Gamma_{i,b}^j=
\begin{bmatrix}
\Gamma_{i,b} &\bm{E}_{i,1} \\
\bm{E}_{i,3} & \bm{E}_{i,2}
\end{bmatrix},
\bm{\gamma}_{i,b}^j=(\bm{\gamma}_{i,b}, \bm{0}),
\end{equation*}
where $\bm{E}_{i,1},$ $\bm{E}_{i,2}$ and $\bm{E}_{i,3}$ are defined analogously to (\ref{eq_defne1e2}) and (\ref{eq_e3}) respectively.
Then we can complete our proof using a discussion similar to (\ref{eq_phibound2}) using Lemma \ref{lem_nuem}. By letting $\Delta A:=\Gamma_{i,b}^j-\Gamma_{i,j}, \ \Delta v= \bm{\gamma}_{i,b}^j-\bm{\gamma}_{i,j}-\Delta \bm{\gamma}_{i,j}$ and the UPDC in Assumption \ref{assu_pdc}, together with  Lemma \ref{lem_coll}, we conclude that for some constant $C_1>0,$ 
\begin{equation*}
|| \bm{\phi}_i^j-\overline{\bm{\phi}}_i^b   || \leq C_1j^{-(\tau-1)/(1+\epsilon)+2} \leq C_1 n^{-1+(3+2\epsilon)/\tau},
\end{equation*}
where $\overline{\bm{\phi}}_i^b=(\bm{\phi}_i^b,\bm{0})^*$ with $\Gamma_{i,b} \bm{\phi}_i^b=\gamma_{i,b} $ and $\Gamma_{i,j} \bm{\phi}_i^j=\bm{\gamma}_{i,j}.$
Hence, our proofs follow from the definitions of $\mathcal{T}_{\rho}$ and $\mathcal{T}_{\phi},$ the triangle inequality, Theorems \ref{lem_phibound}, \ref{thm_finalresult} and \ref{thm_consistency}.
\end{proof}

{
\begin{proof}[Proof of Lemma \ref{lem_distributionpacf}] By a discussion similar to  Proposition \ref{prop_normal}, we obtain that 
\begin{equation*}
\frac{n \mathcal{T}_{\phi}-g_1}{g_2} \rightarrow \mathcal{N}(0,2). 
\end{equation*}
Then the proof follows from Lemma \ref{lem_controltbtrho}.
\end{proof}

}


%

%

\end{appendix}
%

%

\end{document}